\documentclass[11pt]{amsart}

\usepackage{amsfonts, amssymb}

\usepackage{amssymb,amsmath,amscd,euscript,verbatim,array}
\usepackage{geometry}
\usepackage{anysize}
\usepackage{fancyhdr}
\usepackage{indentfirst}
\usepackage{graphicx}
\usepackage{color}
\usepackage{ifpdf}
\usepackage{marginnote}
\usepackage{enumitem}
\usepackage{mathrsfs}
\IfFileExists{CJK.sty}{\usepackage{CJK}}{}
\usepackage{bm}

\setlist[itemize]{leftmargin=2em}

\newcommand{\N}{\mathbb{N}}
\newcommand{\Z}{\mathbb{Z}}
\newcommand{\R}{\mathbb{R}}
\newcommand{\Q}{\mathbb{Q}}
\newcommand{\C}{\mathbb{C}}

\newcommand{\T}{\mathbb{T}}

\newcommand{\dd}{\mathrm d}
\newcommand{\Inv}{\mathcal E}
\newtheorem{Theorem}{Theorem}

\newtheorem{Definition}{Definition}[section]
\newtheorem{Proposition}[Definition]{Proposition}
\newtheorem{Corollary}[Theorem]{Corollary}
\newtheorem{Lemma}[Definition]{Lemma}

\theoremstyle{definition}
\newtheorem{Remark}[Definition]{Remark}

\newcommand{\bthm}{\begin{Theorem}}
	\newcommand{\ethm}{\end{Theorem}}
\newcommand{\bpr}{\begin{Proposition}}
	\newcommand{\epr}{\end{Proposition}}
\newcommand{\blm}{\begin{Lemma}}
	\newcommand{\elm}{\end{Lemma}}
\newcommand{\bex}{\begin{Exercise}}
	\newcommand{\eex}{\end{Exercise}}
\newcommand{\be}{\begin{equation}}
	\newcommand{\ee}{\end{equation}}
\newcommand{\beal}{\begin{aligned}}
	\newcommand{\enal}{\end{aligned}}
\newcommand{\brm}{\begin{Remark}}
	\newcommand{\erm}{\end{Remark}}

\newcommand{\HermanHyperbolicFigure}{%
  \leavevmode\hbox{%
    \pdfliteral{q 283.464566929 0 0 193.824463042 0 0 cm
      BI /W 1910 /H 1306 /BPC 8 /CS /RGB
      /F [/ASCIIHexDecode /FlateDecode] ID
78DAECDD799854E59D2FF0FC31C9189DA85127373337A3D79B1822598CB8818834222804507005C5888AE2AEA8444684
80A2B423635410772523222E6911890A6A351182B8E2CAA2805B5C400C88C65D6E3FD4633F75ABBA4E9F3AE754752D9F
CFD37F24A6BABBFA8CEFF7FCCEF7BEF7AD8D1B0100000000000000000000000000000000000000000000000000000000
000000000000000000000000000000000000000000000000000000000000000000000000000000000000000000000000
0000000000000000000000000000000000000000000000000000000000E2FAD763C7B5F8E5CA0000000000104CC30C00
00000040341A660000000000A2D1300300000000108D861900000000806834CC000000000044A3610600000000201A0D
330000000000D1689801000000008846C30C0000000040341A660000000000A2D1300300000000108D86190000000080
6834CC000000000044A3610600000000201A0D330000000000D1689801000000008846C30C0000000040341A66000000
0000A2D1300300000000108D861900000000806834CC000000000044A3610600000000201A0D330000000000D1689801
000000008846C30C0000000040341A6600C0E001000080073D00C0E001000080073D00C0E001000080073D0000830700
0080073D000083070000001EF4000083070000001EF4000083070000001EF4000083070000001EF40000000000A8051A
660000000000A2D1300300000000108D861900000000806834CC000000000044A36106041D000080070A00390920E800
00000F1400721240D00100001E2800E42480A003AC4700C00003202701041D603D02000618030C809C04049DA003EB11
00C00003202701041D588F0000061800390920E8C07A040030C000C849000083070060800190930000060F00C0006380
0190930080C10300C000032027010000004A437302202701000000A2D19C00C8490000C0D30480AC039093000080A709
00590720270100004F1300B20E404E02083A007093759305641D809C949380A00300DC6401641D809C04107400E0260B
20EB00E42480A003EB11B0A801641D809C041074603D5A8F605103C83A00390920E800EB110030C0002027014107588F
0000061800390920E8C07A040030C000C84900C0E0010060800190930000060F00C00003202701000C1E008001C60003
202701000000C2D09C00C849000000806834270072120000F0340120EB00E4240000E0690240D601C8490000C0D30480
AC03909300820E00DC64DD640159072027E52420E800003759005907202701041D00B8C902C83A00390920E8C07A042C
6A005907202701041D588FD62358D400B20E404E02083AC07A04000C3000C84940D001D62300800106404E02083AB01E
01000C30007212003078000018608042A552A9825EDC64CC370AFA5EE4240060F0000030C040D548A552DFDAA4D5A2B8
E905757575DFCA63CC98312EA69C04000C1E00000618A81D998D71D37F0EF9CA007A663909000000109FE604CA5CF3D6
E556CBE1DC57D66D923E2243C92C270100000012A739817296DB0C87AF97B35E99FB8230076E20270100004F1300B20E
2A516EBD9CAF100EB9CFB9C54DCEAEB39C0400003C4D00C83AA832B9C72907EC370E7FFC45413F1639090000789A0090
7550710AEA8173B73A27F593919380A00300DC6401641D5496DC1238604F72F8839AF3BDDE4119721240D001809B2C80
AC83EA10BF310EFEF98E62969300820E00DC6401641D54A5081B8C0BAAA337B6D430B75A4A2327014107588F80450D20
EBA0CC45E87E0BDDC0DCE2B76898E52480A003EB11B0A801641D54BADCE2B7D50DC911CEBB48AA614EA552759B64BD81
A69F5F7D1F1D2827010321603D02001860A09C45287E0B3DB139F22FCA92EE96BFD59A306F464E02083AC07A04000C30
407C111ADADCAE38CCE6E108473DB7FAEDB99B99ABAC679693808110B01E01000C3050B6A2ED2B8EB61539B70A2EA861
6EF51B5BFC5B2ABD6796930080C10300C00003652BDA7917D18AE238F56FD6FB0CFFE7145462CB4900C0039AC1030030
C000E12572444698A238954AC5699833F73F077F57D6DBB3871900C0E00100186080628876444652873087FCC6B4CC86
B9D56DC9E1EB68390900000050B93427D0B6A29D29917B9C72B4DF15F21B737F69AB6F3273BF74F8125B4E0200000054
16CD09B4AD680D73E90F61DE9873C846ABDF5BF70D39090000D4264F1380AC038A2D773772ABCD6DEE71CA618AE21637
30C76998AB6073B29C0400003C4D00C83AA868111AE6681FF317F3888C7C3FA4164A6639090000789A009075509E1269
98C3FCA2981B98F3BDE12A3804434E02083A00709305907550A1229CC31C612B72221B98D372CFCAA8FA92594E020642
00C04D1640D641798A50D8C63FB739F206E6B4689F4E282701041D00B8C9BAC902B20E48568BF56FC0EB231CC29C7821
DCE28EE82A2E99E524602004AC47C0A206907550B6728F620EF8F8BC423F6BAFC53638F2C7F3B5F8D312D9172D270104
1D603D82450D20EB80080ADA639C5B4717F49323F7C0A9542AEB578FD9A4164A663909180801EB1100C00003E52CFC36
E3F05D7482C72F6735C999BFB4A00DD8721240D001D623006080011217F22C8BDCD2385F639C54BDDCE2D6E5ACD754FD
81CC7212301002D62300800106CA5C98132D429E4A91D4E118B935758BFB9343BE4C4E0200183C0000030C503CAD36C3
B92FC8FA09B95B8ED39B8AA355BEE13BEAAC5F5A65DB98E5240060F0000030C0404568B1644E77B6F94ED248CB2D9663
7EF45EAB876364AAEE6DCC721200307800001860A082E4EB990B1279EBF2C6109BA57325D56CCB49000000800AA23981
B215AD678E532C378B501757F141197212000000201FCD0994BF319BB478C0725AD3FF9A3E2B23A95F57E806E68D1A66
0000004F1380AC937550C6721BE622FDA268E75D44EBA5E524000050C53C4D00B20E281FB91B988BF15B227F669F8619
0000C0D30420EB641D94A7DCE2B7481FA517B9288EF08D4D7FD4986F84794D5B7D7AA09C040C8400809B2C80AC838A96
DBDF96A6610EBF53BAA06FCC2DCC73374B37FDD7DC83418AF787CB49C04028E800C04D1640D641B52AD9191459A56EF8
3A37FC37E6FE2DB9A574BED7147470879C04107400E0260B20EB808D3987309761C31C72A771F3CFAFDBA4C53FAAB95E
4EBF26B76D2ED219D4721230100A3AB01E018B1A40D641552AD94911D11AE6905BAC9B7F78F38FCD3A2E23F347656E54
CE7A99861940D081F50858D400B20E08A96487306FCC69984376B959EF30DFDB4BBF2CF37FCDAA8E5BAC9773DF988619
40D081F5080060800142CA6D988B77107184DDC2059D119DFB717E218F59D63003083AB01E01000C304004F9CE2B2E86
DCCAB7D5363BCEFEEAF0DBB34BB3855B4E0206424107D62300800106AA4C893FE72EFC4119A9542AF22703A6853C5E63
A3861900F08006006080010A97BBA9B804FD6AEEAEE9DC5F9AC8E9D021F76667FDAEE21D12222701000F6800000618A8
26A5FC98BF4CB925737A3F735AEE3F8FD6FA86FCBBB22E829C04003CA00100186080300AFA1CBD62FFEA163BE7C88D77
F89DC96DF8317F72120000002080E604CA5C6EA95BE237306693CCADCBE9FFDCF40F639E56117E67721B1EC22C270100
000002684E80B612726772DB1EC22C270100004F1300B20E2843157108B39C0400003C4D00C83AA0DCA452A9900D73D6
B1CF72120000A8149E260059075024957208B39C040C8400809B2C80AC03CA4DE621CC010D73D656E7D21FC22C270103
2100E0260B20EB807213FF10E674F95C82CE594E02064200C04D1640D601E523FC21CCF91AE6CC9F202701041D588F80
450D20EB80DA91D51B07EC43CEEDA29BFE49E6B7DBC30C20E8C07A042C6A005907D494F01FF397D53067AAABAB2BCDB1
CC7212301002D62300800106281F594571F08BB3EAE890DF252701041D588F00000618A02A8DC910661F72D3CBEAEAEA
D2C572C9B62ECB49C040E8CA80F5080060800190930080C10300C0000320270100830700800106404E0200183C000003
0C809C04000000A8599A13003909000000108DE604404E0200009E2600641D809C0400003C4D00C83A0039090000789A
00907500721240D001809BAC9B2C20EB00E4A49C04041D00E0260B20EB00E42480A003003759005907202701041D588F
80450D20EB00E42480A003EBD17A048B1A40D601C849004107588F00800106003909083AC07A040030C000C849004107
D62300800106404E0200060F0000030C809C04003078000006180039090060F000000C3006180039090000001086E604
404E0200000044A339019093000080A70900590720270100004F1300B20E404E0200009E2600641D809C04107400E026
EB260BC83A0039292701410700B8C902C83A00390920E800C04D1640D601C849004107D623605103C83A00390920E8C0
7AB41EC1A206907500721240D001D62300608001404E02820EB01E01000C3000721240D081F5080060800190930080C1
0300C00003202701000C1E00800106404E0200183C0000038C0106404E0200000084A13901909300000000D1684E00E4
240000E0690240D601C8490000C0D30480AC039093000080A709005907202701041D00B8C9BAC902B20E404ECA4940D0
01006EB200B20E404E02083A00709305907500721240D081F50858D400B20E404E02083AB01EAD47B0A801641D809C04
107480F50800186000909380A003AC470000030C809C041074603D0200186000E4240060F0000030C000C84900008307
0060800190930000060F00C00063800190930000000061684E00E42400000040349A130039090000789A009075007212
0000F0340120EB00E4240000E0690240D601C849004107006EB26EB280AC0390937212107400809B2C80AC0390930082
0E00DC6401641D809C041074603D02163580AC03909300820EAC47EB112C6A005907202701041D603D0200061800E424
20E800EB1100C00003202701041D588F0000061800390900183C00000C3000721200C0E00100186000E4240080C10300
30C0186000E42400000040189A13003909000000108DE604404E0200009E2600641D809C0400003C4D00C83A00390900
00789A00907500721240D001809BAC9B2C20EB00E4A49C04041D00E0260B20EB00E42480A00300375900590720270104
1D588F80450D20EB00E42480A003EBD17A048B1A40D601C849004107588F00800106003909083AC07A040030C000C849
004107D62300800106404E0200060F0000030C809C04003078000006180039090060F000000C30061800390900000010
86E604404E0200000044A339019093000080A70900590720270100004F1300B20E404E0200009E2600641D809C041074
00E026EB260BC83A0039292701410700B8C902C83A00390920E800C04D1640D601C849004107D623605103C83A003909
20E8C07AB41EC1A206907500721240D001D62300608001404E02820EB01E01000C3000721240D081F508006080019093
0080C10300C00003202701000C1E00800106404E0200183C0000038C0106404E0200000084A13901909300000000D168
4E00E4240000E0690240D601C8490000C0D30480AC039093000080A709005907202701041D00B8C9BAC902B20E404ECA
4940D001006EB200B20E404E02083A00709305907500721240D081F50858D400B20E404E02083AB01EAD47B0A801641D
809C04107480F50800186000909380A003AC470000030C809C041074603D0200186000E4240060F0000030C000C84900
0083070060800190930000060F00C00063800190930000000061684E00E42400000040349A130039090000789A009075
0072120000F0340120EB00E4240000E0690240D601C849004107006EB26EB280AC0390937212107400809B2C80AC0390
9300820E00DC6401641D809C041074603D02163580AC03909300820EAC47EB112C6A005907202701041D603D02000618
00E42420E800EB1100C00003202701041D588F0000061800390900183C00000C3000721200C0E00100186000E4240080
C1030030C0186000E42400000040189A13003909000000108DE604404E0200009E2600641D809C0400003C4D00C83A00
39090000789A00907500721240D001809BAC9B2C20EB00E4A49C04041D00E0260B20EB00E42480A00300375900590720
2701041D588F80450D20EB00E42480A003EBD17A048B1A40D601C849004107588F00800106003909083AC07A040030C0
00C849004107D62300800106404E0200060F0000030C809C04003078000006180039090060F000000C30061800390900
00001086E604404E0200000044A339019093000080A70900590720270100004F1300B20E404E0200009E2600641D809C
04107400E026EB260BC83A0039292701410700B8C902C83A00390920E800C04D1640D601C849004107D623605103C83A
00390920E8C07AB41EC1A206907500721240D001D62300608001404E02820EB01E01000C3000721240D081F508006080
0190930080C10300C00003202701000C1E00800106404E0200183C0000038C0106404E0200000084A139019093000000
00D1684E00E4240000E0690240D601C8490000C0D30480AC039093000080A709005907202701041D00B8C9BAC902B20E
404ECA4940D001006EB200B20E404E02083A00709305907500721240D081F50858D400B20E404E02083AB01EAD47B0A8
01641D809C04107480F50800186000909380A003AC470000030C809C041074603D0200186000E4240060F0000030C000
C849000083070060800190930000060F00C00063800190930000000061684E00E42400000040349A130039090000789A
0090750072120000F0340120EB00E4240000E0690240D601C849004107006EB26EB280AC0390937212107400809B2C80
AC03909300820E00DC6401641D809C041074603D02163580AC03909300820EAC47EB112C6A005907202701041D603D02
00061800E42420E800EB1100C00003202701041D588F0000061800390900183C00000C3000721200C0E00100186000E4
240080C1030030C0186000E42400000040189A13003909000000108DE604404E0200009E2600641D809C0400003C4D00
C83A0039090000789A00907500721240D001809BAC9B2C20EB00E4A49C04041D00E0260B20EB00E42480A00300375900
5907202701041D588F80450D20EB00E42480A003EBD17A048B1A40D601C849004107588F00800106003909083AC07A04
0030C000C849004107D62300800106404E0200060F0000030C809C04003078000006180039090060F000000C30061800
39090000001086E604404E0200000044A339019093000080A70900590720270100004F1300B20E404E0200009E260064
1D809C04107400E026EB260BC83A0039292701410700B8C902C83A00390920E800C04D1640D601C849004107D6236051
03C83A00390920E8C07AB41EC1A206907500721240D001D62300608001404E02820EB01E01000C3000721240D081F508
0060800190930080C10300C00003202701000C1E00800106404E0200183C0000038C0106404E0200000084A139019093
00000000D1684E00E4240000E0690240D601C8490000C0D30480AC039093000080A709005907202701041D00B8C9BAC9
02B20E404ECA4940D001006EB200B20E404E02083A00709305907500721240D081F50858D400B20E404E02083AB01EAD
47B0A801641D809C04107480F50800186000909380A003AC470000030C809C041074603D0200186000E4240060F00000
30C000C849000083070060800190930000060F00C00063800190930000000061684E00E42400000040349A1300390900
00789A0090750072120000F0340120EB00E4240000E0690240D601C849004107006EB26EB280AC039093721210740080
9B2C80AC03909300820E00DC6401641D809C041074603D02163580AC03909300820EAC47EB112C6A005907202701041D
603D0200061800E42420E800EB1100C00003202701041D588F0000061800390900183C00000C3000721200C0E0010018
6000E4240080C1030030C0186000E42400000040189A13003909000000108DE604404E0200009E2600641D809C040000
3C4D00C83A0039090000789A00907500721240D001809BAC9B2C20EB00E4A49C04041D00E0260B20EB00E42480A00300
3759005907202701041D588F80450D20EB00E42480A003EBD17A048B1A40D601C849004107588F00800106003909083A
C07A040030C000C849004107D62300800106404E0200060F0000030C809C04003078000006180039090060F000000C30
06180039090000001086E604404E0200000044A339019093000080A70900590720270100004F1300B20E404E0200009E
2600641D809C04107400E026EB260BC83A0039292701410700B8C902C83A00390920E800C04D1640D601C849004107D6
23605103C83A00390920E8C07AB41EC1A206907500721240D001D62300608001404E02820EB01E01000C3000721240D0
81F5080060800190930080C10300C00003202701000C1E00800106404E0200183C0000038C0106404E0200000084A139
01909300000000D1684E00E4240000E0690240D601C8490000C0D30480AC039093000080A709005907202701041D00B8
C9BAC902B20EA80EA96FC4F9F631DF88FC73224BFF5E390920E800C04D1640D601A5944AA5EAEAEABEB549D37F88FCBD
B94AD6FA36FD96CC5F2A2701041D00B8C902C83AA841E99DC0759B3457A6E9FF5A8CE234B71F2EA861CEEC759BDF678B
8573B15BDFA63FA4C4BF514E0206C2B6F2E9679FBDB47CC57D0FCFBBE2A6DB464D9C7CF645979F74C1C583CFBE60C0B0
E1071C737297C38774E833B05DF783FFA3D301DBFC7ADFCDDAEDF9FD5DBAECD8A5F7AE7D8EEC367068FF93CE1E72DEE8
A66F197BE57557DE326DEA3DF7DD3B27D5F8F853CF2D59F6D13FFEE1FFE2588F80450D20EBA072056F062E46719AD50F
17FA9333BFBDE99D671D8B91F5C34B53F936FD8ACC6B58ECDF2827010361B17DFC8F4F9E5FB2BCE1A1472FBF61EAB051
E37B0C1EF6E3AE7DFEE927BB15E96BC72EBDF73FFAA4A65F74F9F553EF9D935ABEF275FF0E603D02163580AC83F2D762
B79CDEB49C96FBBFC6EF4EB3F64847383639F3FD84F92DC1AF4C50CCE65C4E02B46DD02D5AFCC295B74C1B78C6F93B76
E95DBC3239E4D7761DEA7A0C1E3662C21FA6CF7A70E98A55FE95C0031A00800106CA4DC8E31D727BE6381FA5177FA36F
E6DB0E385823EBAF2BD909C9A5F9BD721230102662F5DA0F663D32EF82CB27753FEAC42DDA776AF35639E06BDB5DBBF6
1C7C72D35BBD774EEA6FEFADF66F081ED000000C30D0B672CF310EE88DB35AD3423F92AF5922E748846C98B35E59B286
39F772C529E4E52460200C1F74F397ACCAFCAA6F684C7F65BDECBDF7D74EBD6756BF93477C7F8F9EE55C29077FFDA4AE
EFA033475E71D36DF39F7AF6D3CF3EF32F0C1ED000000C30504AB9F572A1DF12A1B0CDFC0971FADEACFE36F8C545AD79
937A937212A81AE96A37DFFFDA5CF9667DE5FB51074F98DAF49515744DFF24DFEB5B0CC6E6D72F7E79D9842937D71D79
42E5B6CA015F1DFB0F3EFDF7F553EF99B5E4D595FE3DC4E00100608081A2CAFD20BC90DF18671B73D62F8DF92784EFBA
D3477C44DE749DD4754EFC0DC849208CAC3DBDCD5FF95EDF5CF9A6ABDDF45740031CDCE896C9EB3B9F36EECC7197B5EB
7E705516CB2D7EFDAFDDF7EB7BC219175D7DFD83F3167CB06EBD8580C1030030C018602059913FB92FCE69CCC5FBACC0
D2EF4F8EF63E937D937212AA43C8531D9A5F9C59FC667EE57B7D4534C0C578FD76475FB865AFE337EFD8F7DB3FDDB376
8AE57C5FBBF43E7CE8C87137DEF1A7E7972CB7E830780000061820A638875D44EE75933A1FA35952E742175BF1CECA90
939088F07B7ADBE45407AF2FF8F5BF1DBBD56F4EFCEE5E7D8A5AD87EA77DE7CD3AF4DCBC73FFEF753F7ACB9EC7761832
E2F6FB1E9839B7F1E1058B163EF3DCF34B97BFFADA1B6FAF5EF3D0D32F35BD9FED8EBE70DB81E76F73D839DFEF7FC6D6
7D4FDEAAD709BB1C7DF6D82BAF3BFBA2CB8F1B3166C0B0E1FB0D1ABA43E75EA5FCB8C05EC79E3AE68A29B3538FADF9E0
EF420000806AA5398122897C3E468BDF1EBE2B4E7603738B6FA66C4BE6443EDC504E523BA29DEA90B5A1B7D24F75F0FA
08AFBFEFE179BD4F1C9EF88EE51D3AF7DA6FD0D0A123C74D9872F3C5374CDBE6D073B61B3C3A7CF686CFEA0F3FFAF8D9
97966ED567D8F7BA1FDDE7A473F738E8A8D214CE3FEF79C87123C6DC7CE7BD2B5E7F53FE0000504D34275024599B900B
2D3CA395BAC53B8E38FE870F964091B631CB49DA4A8B473A1474AA8346B4665F1F21BB5A7DFDA37F7DE2D4D197FEDB9E
FB2756BAEEB4C766BBF51C79D95533E736BEF7FEDA62BFFF165F9FFE7F52F9FC8B2F162D7EE1DA69779DF89F17ED79F0
D125689B7FB67FFF932EB878FAAC07DF5EBD46D601685D00641D902BE606E68D511BE6A296C0313BF3D228C691D172B2
3615F4496DE14F75D0885653E35A23AF7F71D9ABA3264EFE71D7BEC99C7AF1F37D36EFD8F77B3D7EFBFD43CE6AF5FD84
3F073B6B3D8659BF013EFBFCF3054F2DFEC3CDD38E3A6BE44EDDFA15BB6DEED067E0D9175D3E736EE3BA0F37C85E00AD
0B80AC03D2E2EFF88DD030677D4BB27F512A95CA6A98CBF353FF32DF6452BBB8E564F98870AA436E3DE55487C8FF6E7B
7D4DBD7EBBC1A3B73CF0B8CD76EB19BF41DDEA97FB1C71DA880BAEBE69C6A30BC3AFDFF2F1F6EA3533E7368E9A38F980
634EDEE6D7FB16B56DDEE7B021175C3E69CE630B3FFFE20B990FA07501907550B3B2CE6A8856F6C66C98133C22A3C56E
B96C0F642EC641197232BC168F74087FAA437936B41A48AF2FDBD787DFD3DBEA7ACCFCB15B1F74EAE67B1FF44F3FD93D
6659DAAEFBC1678EBBEC81C6055516742F2C7DE5E63BEF1D366AFCEEFD0615AF6AFEE7767BF43CE6E4F1936E7CF2B997
DC5C00B42E00B20E6A4D229F8B17E1488A627FC65FD35B4AA55265FEA97F590D7322BBAC2B22270BFAA4B6CCAF087B7A
357E5EEFF52D6AF14887300D7039ECE96DFABBB639FCDC7FE936F03B3FEF12B31AAD3BF28409536E5EFCF2B25AB8E36F
F8F8E347162C6AFA7B070C1BFEA34E3D8BD436EFD4ADDF99E32E7BE82F7F356201685D00641DD488AC8DBED1CADE421B
E6AC66357EC39CB57539F30796F981CC8917E07172B2D0531D0A6DA8347E5EEFF5911BE0E03DBDF9D66F150F84DFDDAB
4F9C16B4F790D36E9AD1F0EE9AB5B53C00BCB2EAF5DBEE9D7DC6D8FA8E030617A36ADEAE43DD31C3474D9FF5E0FA0D1F
19B700B42E00B20EAA552247646C2CBCA6CEDA5A1C73EF6EF04EE0A4FEC62249FC28E67C3999F9A96D4965ACD77B7D15
BFBEA04F6A6BF554878DDF34C0EE3B490D845BF51916A1F3DCF788E3AEBC65DA9BEFBCE71A66F9E4D3CFE62D7AFABFAE
9F7AD829E7EDD0B957E26D73FF93CEBEEDDED94DBFC5A5F6800658D400B20EAA4C56D39B54C3DC6A639CE0C7FC8529AB
E37F9461F1646DB12E5E4E6A14BDBE364F75A89D6DBDB5A3F9489CEFB4DF3B64BDB95BDF81974CBE69D98AD75CBD9056
BEF1D66DF7CE3EE1FCB13B75EB9760CFBC45FB4EC70C1F35736EA32BEC010DB0A801641D548D44AAD708357552B56AF8
D336CA761B73B2DBB9376A98BDBE3C1AE0F07B7A75BF8494FEB4CDCC7F15FFA5EBE1AD1E087CC1E5939E797189AB17C7
F295AFDF30FD9EA3CE1A99E0D1CD3FDCA3FBA9A32F4D2D7CD2E5F58006006080814A97C821CC113E4D2FA9C387C3BFF9
AC4E3B91CFD44B44E21FF6A761AECDD717F4496D995F01BDAECA973291DB2DA7BFB63974788BEDE566EDF61C72DEE807
E72D70E912B7F8E56557DE326DC0B0E1DBFC7ADF44AAE69FEE77D0F8C937BEF1F6BBAEAD07340000030C54A8424FB768
51848FD24BA4612EE8A88DAC223791138F1391F8871E6A982BFA54879AFAB0360896AF58CEFCDA6CD71E997565F7A34E
BC61FA3D1B3EFED8D52B81BF3EFDDCA5D7DC74E06F4FD9ACDD9EF1ABE6234E1B71FFA37F71553DA001001860A0B2B4D5
21CC1B0BD97B9CD40FD1306B8073B578A483531DA0CD85E996D35F5B1E706CFA348CD1575CF3F22B2B5CBA36F1C5975F
3EBC60D139E327EEDCA37FCC9EB9E9275C7ACD4D6FBDEBA3183DA001001860A0324438DDA2D51F12A6A64EA4538DF059
81E5791473CD36CCD14E75D0EB42156B5ADD4DCB3C64B7DCFCB5F541A78985F2F1EC4B4BC74FBE71EF437E1BB36A1E78
C6F98F2C58E47A7A40030030C040992BC6C7FC85F9218974AA11BA710D73E2A73AB478A483AA072854F84DCB2D7E05FC
7F3DA0ADAC7AF36FD7FCCF8C3EC79F1EA767DE6FD0D03B67CF71313DA001001860A06C45383F395784CF0A8CDFA946EB
C66BBC6156FC02652866B7DCF4BDC2ADCCBDB3FAFD4953EFE8367068E49E7997DE875F77FBDDAE24000049D1304382E2
37CCD14E724EBC610EF9018535DE30FB171E281F8AE51AB47CE5EB975E73D36E7D0746EB99B7DFFBC04B26DFF4C1BAF5
AE24000031694E2041F11BE6081B9873BF31C2F9CF59EF3CDABB2D93FF2B6898819AA25BE6E9175E1E79D955DBEF7D60
849E798BF69DCE193FF1DD356B5D468022F13401C83AA020311BE6681B98D36236CCD17E6FFC8F352C060D33500BE217
CB0E5BAE3E33EE7F28DA41CD9BEFDC71CC1553367CFCB16B0890384F1380AC030A12B3618EBC817963A20D73C86F2FE8
E8E6D426715E109E8619A86E362D136CE98A55232FBBEA479D7A16DA33FF60F76E975D77EBD75F7FED1A0224C8D30420
EB808244FBBCBC16BFB7D09638B3DC2E9F8639954A65BEB1DCD764BD20914238DA99D272122873896C5ACED72D0BBAAA
347DD683BD879C5668CFBC6397DE93FF38C3D503488A9B2C20EB808264ED9E2DA8E9CDDAC05C682F1AED20E5167F7BC8
B7DDEAC11A595723F78767BDE7A43AE1AC86594E0295AE049B96055D155BB4F88521E78D2EB467FE79CF436EBD7BA6AB
07E0690240D641E9456B98E3EFE38DB971B7D082BAD50DCCCD2F68FAC92D56BEE9DF58B749CCFDDB31FF16390994A798
C57241076208BAAAF7EA6B6F9C337EE2E63B772CA867DEF788E3520B9F74F5003C4D00C83A28A5AC7A334CD31BE76C8D
66591B866336CCAD7E7B7021DCFC669AFF96CC9FDFFC5F33FFD2CCF71FB3614EFCF307E5245062A53F6959D0D5880FD6
AD1F3FE9C6FFE87440413DF3B051E3DF7B7FADABE7010DB0A801641D9446A1076524522FA7C5F93905BDED30FBA59B5E
93F9CF73B734E7BEC3A4361E277BAAB39C044A66FE9255F50D8D6DF2297E82AED65C3BEDAE9D7BF40F5F326FBB6BD7AB
6F9DEEBA7940032C6A005907A511B2344EFC43EE627ED85FC83713AD120FF3C313D9789CF8C7FCC949A0044ABF6959D0
D164D2D43B76E8DC2B7CCFDCE5F0218FFEF509D7CD031A008001064A20ABE74C77AAA90CB91F7217BF0BCDDA871CE127
E4F6C099EF2AF76D877FCFADD6CB996F3E4ED39EF8C7FCC949A0A8DA6AD3B2A023EDABAFBE9A30E5E6ED3AD485EF994F
BAE0E277D73834C3031A008001068A2EB764CE2791931C36C63E8AB9A0B75DE836E330A76AB4F951217212289936DFB4
2CE8C8B47EC34717FEF7E4EFFE6CAF9025F3F777E972D38C06D7CD031A008001064A60CC9831B9DB959B4BDA44CE70C8
14F3A08CCCB79DEF3D4778DB61DAE3CCD744BE2CC5D8C02C278164955BB70CCDDE5EBD66F8C513C36F661E74E6C87756
BFEFBA7940030030C04069A4FE7FC5FB2D899FBC11FF3D87699813698693DA082D2781C42996A9142F2D5F71E829E786
2C997FB847F76933FFECA2794003000C300618A82699DB98136C59937A4BF95E93C8D6EB621C912127819874CB54A23B
67CFD9B947FF903DF371237EBFEEC30D2E9AC1030030C01860A03AC4FFBCBFC4B5DA1E27B2F7B8481B98E524104DFC62
B9BEA151B74CDBFAFD1FAE0D5932EFD8A5F73D0F3EE28A0100D426CD09549FA44E634E4466E31DE610E668E570914E60
96934004362D534D5E7E65C561A79C17B2673E75F4A59F7DFEB98B0600506B34275095324BE6E21DFB1C46690E612ED2
F9187212082F66B1AC5BA69CDDF5E7B9ED7B0C085332FFE28043173CB5D81503F03401C83A5907952EF18FFC8B2C4CD7
1D73C775F1CEC790934018362D530BBEFEFAEBB3C6FD57C8CDCCD74EBBCB1503F03401C83A57062A5DD6C1116D5532B7
BA3F39E0188D30EFB9D8F5B29C04F299BF64557D43A34DCBD494598FCCDBA95BBF3025F329175EE27201789A00641D50
E9B24AE62215B00132DBE3423FE62FBDF939B864CEDC205DBCBF4E4E0259AA6FD3B2A023BC8FFFF1C9B051E3C394CC5D
0E1FB2E4D595AE1850E3DC6401590754BAACE3324A5C3287D9609CDB4237FD9374751CF06E4BF977C949202DBD69B92A
0FC4107414EACED973B6DFFBC0564BE6AD7FD565C6FD0FB95C402D739305641D501DDA6A3373E6EFCDB71B39AB2B0ED3
48676E5D6EFACFC53E00444E02557FD2B2A02382F7FFBEEED8732F0CB39979D4C4C92E1750B3DC64015907548D319BB4
7A2A72B242FEBAAC371672C373C9AA723909B5AC463EC54FD011D9D47B666DF9CBCEAD96CC070D3DEBDD356B5D2EEB11
2C6A8B1A9075401548D7B9A5E96653195A7D5769C1AF4C37CC25D8B72C27A1C6D548B12CE848C48BCB5EED7CE8B1AD96
CCED7B0C58FCF23297CB7A048BDAA206641D501D4AD9D0CA49A082D45AB72CE8484A988FFFDB76D7AE0FCE5BE05A598F
00800106404ECA49A832F18BE5FA86C64AEC96051D09BAE67F66843996F9D6BB67BA56D623006080019093AE0C5487DA
DCB42CE82892C79E7C66E71EFD5B2D99274CB9D9B5B21E0100030C809C042A57CC62B9E9ABA2372D43F1ACFB70C311A7
8D68B5643E73DC65AE95C1030030C000C849A0E2D8B40C2570F1A41B5A2D998F3CFD775F7EF9956B65F000000C300072
12287FF1372DEB96A12033EE7FA8D592B9EB91C7BFF9CE7BAE95C1030030C000C849A06CD9B40C6D65DEA2A77FD4A967
70C9DCAEFBC12F2E7BD5B50200A85C9A130039095569FE9255F50D8DBA65685BCB57BEBE7BBF41C125F38E5D7A3FB764
996B050050A13427007212AA8C4DCB5056D66FF8A8CFF1A70797CCDBEF7DE0332F2E71AD004F1300B20E404E026D48B7
0C65EB84F3C70697CCFFBB638F279F7BC985023C4D00C83A003909949862192AC2A88993834BE67FDB73FFC79F7DDE85
023C4D00C83A00390994866E59D0515926FF714670C9FC83DDBA2D786AB10B05B8C902C83A003909144FFC62B9BEA1D1
651474B4893B67CF092E99B7DDB5EB5F9E78DA8502DC6401641D809C041267D3B2A0A30ADCFDC0C3C125F3D6BFEA925A
F8A40B05B8C902C83A0039092422914DCBBA654147F96878E8D1E092F97BBFD8FBE1058BAC47EB112C6A005907202781
386C5A167454ABFB1E9ED7EA4EE6458B5FB01EAD47B0A801641D809C040A15B358D62D0B3A2AC2ECD463DFDE69F78092
F9FFECD36BD9CAD7AC47EB110030C000C84920249B96051D35E5C1790B366BB76740C9FCAB5E87BDB3FA7DEBD1BF2A00
800106404E0201E62F5955DFD0A85B1674D4A0B9F31FDFA27DA7809279DF238EFBE4D3CFAC470000030C809C0472D9B4
0C3CB260D196BFEC1C5032F73BE14C830700800106404E029974CB40B3C6C79FFAEECFF60A2899878D1A6FF0000030C0
00C84940B10CB4E8BE87E70534CC4D5FE327DD68F0000030C000C849A859BA6520D81FFF747F70C9DCF402570900A0AC
684E00E424149B621908EFEA5BA70797CC8D8F3FE52A0100940FCD09809C84E2D12D03115C3CE986808679FBBD0F5CF9
C65BAE12E0690240D601C849A856F18BE5FA8646DD32D4B253475F1A5032EF3768A84B04789A009075007212AA8F4DCB
4052FA9D706640C97CCA8597B84480A7090059072027A13AC42C9675CB820E72ADFB70C3AE7D8E0C2899274DBDC35502
DC6401641D809C848A66D3B2A083E27961E92BDB75A80B28991F59B0C85502DC6401641D809C848A337FC9AAFA8646DD
B2A083629B9D7A2CA061FEBFFBFEE66FEFAD76950037590059072027A152D8B42CE8A0C4AEBAF5F68092B9CFF1A75B8F
80450D20EB00E424943F9B96051DB49561A3C60794CC63AE98623D02163580AC039093509E6C5A1674500EBA1E797C40
C93C736EA3F508006080019093505674CB820ECAC72BAB5EFFC16EDDF235CC3FEAD4F3AD77DFB31E01000C30007212DA
9C6259D0093ACAD3CCB98D01DB980F1A7A96F508006080019093D08674CB40991B3FE9C6809279EC95D7193C00000C30
0072124A2C7EB15CDFD0A85B064A63C0B0E10125F39CC7161A3C00000C300072124AC3A665A0E2BCBB66ED0E9D7BE56B
98DBF718F0D9E79F1B3C00000C300072128A2766B1AC5B06DAD6ECD46301DB984FB9F012970800A06434270072929A62
D332501DC65D755D40C93CE3FE875C220080D2D09C00C8496AC1FC25ABEA1B1A6D5A06AAC94143CFCAD730FFFB5E3DDE
5DB3D625023C4D00C83A00390931D9B40C54ABD7DE7A7B9B5FEF9BAF641E74E6489708F0340120EB00E4244493DEB4AC
5B06AADBAD77DF177056C675B7DFED12019E2600641D809C8482D8B48CA0A3A61C37624CBE8679F39D3B2E5D21D00037
5900590720272114DD32828E1AB47EC3473FEEDA275FC9DCEBD8535D22C04D1640D601C84908A05846D051E366CE6D0C
382B63C2949B5D22C04D1640D601C849C8A55B46D041DAF08B270694CC0B9F79CE7A042C6A0059072027212D7EB15CDF
D0A85B46D0516576EB3B305FC3DC71C060EB11B0A801641D809C049B96117490CF5F9E783A601BF3EF265C693D020018
6000E424B52966B1DCF465D332828E5A70D1D5D70794CCA9854F5A8F000006180039494DB16919410705D9FFE893F235
CC7BF51F6C3D0200186000E424B520FEA665DD32509B5E58FACAB777DA3D5FC97CD1D5D71B3C00000C300072922A66D3
32404C93A6DE117056C6B32F2D35780000186000E4245566FE9255F50D8DBA658044F43FE9EC7C0DF301C79C6CF00000
30C000C849AA864DCB00895BB6E2B580B332AEBAF57697080020299A130039495BD12D0314CFC41BFE98AF61DEA27DA7
956FBCE51201002442730220272931C5324069EC376868BE92F9F053CF737D004F1300B20E404E525974CB00A5F4C473
2F067CE4DFD47B66B94480A7090059072027297FF18BE5FA864697112082D1575C93AF61FEF7BD7AAC5DB7DE25023C4D
00C83A003949D9B2691941076DAE439F81F94AE6A123FD7B0EB8C902C83A003949D94964D3B26E1941078998F3D8C280
B332EE9D93728900375900590720272913362D23E8A00C9D31B63E5FC3FCB3FDFB7FF9E5572E11E0260B20EB00E4246D
2866B1AC5B46D041517DF2E9673FEEDA375FC97CEE25FF6D3D02163580AC039093B4099B9611745011EEFAF3DC80B332
1E7FF679EB11B0A801641D809CA464E62F5955DFD0A85B46D0410539F6DC0BF335CC3D060FB31E01000C3000729212B0
6919410715EA9DD5EFFFEB6EDDF295CCD74FBFC77A040030C000C8498A47B78CA0834A77E31D7FCAD730FF708FEE6BFF
BEDE7A040030C000C84992A55806A826BF39EEF47C25F3A9A32F35780000186000E42449D12D03549FA75F7839E023FF
520B9F34780000186000E42471289601AADBF9F557E66B98BB1C3EC4E00100608001909344A35B06A8055F7CF9E54FEA
FAE62B99AFBAF576970800203CCD09809C247EB15CDFD0A85B06A820B7DFF740BE8679AB5FEEF3B7F756BB4400002169
4E00E4642DB36919A0661D72F239F94AE6E37F37D6F5013C4D00C83A0039493E318B65DD3240157869F98A808FFC7BA0
71814B04789A0090750072922C362D03D0ECF77FB8365FC3BCC74147B93E80A7090059072027499BBF64557D43A36E19
041D64F9C50187E62B99EBAFBDC5F501DC6401641D809CAC71362D83A083007F7AE8917C0DF33FB7DB63E51B6FB94480
9B2C80AC039093B5C9A665107410C651678DCC57320F3EFB02D70770930590750072B2A6D8B40C820E0AB2E2F537BFF3
D33DCAE423FFAC477093059075007292B6A25B064107D14C987273BE86B9E380C1D623605103C83A003959C514CB20E8
20BEDDFB0DCA57325F3D75BAF508006080019093D547B70C820E92F2E7D4FC7C0DF3761DEAD6AE5B6F3D0200186000E4
6475885F2CD73734EA9641D04196A3CFFACF7C25F31963EBFF1F7BF701265579E80DDC5C638B25764D72538D1A352616
88221640111551C106226041B122D8B1AD581256A382BD46B1F715153BD935804AEC158D05EC8ABDA02228DFB9EEBDFB
CC37BB677666CECCEC29BFDFB3CFF7E4F32ECBF2CEBEFFF3BEFFF3EE19F31100C00206404E269D43CB0054CFCB335E0F
6B98838F694F3D6BE101006001032027932862B1AC5B06A048278EBD30AC61EEBDD7C1161E00001630007232591C5A06
A096BEFFFEFB55BB6F1756325F7DDB440B0F00000B180039197F53A6CFA86F68D22D03507B57354C0C6B9857EBB1BDF1
0100C8A339019093B1E2D032001D6EEB3D0E0C2B99478FBBC8F80000E4D29C00C8C93870681980F878E4C9670ABCE5DF
2B33DF304480DD0480AC03909331E1D032003174F089F5610DF3A091C71A1FC06E0240D601C8C90EA75B0620B63EFCE4
D3E5D6DD2CAC64BEBB69AA2102EC2600641D809CEC108A6510749008675F716D58C3DCA5DF60E303B8C802C83A003959
63BA651074902C1BF41D1456329F73C575C60770910590750072B206A217CBF50D4DBA651074507B131B278735CC2B75
EAF1F997B30D11E0220B20EB00E464F538B40C820E926EE088516125F3117F3BCB7C044C6A00590720272B2E62B11C7C
38B40C820E62E23FAFBD1ED630071FCFFFE755F31130A9015927EB00E464A538B40C820ED2A7EEAC0BC21AE6FE071F65
3E020036141630007232A229D367D437344539B4AC5B064107B13577DEBCDF6EB24D58C97C77D354F31100B0A1B08001
9093E5716819041D64C16537348435CC1BED38C47C04006C282C6000E464499A0F2DEB9601C88E4D77DD2BAC64BEE8DA
9B2D3C0080CCB28001909325716819806CBA6FF2C3610DF32FBBF4FA66CE1C0B0F00209B2C6000E4649174CB0064DCEE
238E092B998FFDFBB9161E00403659C000C8C9C214CB00D0ECC557678435CCC1C7CB335E3744004006694E00E46418DD
3200E41975DAD9610DF3E0438F333E004006699801E4649EE8C5727D43936E1980549AFDD5D7BFD8B06758C9FCC0D469
8608C8F86E029075B20E20CB39E9D03200B4EB82AB6F0C6B98BBF51F6A7C80CCEE26005927EB00329B930E2D03404936
EC3B28AC64BEFCA609C607C8D46E0240D6016439271D5A062C08A10C77FEF35F610DF32A9B6DFBFDF7DF1B22C0451690
75460620C53919B158D62D83A0037639F088B09279F4B88B8C0FE0220BC83A230390CA9C746819B020848A7866FA7FC2
1AE64556EFFCC63BEF1922C04516907500A42627A74C9F51DFD0A45B062C08A1820E3BF58CB09279D8B1A7988F80490D
C83A005290930E2D03168450259F7CF6F98A9DBA8795CC8F3CF98CF90898D480AC0320B939A95B062C08A1DAC65D7E4D
58C3DC67EFE1E6230060430140E27252B10C5810422DADD37BD7B092F9D67B27998F00800D05809C4C4A4EEA96010B42
A8BD1B27DE17D630AFDF6780F90800D85000C8C998E764F462B9BEA1C90B0D0065EBBDD7C16125F3B9E3AFB7410300D2
CD020620B939E9D03200C4C1D4C79E0A6B987FF6972DBEFCEA2B1B340020C52C60001297931539B4AC5B06800A1A7AF4
E8B092F998D3CFB141030052CC0206204139E9D03200C4D38C37DF0E6B98838F9767BC6E880080B4D23003C43F272316
CBBA6500A88113CE3C3FAC61DEE3F0E38D0F0090561A668038E7A443CB00901473BEFDF6D75DB70E2B999B1E79CC1081
DD84D60590750072B236393965FA8CFA8626DD320024CB45D7DE1CD6306F39787FE30376135A1740D601C8C96AE7A443
CB0090681BF61B1456325F7BFBDDC607EC26B42E80AC03909355CA498796010B424881DB1F7830AC615E6BCB1D8D0FB8
C8BAC802B20E404E5636271D5A062C082165FAED776858C9FCF78BC71B1F7091759105641D809CAC484EEA96010B4248
A5C79F7D21AC615E769D4D3FFAF43343042EB2460690750072B2EC2FA858062C0821F50EAA1B1356321F76EA19E623B8
C81A1940D601C8C932BE946E19B020848C7867D6074BACD525AC647EEEA557CC4770910590750072B248D18BE5FA8626
DD32604108C9F2B7F32F0B6B98773B6494F90800D85000C8C97639B40C58104296ADDA7DBBB092F9BEC90F9B8F00800D
05809C6C53C46259B70C5810423A5C71F384B086B95BFFA1E623006043012027F338B40C00E4DAACFFDE6125F31537DF
6E830600A480050C40F49C9C327D467D439343CB00409E7BFFF55058C3BC5A8FED6DD0008014B080018892930E2D0300
850D187E7458C93CE6827FD8A0010049670103505E4ED63734E9960180763DFBE2CB610DF3927FDCE8BD0F3E32440040
A2699801CACB49C5320050A4434F3923AC641E3EBADEF8000089A66106282F2775CB0040913EFCE4D365FEBC4958C9FC
E4F32F1A22C8DA6EC2C800B20E32629145165970C105175A68A185175E78D14517FDC94F7EB2C4124B2CB5D4523FFDE9
4F97596699E5965B6E8515565869A595565E79E59FFFFCE7FFFDDFFFFDAB5FFDEAD7BFFEF56F7FFBDB55565965D55557
5D7DF5D5FFF0873FACB9E69A7FFCE31FFFF4A73FADB3CE3AEBAEBBEEFAEBAFDFB973E70D36D8A04B972E1B6DB4D1C61B
6FBCE9A69B76EBD6AD478F1E9B6FBE79CF9E3D7BF5EAB5F5D65BF7EEDD7BDB6DB7DD6EBBED76D861877EFDFAEDB4D34E
3BEFBCF3AEBBEE3A60C0808103070E1A3468F0E0C17BECB1C75E7BED3574E8D07DF7DD77D8B061FBEFBFFF81071E78F0
C1070F1F3E7CC48811871E7AE8E1871F7EE491471E75D451A3468D3AF6D8638F3FFEF8BABABA134F3CF1A4934E3AE594
53FEFAD7BF8E1933A6BEBEFEF4D34F3FE38C33CE3AEBAC71E3C69D7DF6D9E79E7BEEF9E79F7FE185175E7CF1C5975C72
C965975D76F9E5978F1F3FFEAAABAEBAFAEAABAFBDF6DAEBAFBFFEC61B6FBCF9E69B6FB9E596868686091326DC71C71D
77DE79E75D77DD75CF3DF7DC77DF7D0F3CF0C0A449931A1B1B1F7CF0C1C993274F9D3AF5A1871E7AE49147FEFDEF7F3F
F6D8634F3CF1C4934F3EF9F4D34F3FFBECB3CF3FFFFCF4E9D35F7CF1C5FFFCE73FAFBCF2CA6BAFBD3673E6CCD75F7FFD
CD37DF7CFBEDB7DF7DF7DDF7DF7F7FD6AC591F7EF8E1C71F7FFCE9A79F7EFEF9E75F7CF1C5ECD9B3BFFEFAEB3973E6CC
9D3B77DEBC79DF7FFF7D87E4A462190028DEDF2F191FD630EF7CC011C607B42E00B20E52E9C73FFEF10224C18F7EF4A3
D6F702965C72C9DC7B012BAEB86279F702CAEE96571F72F43683F6C9BD17B0CF3EFBD4FE5EC04D37DDD4E1F702BEFDF6
DB1ADC0B0080985BB367BFB0927962E364E3035A17005907E9B3E0820B2A6F29B5585E7AAB3D7FB24EF78556FE8DA10B
BB17F0E31FFFB81AF7020AFF5EC036DB6C93FB7B013BEEB863EEEF05ECBEFBEE35B81770D14517E5DE0BB8F2CA2B6B7F
2FE0BDF7DECBBD17F0D9679FB917800521D4CC3513EE0A6B98BBEEB487F1011759005907E9F3A31FFD48254849DDB262
19F7024ABD17D0BD7BF7DADF0BF8DBDFFE967B2F60ECD8B1B5BF17F0C20B2FE4DE0B9831634607DE0BB020849AD962F7
616125F3A5D7DF6A7C40EB0220EB204D829DBB9689221B6687968144DF0B080BB7DADF0B38E184133AFC5EC0E38F3F1E
DB7B01A440E3C38F8635CCBFDD649BB9F3E61922D0BA00C83A488D604BA8BBA070C3ECD03290EEA0333231B917B0DA6A
ABE5FE5EC09FFFFCE7DCDF0BD870C30DAB712FE0A0830ECABD1770C41147D4FE5EC0FDF7DF9F7B2F60CA9429E9B81730
F8D0E3C24AE613C75E6883065A17005907A93177EE5C3B620A142FBA6540C30C59BB17B0FCF2CBE7DE0BF8E52F7F59C6
BD80CE1B6D1CD630071FCB0D38BACDF9D8722F60EFBDF7AEFDBD801B6EB821F75EC0EDB7DF5EFB7B01DF7CF38DDF0B40
EB0220EB20598215AC7D168A1740D00115F75FCBFF7758C3BC78D7BEE66331F702FEEBBFFE2BF75EC0E28B2F5EF17B01
EDFE5E40DFBE7D737F2F60B7DD76CBFDBD80DADC0BB8E28A2B3AFC5EC0279F7CE25E00808659C30C61824551B0340A16
4873E6CC09164BC19229583805CBA74F3FFD34584A050BAA6059152CAE822556B0D00A965BC1A22B587A050BB0601916
2CC6822559B0300B9667C1222D58AA050BB660D9162CDE82255CB0900B9673C1A22E58DA050BBC6099172CF61E78E081
60E1172CFF824560B014BCE38E3B826561B0380C9688C142F1C61B6F0C168DC1D231584006CBC8F1E3C7074BCA606119
2C2F834566B0D40C169CC1B233587C064BD060211A2C47834569B0340D16A8C1323558AC064BD660E11A2C5F83456CB0
940D16B4C1B2F6C8238F0C96B8C1423758EE068BDE830F3E38580007CBE060311C2C89870E1D1A2C8F834572B0540E16
CC03070E0C16CFC1123A584807CBE97EFDFA054BEB60811D2CB383C576B0E40E16DEC1F23B5884074BF160411E2CCB83
C579B0440F16EAC1723D58B4074BF760011F2CE383C57CB0A40F16F6C1F27EF5D5570F96FAC1823F58F6078BFF5FFDEA
57C14620D80E049B82606BB0C20A2B04DB8460B3106C19828DC3124B2C116C2282AD44B0A108B6150B2EB860B5DF9051
F1026898814A57CC0BFEF8B77F0E2B9997E937C27CC4BD8028F702B6DA6AABDADF0B38F3CC3373EF055C70C105B5BF17
F0EAABAFE6DE0B78E79D77DC0B0034CC001D752FE089279E68B917109693D5BE17505757D7E1F702D659679D04DD0B00
34CC90A48E79E915C31AE6C53A6F6D3E02E9B817F097BFFC25F75EC0669B6D56FB7B01A79E7A6A87DF0B78EEB9E7DC0B
40C30C2027E56447DD0B98397366EEBD80E9D3A787DD0B08967353A74EADC6EF058C1B372EF75EC09831636A7F2FA047
8F1EB9F702BA74E9E25E0040D22DF8CB35C34AE69FF61EA66106A082F7027EFFFBDFE7DE0B587BEDB56B7F2FE0B0C30E
CBBD1770DC71C7D5FE5EC0BFFEF5AFDC7B018F3EFAA87B019A130039091D752FE0FDF7DFEFF07B01175F7C71EDEF05EC
B4D34EB9F7027AF7EEED5E00506651B0C4D2610DF3227FEAA661068054DE0B5876D96573EF05FCE217BFE8F07B01071C
7040C47B0161CD49817B01B7DD765BEEBD807BEFBDD7BD0040C30CD051F702E6CE9D9B7B2FE0F3CF3FAFFDBD809B6FBE
39F75EC055575D55FB7B0183060DEAF07B014B2DB554EEBD8085165AC8BD000A5BF067BF0F2B9997EC3958C30C002442
421FBB97772F60B1C516EBF07B01FDFBF7CFBD17B0E79E7B56F65E4031BF1770DD75D775F8BD80AFBEFAEA9B6FBEB1E5
47C30C001D722FE09FFFFC67EEBD80891327D6FE5EC021871C927B2F60DF7DF7ADFDBD80DFFCE637B9F702565A69A578
DE0BF8D1A28B8735CC3FFEDD3A0B2CF8633B560040C34C362DB6D862F69268980100F7025AEE053CFCF0C36DDE0BE8B3
C7FE612573EF41FB56F65EC0D1471F5DFB7B01EBADB75EEEBD8035D6582389F7020080362DB4F26F96DE6A4F0D33D5B0
F8E28BDBA1A061060068D7DBEFCFFAC91A1B8695CC4F4F7FC91055FC5EC04B2FBD947B2FE0A9A79E2AE65E40657F2FE0
B4D34ECBBD17307AF4E8DADF0BE8DAB5AB7B01CEF50154A35896755444B0BAB0EE45C30C00508CBF9E775958C3BCCB81
47181FB2732FE08D37DEE8F07B01E79C734E65EF0584ED2672EF05F4E9D327F75EC0165B6CE15E0090F46E59C34C74C1
65CB6A0A0D3300409156EDBE5D58C93CB171B2F101BB891ADF0BF8F2CB2F73EF057CF0C107B5BF1770CD35D7E4DE0BB8
F4D24BAB7A2FA0CDDF0BD865975D3AFC5EC0228B2CE25E00892B9635CC54C4D24B2F6D2181352180A0038A34FE963BC2
1AE68D761C627CC04516E6FF702FE0BBEFBEABFDBD80BBEFBE3BF75EC0ADB7DE5AFB7B01FBEDB75FEEBD80214386D4FE
5EC0CF7EF6B3DC7B014B2FBD7446EE0594D72D6B98896ED9659795FC581302083AA078DD07EC1356325F7CDD2DC6075C
6481A4A8D9BD80A6A6A6EADD0B38F084BFF639EC9475F73FA9BC6EB9F9A34AF702D65A6BADDC7B01BFFBDDEFDC0B48A5
E0159427581302083AA078931E9A16D630FFB24BAFAFBEFEC610818B2C400D4C993E638731E3A314CB49C9BA0ADE0B78
E6996772EF054C9B36AD66F7025A7E2FE0E4934FCEFDBD80638E3926F7F702468E1C5983DF0BE8D4A95305EF05049F63
3E624D0820E880920C1A796C58C93CEAB4B38D0FB8C802545505BB655907604D08083A4107B5F7D26B33C31AE6E0E385
975F3544E0220B5071118BE5027FD6D802581302820EA8B125BAF50F6B98070C3FDAF800001514BD5B0EBE820D0540BB
E42420E8809A597E70DDC26B760D2B9927364E3644004044D18BE5FA86261B0A80E2C94940D001B59C8F4B6DB55758C3
BC61BF41860800285BA50E2DDB500094BAD1939300402D171E8BAED733AC643EEFCA1B8C120050928A1C5A6EB35BCE5D
C0684E00E42400109385C7D23B1C14D630AFDC79F3CFBEF8D2400100C5A8D2A1E53617309A1300390900C467E1F1932E
DB8795CC879D7A868102000A88582C17DF2DE72D6034270072120088895766BE11D630071F4FBDF0922102005AABCDA1
E53C9A1300390900C450DD59178435CC3BEE7F98F101005A4C993EA3BEA1A9F6DD7233CD09809C04006268DEBCEF56D9
6CDBB092B9E1DE7F1A22B09B00E89043CBB20E404E02008970C5CD13C21AE6F5B61D607CC06E02C8B20E3CB42CEB00E4
240090149B0FDC37AC643EEBB2AB8D0F2D1A1B1BEBEAEA82FFD750141EA51A0F91DD0450717138B42CEB00E42480A083
A4687AE4B1B08679D97536FDE0E34F0C1181BABABA057EA0610E138C4CB76EDD82210AC6CA451648A81876CBB20E404E
02083A88BFBD8F1A1D56321F78C2DF8C4FC6B514A71AE6228748C30C244E6C8B655907202701041D24C2CCB7DE5964F5
CE6125F3E4479F3044D994579C6A98DBD472BA5BC30C2451FCBB655907202701041D24623E9E7ADEA5610D738FDDF631
6E19D4BA38D530E769B381D730038910BD58AE6F68AA4DB72CEB00E42480A083A4CCC73F6CD137AC64BEF8BA5B0C5D76
8415A71AE6228748C30CC459520E2DCB3A00390905766406A1DD214ADF28093A884F1E16988FD7DD714F58C3FCF30D7A
7EF6C5975EB52CFC58B614A7C1FF68FE29CD3BCCEC529E3B20C1FF0E06246F886ADC3003142362B1DC81DDB20D05809C
84D69B56DBAE768728F87F051DA47BB237D750319C8F3BEC3B32AC641E79F2DFBD76A9D7522FE75EAC5B7E6235CCB9A3
9177B1EEC077FA03282CA187966D2800E424E4EDC56CBBDA957BFC49C30C19C9C318CEC7A7A7BF14D630071F8F3CF98C
17310B3FA2ADFFBB8639EF92DD7A102C7580B899327D467D4353720F2DDB5000C849C8DD8575ECC30993B29DCF1DA2F4
35CC409B7918CF85C7A8D3CE0E6B98B71A7280D7319B34CCEDD23003F1918E43CBA52E6000324E4E9256317983F5240E
918619D237D9DB7C3BB0782E3CE67CFBEDEF36ED1D56325F71F3042F68066998DBA561063A5CF3A1E5F475CB452E6000
324E4E923EB17A83F5D86A7D9A51C30C59CBC3D82E3CAEBCF5CEB086F9D75DB7FEEAEB6FBCB259A3616E978619E840A9
3CB45CC6020620CBE42429E30DD6DB957B9A31D890366FD57377A61A66486B1EE6B5CD71FEE67BEF757058C97CE498B1
5EDCACD130B74BC30C74882C74CBCDE2D39CF4EAD52BF7B23879F2643F87809C84CA6AA953F2F6561AE616B9EF419FBB
494FF73BFD4106B5197A798FCB88F3F7FFE8D3CF1778CBBF279E9BEE25CEE6CFB386398C8619A8A5EC14CB2DE2D09C4C
9E3C39AF5ED630037212AAB7C3F206EBED0E51EB41D03043CA0493BA751E26A8610E1C76EA19610D739FA1C3BDC499A2
612EE6E2DE514B1DBB09C8940C76CB31C9BAFAFAFA361F7AA66106E4246461DB95141A66C8826435CC5F7EF5D52F37EA
1556325F7DDB442F68766898E3BCD4B19B802C885E2CD7373425B45B8E43D6B53EBAAC6186440866687D7D7DEB291CFC
97FA1FA4E91F6B4D886D17F335CC900DC96A980397DDD010D630FFBE5B9FB9F3E6794D3342C31CE7A58EDD04A45B660F
2DC724EB264F9EBC40411A6688A7B0DF3B682D353DB33521B65DCC4F7BC32CE8A059E21AE640CF41FB79CB3F34CC715E
EAB8C8422A452C96838FA41F5A8E43D6859D7BD430439CB57B63A8B5606AA7602E5B1362DBC57C0D336443121BE6871E
7FBAC05BFE05FF572F6B166898E3BCD4719185947168390E59D7E69BFA351F74D430439C15AE977BFD20ADBF92604D88
6D17F335CC900D496C9803C347D78735CCDD07ECE365CD020D739C973A2EB2900ED10F2DA7B55BAE7DD6B5FEE5FA5EBD
7A85FD5F35CC101F6DD6CBC19C6D739EB6F97CE644CF686B426CBB98AF61866C8843C35CC67CFCF8D3CF56EAD423AC64
1E77F9355ED9D4D330C779A9E3220B49E7D072ACB22EECE8720B0D33C453EB7A39F7DE50F1533EB993DA9A10DB2EE66B
98211B12DA30072EB9EE96B0867989B5BACC7CEB1D2F6EBA6998E3BCD4719185849A327D467D43936E39565997D750B5
F968560D33C4535E575C4CBD1C36F1E524D87625978619B220B90D73A0F75E078795CCBB8F38C68B9B6E1A664B1DA082
1C5A8EED8622B768CA3BBADC42C30C3114B125CE6BA7C3A6BF9C04DBAEF8D330431624BA617EEA85970ABCE5DF4D77DD
EFF54D310DB3A50E5011BAE5986F289A4BAA368F2EB7D03043CC17ABE5EDB3724BE6841E6356BC60DBC57C0D336443A2
1BE6C089632F0C6B9857EBB1FDDC79F3BCC45958B46B982D758052299693B2A1983C7972BB671735CC1037790798CB3B
819C37B593788C59F1826D17F3D3DE3003CD92DE3007D6E9BD6B58C97CE498B15EE2B4D2305BEA00E5D12DC764015341
1A66889BBC595991156F128F316B98B1ED62BE8619B221050DF3C4C6C9059E95F1D0E34F7B955349C36CA90394247AB1
5CDFD064182BB880A9200D33C479A51AE5EC71DED398E524D87625918619B220050D73609F51278535CCDD07ECE3554E
FDBA5DC36CA90314E0D0726C173095A2618658C97B4446942999F407656898B1ED62BE8619B2210E0D7374B33EFA78C5
F5BB8795CCE32EBFC60B9DEE9F5B0D739B7287C8520732A822879675CB0568988162A6640557741A6688210D73BB34CC
9005E9689803975C774B58C3BCC45A5D66BEF58ED73AAD5728D7F162A6B6EB38648A43CBB5A16106DAB4C0FF2FE257CB
7D5046E21EC5AC61260B34CC25EDDFED4C21AD52D330077AEF757058C9BCFB8863BCD6A9F989CDAB975DCADB9DD72D43
54B3C3DE7613D0212216CBBAE5E4669D861962A5B2A78E13FD28666B42B240C3DC2E0D3364419A1AE6A75E78A9C05BFE
5DD530D1CB9DDCEB51B71F2C5084E6CFCCDA953D98C8A50E5155AFEC761350630E2D67BC39D130437CE43D84397AC35C
D9676E643627A17A34CCC5ECE835CC907A696A9803278EBD30AC615EB9F3E6EF7DF091573CE997EC2265EDB2D5E6A1E5
76D94D40D24D993EA3BEA149B7AC39D130437C54F06DFEDAFC82C97A14B3352159DBAE6A98DB94EE8659D041B39435CC
81757AEF1A56320F1A79AC571C5C6421051C5A9675B934CC101F159F8F1A6688390F6F6C978619B2207D0DF3C4C6C905
9E9571E5AD777AD1C14516924BB72CEB5AD330437C54E3A1161A6688ADBC3709F2088876432C1DBD93A08376F3B0666F
045655071CFFD7B08679A54E3DDE9DF5A1D71D5C64215914CBB2AE000D33C44735DE982FF70B065F5F4E424C84BDC1BA
91C91DA2D60FBD4CD91B27093A98DFAA5EEEA8995EF1F9F8D5D7DFACDA7DBBB09279F711C778E9C145169242B72CEBDA
A56186F8D0305B13926265BC077D064F35B7FCC38B1FA2449F7514746453306D8B8CC4E6AA395083995E8DF938E1FEA6
02CFCA187FCB1D7E18C04516E24CB12CEB8AA76186F8A846C39CFB3535CCD0815A1FD2EBD837588FA7328648C30C8953
E4BDB61AFF8A4795E6E3417563C21AE615D7EFFECEAC0FFC3C001043BA651B8A526998213EAA71DE58C30C20E820561A
CB92D0F9F8F5377356EBB17D58C93C70C4283F0F00C447F462B9BEA149B79CCD0D858619E2A3DA0D73B2CE432A5E000B
422005F3F1F6071E2CF8AC8CDB0D3E65687EDA4C33A30144E7D0B2054C441A66880F0D733C731200B0418BE2E013EBC3
1AE615D6EFFEF6FBB38C3FA56A79FA5606DFB501A8A088C5B26E39DD0B98926898213E34CCF1CC4900C0062D8A6FE6CC
597DF31DC24AE6DD0EF1AC0C4AA66106227268D902A6B234CC101F1AE678E62400608316D11D930A3D2BE3F29B267809
2889861928CF94E933EA1B9A74CB163015A76186F8D030C733270100A21B3E3AF459198BAFD9E585975F3544144FC30C
94CAA1E554D23003AD6998E399930000D17D3B77EE1A3DFB8695CC3D76DBC710513C0D33503C8796534CC30CB496DB06
57A361AED4D7CC5A4E020054C49DFFFC578167651C5D3FCE1051240DB3DD04B4CBA16559574B1A66888F6AB4C1D53817
6D4D0800509E434E3AAD40C97CEBBD930C11C5D030DB4D4001BA655957FBEF44C30CF1518D275A68980100E263EEBC79
EB6EDB3FAC61FEC5863DDF7AEF7DA344BB34CC7613D09A6259D6699881F9D56F9883F92E2701041DD0B1A63EF6548163
CC7D878D3444B44BC3EC220BB974CBB24EC30C84CDC7E85F3098D11A66004107C4CD59975D5DA064FEDBF99719220AD3
30BBC8C2FC4A14CBF50D4DBA655957111A66888F8ACF470D3380A003E269C0F0A30B94CC93A64E334414A061769125E3
1C5A266E59A76186F8A878219CF7059335C1AD09010B4220C5F3F1C34F3EFDDDA6BDC31AE6357AF6FDF2ABAFBC2E84D1
30BBC8924D118B65DDB2ACAB1E0D33C44A651BE6444F706B42C0821048F77CBCAB714A8163CC7B1C7EBCD785301A6617
59B2C6A165629E751A668895DCF9D8AB57AF885FAD1A6F1D684D0820E8C07CAC9413C75E58A0643EFFAA1BBC34715896
5744917F6F6363E30255505757D781FFF60E1949488729D367D4373439B44CFC37141A668895DC4E387AC39CDC87302B
5E000B422023F371EB3D0E2C50323FFCC4D35E9D0EA161D63043C7726899646D2834CC102B795332CA974AF4DBFC295E
000B422023F3F1B537DE5A6EDDCDC21AE635B7ECF7F1A79F79816A4FC3AC61860ED17C6859B74CE236141A6688950AD6
C2499FDD8A170020230B8FEBEEB8A7C031E67EFB1DEA05CA94C62274EBD6ADE539CC8DC531B010730E2D93B8054C2E0D
33C44DA51EC59CF47BFD1A6600203B0B8F11279D5EA0643EE6F473BC46E4F24E7F9026BA6592BB8069A16186B8A9C8DB
F325FD1119F335CC0040C6161E1BED38A440C93CFE963BBC4CB4D030430A289649C702A69986196228FA31E68AD4D472
122A35916BF9D4C1CF3EFB6CC92AE8DAB56B07FEDB3DBF11C882A7A7BFB4E41F370A6B98175AB5D3438F7BD73FFE9786
19124DB74C829A93FAFAFA5EC5C9DB7615FE641534D440DEC42C75DEA5E000F37C0D33C9D751BDE8AC59B3AAD1CA76EE
DCB903FFED1A6690871999C537DC796F8163CCAB6CD6E7F5B7DFF513C57C0D332453F462B9BEA149B74C8D9B93BCC3C9
9592D0AA0A9225AF222EF51873DEB44DE88D210D331A150DB38619E4610667F1E8711715289937EBBFF7F7DF7FEF870A
0DB3DD04C9E2D032C9CD3A0D33245ADE31E6E24BE6BC3F98DC396B4D08E5993D7BF67355306BD62C630B9447C35CAADD
0E1955A0641E3862941F2A34CC76132442C46239F87068990ECF3A0D33245DA925F3E4C993CBEEA5AD09010062E2ABAF
BF59B777FF0225F39163C61AA58CD330DB4D10730E2D23EB8098C87B564681BB3CC167B67953494E02083A20899E786E
7A8177FD0B3ECEBAEC6AA394651A661759E229FAA165DD32B20EA8B8364BE6DC37E52CF07F4DFAFB72CA49C08210C8B2
DBEE6B2CD030071F97DF34C128659686D94596B87168195907C459EB675FB42BD10FC790938005A191019A9D3BFEFAC2
25F3ADF74E324AD9A46176912526A64C9F51DFD0A45B46D60189505F5F5F64CF9CF4A3CB7212B0203432603EB638BA7E
5CE192F99F0FFDDBCB97411A6617593A9C43CBC83A20A1EA7FD0BA6A0EFE4BF0DF53D32DCB49C08210301F5B0C1A796C
818679F9F5BA3DFEEC0B5EC19A696C6CACABABEBF68396D578F0BFEB7E50B36F230E0D733014CD8310FC8F885FA4F560
BAC81267BA652C6000E42480A00392351F7BECB64F81927995CDFABCFAFA9B5EC4AA6A6953DB559B9EB9A598AD65AD9D
ABA5E28ED230171ED28EFAA74118C532163000721240D001099D8F1F7CFCC97ADB0E285032AFDF67C0879F7CEA75AC86
D6DD72F0FF6DFC3FCD479733558DB61E90F21AE6962FD23C9ECDFFB1F97CB89299B8D12D630103202701041D90F4F9F8
CACC377EDFAD4F819279F381FBCE9D3BCF4B59717905729B9D676363635EC91CE5C111B1157690BB8C7F6CEED729FC7F
5532D381A217CBF50D4D86111B0A003909005878C4C4E3CFBEB0FC7ADD0A94CCFDF63BD44B5971B90D7381B633AF644E
DF1BF0B53EAA5D76C3DCEE90E60D662AEB7A62CEA1652C6000E42400402A171E93A64E2BD030071F7B1E718257B3B272
EBD0C29F9977BE373523907B74B9F9C9CF119F92514C0F9FF7DE7F7E0EA98D8A1C5AD62D6301032027A106DBB4B0B74D
CFE0EF81366FDB23FEC3F3C6B303DFF6082863FE36B7552D3A2A0C13B4F0B8E59E49854BE6A1478FF6A355F14B553157
AB5436CC6D3E2D39CA19E3F2CE84FB39A4DA1C5A467302202721FE5AB7CA197FE7F4DC035165FF930BFCC6AE2737429C
A77FD8D35C4DE1225D7ED384C225F37EC79D6A946A2FEFAA948E673B84DD0B2EFB5F5AFC1FF4A00C6A2062B1AC5B4673
022027A1665A1729796F439FF1F7A02FEFDF9B7708AA79EF99379E7EAF16623EFD737F89A3CD7B466AA530675D7675E1
92F9A0BA3146A9C6F27E86D334735BFFC7F2E6694927933D2883AA726819CD09809C8464EDCB8A79EB9F888F344CE8EE
BBEC8639F7EBB41E2B2533242201DA9C9EC57C0EFF3B56675D50B8641E71D2E946A9037FC2D3FD8F2D6FD152D2053A62
C36C37419BA64C9F51DFD0A45B467302202721598A3CAE93FAF7A02FF04BF1A536CCC53CC231F7EF72061262A2C8E7AF
66E48E5B458C38F9F4C225F3E17F3DD328D54CEE8F6EEA1FF35283863962636F37411E8796D19C00C849394972E556C7
85F79B692D55F2DE83BEF937E22BD5301733ECCE40424C143379E767E08E5B65ED7FDCA9854BE6A3EBC719A5DAC8D483
C4CB5BB19454C247793F41BB0972E996D19C00C8493949D2E5EE8FAABA938AA7B0C75994BD0D2FEF3DE8E33998828EAC
6999BFEDCEFABC3B6E86AEB07D469D54B8643EEE8CF38C522DAF7759789FCAF22EB2255DFD35CC44A458467302809C24
8352D93037FFA35A1F412CFB57898BDFC2C7FF1D82041DD90CBA62C24DC35CAABD8EAC2B5C328F1E779151AAD9153C0B
FF640D3371A65B4673626400E424F6A7E97B4A46EBFF5876C35CFC10699821B934CC65187CE871854BE651A79D6D946A
F0139B8503CCF335CCC4926219CD89AC03909390F774E274BFB955799BF1920E89C5BFB11774504C44D4E60E513AE6E3
6E878C2A5C328F38E9743F5D55BD7C67A45E9EDF110D73A963EB229B29BA653427B20E404E42EB2D6AEACFEC95D73097
34441A6648AEDABF635A6AE6E3CE071C51B864DEEFB853FD8055E9DA9D9D7A79BE869978885E2CD73734E996D19C00C8
4948EB662DF5BBD4E80D73BBC71A35CC9050116B25F3B1EFB091854BE63D0E3FDE8F59C57F5633552FCF4F42C34CBA39
B40C36140072125ACBD401E6F9356998E7C7BEB41774D0EE4CAFD9CC4DD97CEC337478E192B9FFC147F9498B28B7FF8C
E7D3FEAB2AFECF6126952216CBBA6534270072524E929D9D5A164EE994D73097FA68560D33244E498F5B371FC3CC9BF7
DD76430F295C326FBFCF88EFBEFBCE8F5CF40B5906EBE5F91A666ACEA165B0A10090935058069FE258FB86399B0D0064
241C2C3CDAB4EB4147162E997B0DDEFFCBAFBEF25317E50735B317170D33B53165FA8CFA8626DD32646A01032027A10C
A53EF921657B730D33D0661ED6F8765B5A171E430E3BBE70C9BCF1CE7BBEF9EEFB7EF6CAFE41CDEC2094D7FD96741DEF
A8DF6820261C5A862C2F6000E4249424B34F712CAF612EB531D630438274EC6F73A478E131ECD8530A97CCABF5D8FE89
E7A6FB092CE30735CBA76A6BD03067ED2D2A68E1D03258C000C8492849EE562BB3FFF0CC3E87196891F1374DABB6434E
3AAD70C9BCEC3A9BDEF3E05403D5AE0C3ED5AA982B6CF10D7349A5B1D1CE1A8796A10C9A13003909B97D69D6CE416998
813667AB7AB94A8E1A33AE70C91C7C5CD530D14015907B1FC435A5BC86B9A4472B77D453D9A93DDD32944D73022027C9
B88CFF9AAD8619683DB5D5CB5575E2D80BDB2D99CFBCF42A03D5EE0F6A316F51177C4EBA2FEE65BF075FF17F30E22332
EC26E24FB10CD1C93A00394996798AA38619683DAF8D46B58DB9E01FED96CC47D78F3350852FDC457E72BAAF3B6537CC
455ECAA39F18B79B8833DD32684E00E4245470979AD9DA337AC35CCC1EBFECFD2F50FB3C34496BE3FCAB6E68B764DEFB
A8D1062A576EDBD9EE0F6ACB4F758A2FF1253DECA2F09FADDE62C96E2286A217CBF50D4DBA659075007212D4CBCDCA6B
984B7A87A022F7B0820EE29087EAE55ABAE9AEFB175AB553E192799B3D0F7AEF838F8C55DECF6A70F16A2C28238F0ECE
BB1C973A7F8B19A5E8C3E8221B2B0E2D830D05809C844AF1264125ED2E0B0F60BB5B5A0D33C499DB6D1DEBC1698FFF62
C39E854BE6D57A6CFFD0E34F1BABBC36B57869BD6F9257A497F104F5BC0B74EB04A8C8C3735C64E32062B1AC5B065907
2027A1B5E2DF24283B4351BD8639FEFD95A023B3DC6E8B83E9AFBCB64EEF5DDB7D62C695B7DE99F181D230371FCFAEFB
415EB79C5714376BFEFC924635F70FE6FE1551C6D045B66339B40C3614007212AABA452DE6A84FF3562BC5A351F6EF11
17FF667F1A66887F02B43B378BA9AACCC7B27DF0F1273D07EDD76EC97CE2D80B33FBB39A7767339B0D738156394C314B
9DE6D63ACA5730A96368CAF419F50D4DBA65B0A1009093500D253D6EB4F993A3EFAD92B25D2DA9FECD1DC9C24314F347
64083AE461F1EFD759839B44599E8F03478C6AB7641E34F2D879F3BECBE08F6BCBD9DD32A469CE96AAF876BD7984BBE5
A8D47D2517D91A7368196C2800E424545B496D6AA54EEF2462404AED8E8A7C5046221EF12AE8C8A0326EB7D5E62C68C6
E7E388934F6FB764DE68C721FF79ED753FC3402E8796C18602404E426DE4352A05E43EE130DD4FC988F29BB0C53C2823
FE0798051D19CFC396938A05D4F26903E6E369175DD16EC9BC72E7CD6F7FE0413FC6C07C8796C18602404E426D95F7FC
C61437CCAD9FBB586A7D54F8211B8938C02CE8C8A038BF639AF918B8F9EE07965ABB6BBB3DF3E87117F961862CD32D83
0D05809C841A2BBB51494DC39C77423BEC0D8382FF5EFC7BD0E73D2B23EC3DE8D37D0E1C1227CEEF9866E1D1ECA9175E
5AA7F7AEED96CC3BEC3BF2DD591FFA91864C512C83E604404E424729BB51494D3B5A5EC71EF12B37FF02BE1F3F888FBC
1B43150F040B8F4AF9FCCBD93B1D7078BB25F3AFBB6E7D57E3143FD89005BA65D09C00C849E85865BF077D6A0AD2E6D3
CBA52A75845BDE833E4D430769525E14941A08161E9572E498B1ED96CCC1C7A9E75EEA671BD22A7AB15CDFD0A45B060B
1800390900904D175F774B3125F34E071CFEC1C79F182E48138796212934270072120020CE1E983AED371B6FDD6EC9BC
CA66DBDEFEC083860B922E62B11C7C38B40C35A639019093000031F7EEAC0FFB0E1B59CC61E6C34E3DC370613791500E
2D83AC039093000054CFE87117155332AFDF67C0BFFEFDB8E1C26E2229A64C9F51DFD014E5D0B26E19641D809C0400A0
18B73FF0E07F77D9B2989EF9E4732E365CD84DC49C43CB20EB00E42480A003A8B137DF7DBFCFD0E1C594CC9B0FDCF7D9
175F3662B8C8C64DF3A165DD32C83A00390920E8003ACAF1679E574CC91C7C9C33FE3AC3858B6C4C38B40CB20E404ECA
4940D001C4C4ADF74EFAF9063D8B2999B7D87DD8A34F3F6FC47091ED40BA6590750072524E02820E301FE3E6DD591F0E
187E749187998F3FF33C2386495D638A65907500C84940D001E663CC5D74EDCD3F5963C3624AE63F6FB3CBC4C6C9460C
93BA0674CB20EB8C0C809C04041D603E26C58BAFCED87A8F038B3CCC3CECD8533EFEF4338306D510BD58AE6F68D22D83
050C809C04107480F9587BF5175E5E64C9FCF30D7A5E7AFDAD460C2AC8A165B080B18001909380A01374603E26DD234F
3ED3A5DFE0227BE6CD76DD7BD2D469060DA2706819B08001909300808547CA9C70E6F94596CCC1C75E47D6CD78F36D83
06A5726819B080019093008085475A3DF6CCF35B0EDEBFC89279A1553B9D7CCEC5060D8A11B158D62D83050C809C9493
00808547529C3BFEFAA5D6EE5A64CFBC5A8FEDAFBE6DA24183300E2D03163000721200206B5E7FFBDDDD471C53FC4333
36DD75AF3B263D68DCA0C594E933EA1B9A74CB4018CD09809C040048BD1BEEBC7795CDB62DBE67DE66CF831A1F7ED4B8
91710E2D03C5487D73D2D8D818FD2B44FF22E5A9ABABF3230A721200808AF8FA9B39478D19577CC91C7CEC72E0118F3E
F3BCA12383BB09DD329085AC6B57636363B76EDD16586081E0FF2DFBCFE6A9FB41CDBEFF96BFB4A32A6E60BE86190020
5D9E7BE9955D0E3CA2A49E79CF234E78E1E5570D1D59D84D289681D8665DF331E0BABABA6E3FA8FB3F552A4E83AF9CD7
0C97F4ADB6D92DB7AE9AABFDD2B434CC35FB1B8174AC09010068D784FB9B3A6DB75B493DF3C011A3263FFA84A123ADBB
09DD3210DBAC2BA6B0AD60779AF7D705FFBBD4123BEF1B6BF9E3CD0D79ED7BE6BCBF54CF0CD68400820EA052C6FEE39A
E5D6DDACA49E79AB2107DC765FA3A1233517D9E8C5727D4393171A645D95B2AECD6EB9DBFF69FDDFA3FF8DD1CBD8DCAF
D0E61FEF909239EFEFADC85801695A1302083A80B27DF0F127C347D7975432071F9DB6DBEDB21B1A8C1E89BEC83AB40C
C43CEBF23AE436CF12477996456B791D6C79CFDF28A6C5CD7B7845CDFA5E2533581302083A802A79F6C597871C767CA9
3DF3EF36ED7DDA45577CFEE56C0348822EB21539B4AC5B06AA9A75AD8F2E1738E59B5732975D9C56A47DCDAD8E0B7F91
768F3A57899219AC0901041D988F54CF234F3EB3E3FE8795DA332FB576D723C78C9DF1E6DB0690984F6A87968144645D
DEF9DE620AD8BC92B98CB3C7795FA1ECA128BE619E9F73DAB9C60F465632838D1E80A003F391AA7A60EAB45E83F72FB5
670E3EF63E6AF4A34F3F6F0089DBA48E582CEB96811A675D19CF288EFED089880575795FAAE59C76ED5FB2DC23E2DEF8
0F6CF400041D988F54C36DF73576DD698F327AE61DF61D79CF83530D2071E0D03290B8054CD96FB497F7548D925AE2CA
3EB022F73B89F309E10A96EA808D1E80A003F39130B7DDD7B8D59003CAE899D7DB76C019975CF9CEAC0F8C21B53765FA
8CFA8626DD3290B8054C94272A97DD3057EAF9182D3AEA5DFCA27C9F9E9501367A00820ECC47AA6ACAA34FEE3EE29832
7AE6E06397038F68B8F79FC6901AFDAC3AB40C24790153C6F3315A947DF8B9ECBFB11ADF4C8DE5D6F28E31838D1E0060
E141B54D7FE5B5E1A3EB175EAD73193DF36F36DE7AD469673FFBE2CB86912A71681948FA0226622B5BDEF9E78A1F606E
11E5A91D35E31833D8E80100161ED4DEAC8F3E3EE9EC8B7EBE41CFF28E34F7D86D9F4BAFBFF5ABAFBF3192548443CB40
6A16307907984BAD64A337CC953D699CF7AC8CD896CC9EC60C367A008085071DE2FBEFBF3FFFAA1BD6DE6AE7F27AE625
D6EAB2CFA8931A1F7ED4485236DD3290A6054CF4C74A94D73057B55F8DF258E99A49CAFB12829C040020ADAEBFE3DE6E
FD8796D733071F7FDA7AE7FA0B2F7FE39DF78C2445522C033114BD3989F204E66665D4B9D57B4446636363DE5332625B
E1E69DB5F6C30CB1CD490000D2ED81A9D3763AE0F0B27BE6E0A3CFD0E1975E7FEBAC8F3E369884D12D03B115B139C96B
7ACBEB39CB6898734BE04A3D2223AF5BEEF68338BFEB5F5EC3EC411910CF9C040020239E7AE1A511279FBEE2FADDA354
CDBDF73AF8E2EB6E79EF838F8CA7DD44B3E8C5727D43936E198873D6457F44C6FC56EFAC57CC17A978F19BFB0F09BE9F
96B636E6EFFA17FF4779809C0400206BAEBDFDEE3E438747E999838FAD861C70C1D537BEFDFE2CE399D9DD8443CB40CD
0471D1FA23EC93EB1B9A5A3E82A8299054CD5FB9DDDC8BFE888C32BE485EAD1DBDF52D70223AEF9C70DC5A5C8F628698
AF090100C8AC575F7FF3E4732E5EA367DF8855F39683F73FEFCA1BDE7CF77D439A91DD44C46259B70CE99357FCB6B4BB
0532A4F54781CF0F4B92923EBFEC86B9228FC898DFAA616EB731AEEC4398DB7DE0469C9F9591F7BD99711093352180A0
03A0C5BDFF7A688FC38F5F78B5CE11ABE6CD07EE7BF615D7CE7CEB1D439AD68BAC43CB105BC59FE9CDFD2830D35B4FF6
521BDDB87D7EA51AE6F2CED0E61D122EA626AD60C35CCCF39CE37C8CB9E2C7B901C50B20E8041D40357CFAF917175C7D
E3C63BEF19B1670E3EBA0FD867EC3FAE79E9B59946354117D925BA0F5874BD9E3FD96887257B0CFC69EF7DA73DF5ECC7
9F7ED6FC39CD27121D5A86E2B5F94887621AE05A9EE9CDDAE7B7DB3087BD401539DC5BC641E84A3D1AA2F8A63AB64F63
F6667FA07801107400C9F2F8B32F1C76EA193FFBCB16D1ABE6D57A6CBFEF31275F77C73DEFCEFAD0C0C6FC22BBE8FABD
5ABF822B74EAB1E2C67D16EFDA77C92D062DBDDD01CBF53F5AB14C3A54E4A90E1ADD647D7ED90D73A94FB7685319EF15
589113C525FDBDB13DC69CF78DC5EA091EA07801107480F9480137DC79EF0EFB8E8CDE33377FFC6587DD8F1C33F6AEC6
295F7F33C7D8C670522FBCD6C6C5BC8E0BADBEC1A2EBF55C7CD39D7FBAED7ECB0FAAD32D539E92DEA92DB34F75F0F995
FDFCF21AE6329E6ED1A632DE2BB022656F497FA98619AC096DF4004107988F54C3C79F7E76C5CD13FA0EAB58D5BCD0AA
9DFEB2C3EEFB1E73F2F957DD30F5B1A7667FF5B541EEF049BDFCE013CB7B3517EFDA6FF94127289653A0BCA73AD4F89D
DA7CBECF2F72F95152C3DCF2F96DFEC057E46DFECAABA9A397BD110F4EC7E73DF534CC60A30720E8C07C241D3EFBE2CB
2B6FBD73A7030E5F68D54E956A9B9B3FFEB4F5CE830F3DEECC4BAF9A3475DAFB1F7E64A82BE58BD9B3DF78E7BD675F7C
F95FFF7E7CC2FD4DE36FB963DCE5D79C74F64587FFF5CC61C79E3270C4A8ED861ED2ADFFD0F5FB0C58B5FB7665BF7C8B
FCA9DB0A434ED42D575B9B8F7428E9A90E1ACBCC7E7E1997EFB47E7EF486B9224D6F915F247AA71AB1D68E4FC35C91D1
006CF4004127E8C07C243EBE983DFBEADB26EE72E0118BACDEB9B25573F3C772EB6ED669BBDD82AF7FD49871175F77CB
FD531E79ED8DB70C7BB339DF7EFBDE071FBDF8EA8C871E7F7A62E3E4E0853877FCF5279D7DD1C893FFBEC7E1C76F37F4
904D76D973AD2D77FCC5863D17FBC306D57875DAFC08FEF6740F7B49EFD456FC531D34A21A579F5FFCE717FF1CECBCF9
D83C73CB5EC094710CB8B5321E9111FDD46E795F21B66FF6A761061B3D004107E623A934FBABAFAFBDFDEEFE071F5583
3273E1D53AFF618BBE5B0EDE7FF711C71C7ACA19F5175E7EC5CD13EE6A9CF2F8B32FBCF9EEFBDFCE9D9BD031FC62F6EC
B7DF9FF5D2AB331F7BE6F94953A7DD72CFA4CB6F9A70D66557D79D75C121279DB6C7E1C7F71D36B2FB807DD6DDB6FF6F
37D966E93F6D52B3D2B8D453E8B51FBA329EEAD0BA9EF25487B2E3DDE7FBFC5CAD6756813B2C65CCDF0E5CC0E435AE65
749B653F6723E2D9E9F2BA710D33D8E8D9E80100161E74886FE6CCB9E1CE7B078E18B5C45A5D3AAAE45CA9538FB5B7DA
B9EB4E7B6CBDC781BB1C78C4DE478D1E71D2E9279C79FEE9178FBFE8DA9BAF9970D71D931EFCE743FF9EF2E893FF7EFA
B9A75E78E9F9FFBCFAF28CD767BEF5CE3BB33EF8F0934F3FFBE2CBAFBF9913FC43667FF5F5E75FCEFEE4B3CF83FFF8DE
071FBDFDFEAC37DE796FE65B6FBFFAFA9B2FBD3AF385975F7DF6C597833FFEF8B32F4C7BEAD9871E7F7AF2A34F343DF2
D8A4A9D3EE9BFCF0C9E75C7CF3DD0F5C79EB9DC1DF78D665579F7ADEA5C7FEFDDC9127FF7DD8B1A70C3EF4B89D0E387C
AB21076CB4E390E0FBFCED26DB2CB7EE660BAFD6399E8D71A91F2B77DEBC408354D23BB5E53ED2211DCF7DF5F93EBF36
9F5FFC99DE76E763CBFCB580995F8586B9F8AF10B161CEFBCECBF84B3D87196CF46CF400000B0F6AEFDBB9736FBAEBFE
FD8F3B758D9E7DD3D19DFA28E6A35BFFA1D1E3C2E7FBFC347D7E9B8F7428A601AED9995E0B989A35CC6577B61AE6161A
66B0D103002C3CC8A65766BE71C975B7EC76C8A89F6FD053079BEE8F07A64E8B1E173EDFE7D7F2F32BF25407399F8505
4CC48639CA639C2336CC6574C5B12D7235CC60A30700008F3EFDFCE9178FEFBDD7C1CBAEB3A93E364D1F0BADB1D12DF7
4CAAC83EC5E767EAF34B7AA7B66C3ED581983427512AE2B29FC0DC2CB7DC8ED83017F9C78BFCC736363676FB3F6D7E4E
BB9F50AA8ABCDF2250A59C040080DA7BE9B599D7DF71EFA8D3CEDE66CF831C6F4E4087BC6AA7C5D7ECB2CC9F3759688D
2E0BAFB5C9227FEAB6E87A5BFE64C33E4B74EBFFD3DEFBAE306474A5F6293E3F894F7570AC97D4372779A7674B6A7A23
9E7F2EEF31172DCAF8B68BF91BF33E27EFDF95375C957AC7C0BC86393EEF3F0872120000E260E65BEF4CB8BF69F4B88B
FAED77E89A3DFB2DF6870D94BA152B8757FBCBC26B76FD63AF9D36D965CF3E43870F3EF4B8834FAC3FE1ACF3FF7EC9FF
BCFBE14F7BEFBBF476072CD377F8B23B1DBAEC2E472CD7FFA8E5068C5A7EE0B1CB0F3A7EF9C1755F7F3367CEB7DFCE9B
F75DD9FB0E9F5FBD06B8F833BDBA5F88DE9C94D730473F761BB1562DA3A02EFC0D379F4C6EDD1EB78C49CB37DC7C7A39
E2533E0A0C859F67885B4E020040DCCC7CEBED4953A75D7AFDADA34E3BBBFFC147FD6587DD975FAF9BBAB8E563C5F5BB
AFBEF90E1BF61BB4F51E07EE76C8A8034FF8DB71679C77E6A5572DB5F5D0A5FBECBF4CDF4396DDF9F0E5763B66F9C175
E5359C617F2ACB0D7049EFD496FB51A0D755F942529A93BCB2B498A637AF112DAF5FCD3B0F1CB1616EF78F176E715BBE
99967F4BEED76FF9E3C17F6CF98B72BF60C48639E20343806AE724000024C2AC8F3E7EEC99E727364EBEFCA609632EF8
C7C893FF3E70C4A82D761FB6F6563BAFD4A947A28F19FF7C839EAB6FBE43A7ED76EBB1DB3E7D878DDCF38813469C74FA
89632F5CB2E790FF2D8D77F89F63C6CBF53F6AF9DD8FAF41E39A9486B98CA73A78B336D09C94D19C94FAA08C888F5F0E
FB7B4B3D055D52415DCC89EBBC872AE7FE91967A39EF8F54EAE0B1873043CC731240D0019002DFCE9DFBE6BBEF3FFBE2
CB0F3DFEF43D0F4EBD71E27DFFB8F1B6B1FFB8E6E4732E3EE26F670D3BF694812346F5193A7C8BDD876DBAEB5E5DFA0D
5EBFCF80B5B7DA79F5CD77F8DDA6BD7FB951AF953AF558E6CF9B2CBE6697E063A9B5BB06FF7BF9F5BAADDC79F35F6CD8
F3571B6DF5DB4DB65965B33EC127AFB965BFE04FADD37BD7E08F6FD077D0A2EB6EB1E87A3D175D6FCB45D7EFB568A7AD
16EBBCF5629DB7596C836D871C76FC7EC79D7AE829671C77C6797F3DEFB2B1FFB8FA92EB6EB9FAB6894BF7D96FE9ED0F
5AA6DF886577397CB90147FF4F5D3C64740CCFF456AA616EF3910E9EEA0024714351FC13955B9F5E8EF2D0E08867778B
7C5046795576EE3331C2BEBD8A34CC119B7640F10220E80048CDE5299E4F757091052476315A97CC79D571EBC714477F
9E43C437FB6BFDBE7BADEBD9B29F175DCCF75691A1F01066484A4E02083A00A8F6E5A9D4A73AB4F948878A9FEC759105
2476915A3FFEA2E5046FEBFF18E5E8728B888F629EDF56C9DCF24E7CADBFF392BE7EBBBD74EE5F1DE5E071EEF7E90033
C43C2701041D603E4249D2F14807931AB080295EDD0F160857A96EB94545DEE4AEE2DF735E71DDEE5F1AA519F6880CB0
26041074603E02263540CAB2AEAE95CA16CB2D8AE9728BFF52ADBFE7F2BEED62DAE3DC6E3CCA387B4406581302083A30
1F01002C60CA965BD556A9C72E556EF11BF62D55A419768019E42480A003F31100C002268ADC63CCD1DF3DB022DA6D8F
2BF210660798414E02083A301F01002C60A28BD5BBDD15D37857E421CC0E30839C04002C3C00002C60A2CB7B1A73C73E
2BA3D4873057EF51CF809C04002C3C00002C608A91DBB876ECB3326AF010E6F8FC63414E1A1900C0C20300201D0B9898
F4AEB9E793DBFC84880F8EF6F865909300000000B1929AE6240E2573BBDF40D8DBFC05FFBBDDEF39AF5EEED8E781809C
D43003000000CC4F577392D7C1D6F819C5C53C1EB9CD33CC2D7FB040699C7B3A3AF8DFEA659093000080DD0480ACABB8
D62573CD7AE6521BE6E6AEB8DDD2B8636B73C09A100000B09B00C85AD6E5B5B2B5796471BB0F616EFD69851FA991F7AF
707419E424000060370120EB6AA6A5A1ADCDB9DFBA1CC57C63DD7E50A0346E39F0AC5B06390920E800C0451640D67588
C61F24F79BF7230A721240D001808B2C80AC03909300820E005C645D640159072027E52420E800F31130A901641D809C
041074603E02263580AC03909300820ECC4700000B1800390920E8C07C341F01000B1800390920E800F31100B0800190
93721200B0F00000B080019093008085070080050C809C04002C3C00002C6000E42400000040E2684E00E42400000040
7934270072120000B09B0090750072120000B09B0090750072120000B09B00907500721240D001808BAC8B2C20EB00E4
A49C04041D00E0220B20EB00E42480A003001759005907202701041D988F80490D20EB00E42480A003F3D17C04931A40
D601C849004107988F0080050C007212107480F90800600103202701041D988F0000163000721200B0F00000B0800190
930000161E0080050C809C0400B0F000002C602C6000E42400000040313427007212000000A03C9A130039090000D84D
00C83A0039090000D84D00C83A0039090000D84D00C83A00390920E800C045D6451690750072524E02820E0070910590
7500721240D001808B2C80AC03909300820ECC47C0A406907500721240D081F9683E82490D20EB00E42480A003CC4700
C00206003909083AC07C0400B08001909300820ECC4700000B18003909005878000058C000C84900000B0F00C0020640
4E02005878000016301630007212000000A0189A13003909000000501ECD09809C040000EC2600641D809C040000EC26
00641D809C040000EC2600641D809C04107400E022EB220BC83A0039292701410700B8C802C83A00390920E800C04516
40D601C849004107E623605203C83A00390920E8C07C341FC1A406907500721240D001E62300600103809C04041D603E
020058C000C849004107E6230080050C809C04002C3C00002C6000E4240080850700600103202701002C3C00000B180B
18003909000000500CCD09809C04000000288FE604404E020000761300B20E404E020000761300B20E404E0200007613
00B20E404E02083A007091759105641D809C949380A003005C6401641D809C04107400E0220B20EB00E42480A003F311
30A901641D809C041074603E9A8F605203C83A00390920E800F31100B08001404E02820E301F01002C6000E42480A003
F31100C00206404E0200161E0000163000721200C0C20300B0800190930000161E0080058C050C809C040000002886E6
04404E02000000944773022027010000BB090059072027010000BB090059072027010000BB09005907202701041D00B8
C8BAC802B20E404ECA4940D001002EB200B20E404E02083A00709105907500721240D081F90898D400B20E404E02083A
301FCD4730A901641D809C04107480F9080058C0002027014107988F0000163000721240D081F9080060010320270100
0B0F00000B180039090060E1010058C000C84900000B0F00C002C60206404E0200000014437302202701000000CAA339
019093000080DD0480AC039093000080DD0480AC039093000080DD0480AC03909300820E005C645D640159072027E524
20E800001759005907202701041D00B8C802C83A00390920E8C07C044C6A005907202701041D988FE62398D400B20E40
4E02083AC07C04002C6000909380A003CC4700000B1800390920E8C07C0400B080019093008085070080050C809C0400
B0F000002C6000E424008085070060016301032027010000008AA13901909300000000E5D19C00C8490000C06E0240D6
01C8490000C06E0240D601C8490000C06E0240D601C849004107002EB22EB280AC0390937212107400808B2C80AC0390
9300820E005C6401641D809C041074603E02263580AC03909300820ECC47F3114C6A005907202701041D603E02001630
00C84940D001E6230080050C809C041074603E020058C000C84900C0C20300C00206404E020058780000163000721200
C0C20300B080B18001909300000000C5D09C00C84900000080F2684E00E424000060370120EB00E424000060370120EB
00E424000060370120EB00E42480A003001759175940D601C8493909083A00C0451640D601C849004107002EB200B20E
404E02083A301F01931A40D601C849004107E6A3F908263580AC03909300820E301F01000B1800E42420E800F31100C0
0206404E02083A301F01002C6000E4240060E10100600103202701002C3C00000B180039090060E1010058C058C000C8
490000008062684E00E424000000407934270072120000B09B0090750072120000B09B0090750072120000B09B009075
00721240D001808BAC8B2C20EB00E4A49C04041D00E0220B20EB00E42480A003001759005907202701041D988F80490D
20EB00E42480A003F3D17C04931A40D601C849004107988F0080050C007212107480F90800600103202701041D988F00
00163000721200B0F00000B0800190930000161E0080050C809C0400B0F000002C602C6000E424000000403134270072
12000000A03C9A130039090000D84D00C83A0039090000D84D00C83A0039090000D84D00C83A00390920E800C045D645
1690750072524E02820E00709105907500721240D001808B2C80AC03909300820ECC47C0A406907500721240D081F968
3E82490D20EB00E42480A003CC4700C00206003909083AC07C0400B08001909300820ECC4700000B1800390900587800
0058C000C84900000B0F00C00206404E02005878000016301630007212000000A0189A13003909000000501ECD09809C
040000EC2600641D809C040000EC2600641D809C040000EC2600641D809C04107400E022EB220BC83A00392927014107
00B8C802C83A00390920E800C0451640D601C849004107E623605203C83A00390920E8C07C341FC1A406907500721240
D001E62300600103809C04041D603E020058C000C849004107E6230080050C809C04002C3C00002C6000E42400808507
00600103202701002C3C00000B180B18003909000000500CCD09809C04000000288FE604404E020000761300B20E404E
020000761300B20E404E020000761300B20E404E02083A007091759105641D809C949380A003005C6401641D809C0410
7400E0220B20EB00E42480A003F31130A901641D809C041074603E9A8F605203C83A00390920E8F87FEDD8310A844010
00C1FFFFCBD0DC4430F023971A6C3034C88154E10B84E99D5D308F00800506009D04840E308F00001618009D04103A30
8F00001618009D04002C1E00001618009D0400B07800001618009D0400B078000016180B0C804E020000004C783901D0
4900000080C6CB09804E0200006E13005A07A093000080DB0480D601E8240000E03601A075003A09207400E09075C802
5A07A0933A09081D00E09005D03A009D04103A0070C802681D804E02081D9847C05003681D804E02081D9847F308861A
40EB00741240E800F30800586000D04940E800F30800608101D04900A103F30800608101D04900C0E20100608101D049
00008B0700608101D04900008B07006081B1C000E824000000C0849713009D0400000068BC9C00E8240000E03601A075
003A090000B84D00681D804E0200006E13005A07A093004207000E59872CA075003AA99380D001000E5900AD03D04900
A10300872C80D601E82480D08179040C3580D601E82480D08179348F60A801B40E402701840E308F00800506009D0484
0E308F00001618009D04103A308F00001618009D04002C1E00001618009D0400B07800001618009D0400B07800001618
0B0C804E020000004C783901D04900000080C6CB09804E0200006E13005A07A093000080DB0480D601E8240000E03601
A075003A09207400E09075C8025A07A0933A09081D00E09005D03A009D04103A0070C802681D804E02081D9847C05003
681D804E02081D9847F308861A40EB00741240E800F30800586000D04940E800F30800608101D04900A103F308006081
01D04900C0E20100608101D04900008B0700608101D04900008B07006081B1C000E824000000C0849713009D04000000
68BC9C00E8240000E03601A075003A090000B84D00681D804E0200006E13005A07A093004207000E59872CA075003AA9
9380D001000E5900AD03D04900A10300872C80D601E82480D08179040C3580D601E82480D08179348F60A801B40E4027
01840E308F00800506009D04840E308F00001618009D04103A308F00001618009D04002C1E00001618009D0400B07800
001618009D0400B078000016180B0C804E020000004C783901D04900000080C6CB09804E0200006E13005A07A0930000
80DB0480D601E8240000E03601A075003A09207400E09075C8025A07A0933A09081D00E09005D03A009D04103A0070C8
02681D804E02081D9847C05003681D804E02081D9847F308861A40EB00741240E800F30800586000D04940E800F30800
608101D04900A103F30800608101D04900C0E20100608101D04900008B0700608101D04900008B07006081B1C000E824
000000C0849713009D0400000068BC9C00E8240000E03601A075003A090000B84D00681D804E0200006E13005A07A093
004207000E59872CA075003AA99380D001000E5900AD03D04900A10300872C80D601E82480D08179040C3580D601E824
80D08179348F60A801B40E402701840E308F00800506009D04840E308F00001618009D04103A308F00001618009D0400
2C1E00001618009D0400B07800001618009D0400B078000016180B0C804E020000004C783901D04900000080C6CB0980
4E0200006E13005A07A093000080DB0480D601E8240000E03601A075003A09207400E09075C8025A07A0933A09081D00
E09005D03A009D04103A0070C802681D804E02081D9847C05003681D804E02081D9847F308861A40EB00741240E800F3
0800586000D04940E800F30800608101D04900A103F30800608101D04900C0E20100608101D04900008B0700608101D0
4900008B07006081B1C000E824000000C0849713009D0400000068BC9C00E8240000E03601A075003A090000B84D0068
1D804E0200006E13005A07A093004207000E59872CA075003AA99380D001000E5900AD03D04900A10300872C80D601E8
2480D08179040C3580D601E82480D08179348F60A801B40E402701840E308F00800506009D04840E308F00001618009D
04103A308F00001618009D04002C1E00001618009D0400B07800001618009D0400B078000016180B0C804E020000004C
783901D04900000080C6CB09804E0200006E13005A07A093000080DB0480D601E8240000E03601A075003A09207400E0
9075C8025A07A0933A09081D00E09005D03A009D04103A0070C802681D804E02081D9847C05003681D804E02081D9847
F308861A40EB00741240E800F30800586000D04940E800F30800608101D04900A103F30800608101D04900C0E2010060
8101D04900008B0700608101D04900008B07006081B1C000E824000000C0849713009D0400000068BC9C00E8240000E0
3601A075003A090000B84D00681D804E0200006E13005A07A093004207000E59872CA075003AA99380D001000E5900AD
03D04900A10300872C80D601E82480D08179040C3580D601E82480D08179348F60A801B40E402701840E308F00800506
009D04840E308F00001618009D04103A308F00001618009D04002C1E00001618009D0400B07800001618009D0400B078
000016180B0C804E020000004C783901D04900000080C6CB09804E0200006E13005A07A093000080DB0480D601E82400
00E03601A075003A09207400E09075C8025A07A0933A09081D00E09005D03A009D04103A0070C802681D804E02081D98
47C05003681D804E02081D9847F308861A40EB00741240E800F30800586000D8CF7BF9F933808510308F000016180000
0BA18510CC23008005060000C0050D00C002030000E082E68206005860000000704103002C300000000000F0E4851900
00000080C60B33000000C01BBCBA005A0700000040E3D505D03A000000001AAF2E80D60100602104001CB2005A070080
8510001CB2005A0700808510001CB2005A0700C0DF6DC7B5FCFC197041030C3580D6010000B8A001861A40EB0000005C
D000002C30000000B8A0010016180000005CD000000B8C05060000C0050D00C002030000E082060060810100007041F3
6700000B0C00000000001FE385190000000080C60B33000000C01BBCBA005A0700000040E3D505D03A000000001AAF2E
80D601000000D06CC7B5FCFC19E04BBC300300000000D078610600000000A0F1C20C0000000040E385190000000080C6
0B3300000000008D176600000000001A2FCC0000000000345E98010000000068BC300300000000D078610600000000A0
F1C20C0000000040E385190000000080C60B330000000000CD7EDECBCF9F010000000000000000000000000000000000
000000000000000000000000000000000000000000000000000000000000000000000000000000000000000000000000
000000000000000000000000000000000000000000000000000000000000000000000000000000000000000000000000
000000000000000000000000000000000000000000000000000000000000000000000000000000000000000000000000
000000000000000000000000000000000000000000000000000000F8B61F5B29A813
      > EI Q}%
    \rule{0pt}{6.837696335cm}\hspace{10.000000000cm}%
  }%
}

\begin{document}

	\title[Herman's converse KAM mechanism revisited]{Herman's converse KAM mechanism revisited}
	
	\author{Yi Liu}
	\address{School of Mathematics and Statistics, Beijing Institute of Technology, Beijing 100081, China}
	\email{yiliu111@foxmail.com}
	
	\author{Lin Wang}
	\address{School of Mathematics and Statistics, Beijing Institute of Technology, Beijing 100081, China}
	\email{lwang@bit.edu.cn}

	\subjclass[2010]{Primary 37J40; Secondary 37E40}
	
	\keywords{Herman's converse KAM mechanism, Gevrey regularity, Invariant circles}

	\begin{abstract}
In his counterexample to the KAM theorem, Herman used a hyperbolic
perturbation together with a bump and asked whether the latter was
necessary \cite[p.~78, Remark~4.7.3]{H1}. Motivated by this question,
we compare destruction and persistence for a sine perturbation,
with and without a Gevrey bump. The combined perturbation destroys all
invariant circles with sufficiently small nonzero rotation numbers,
whereas the sine perturbation alone preserves a sequence of analytic
invariant circles in the same frequency range.
\end{abstract}

	\maketitle
	
	\vspace{2em}
	
	\tableofcontents
	\newpage
	\section{\sc Introduction}
	
	{
Throughout, an invariant circle means an invariant essential Jordan curve
of the cylinder. Birkhoff's graph theorem makes every such circle for
the exact positive twist maps considered here a Lipschitz graph
\cite{B20,H1}. Aubry--Mather theory gives ordered minimizing orbits at
every rotation number \cite{B,M3}. At an irrational rotation number,
the Peierls barrier detects whether these orbits fill an invariant graph.

For a fixed rotation number of constant type, Herman proved persistence
under sufficiently small $C^3$ exact twist perturbations of an integrable
twist map \cite{H33}. Here constant type means that the partial quotients
are bounded. For each $\varepsilon>0$ sufficiently small, a perturbation
arbitrarily small in $C^{3-\varepsilon}$ can destroy the circle at a
prescribed rotation number \cite{H1}. For rotation numbers admitting
faster rational approximation, destruction is possible in finer
topologies \cite{F,M4}.
}

	\subsection{Herman's converse KAM mechanism}
	{
In Herman's construction the first perturbation produces a hyperbolic orbit.
The second is a localized perturbation used to exclude invariant circles
with rotation numbers in an interval adjacent to zero. A finite covering
then relates this interval to a prescribed irrational rotation number.
For the sine family below, we quantify the effect of the second perturbation
by comparing the actions of a free minimizing heteroclinic configuration
and one constrained to pass through the bump.
}

	Subsequent work on the destruction of invariant circles for twist maps, and of invariant tori for Hamiltonian systems with multiple degrees of freedom, in the $C^r$ topology ($r \in [0, +\infty] \cup \{\omega\}$), uses related constructions (see, e.g., \cite{Bes,BF,CW,F,M4,W,W15,W22}).
	
	To distinguish the historical construction from our families below, we use the superscript $\mathrm H$. In \cite[Chapter II]{H1}, Herman defined the maps
	\[
	f_{n,a}^{\mathrm H}(x,y) = \left(x + y,\ y + \varphi_{\delta_n^{\mathrm H}}(x + y)\right), \quad \text{where } \delta_n^{\mathrm H} := \frac{1}{n^a},\ a > 1,
	\]
	\[
	F_{n,a}^{\mathrm H}(x,y) = \left(x + y,\ y + \varphi_{\delta_n^{\mathrm H}}(x + y) + \chi_n^{\mathrm H}(x + y)\right).
	\]
	
	In this construction, $a > 1$ is a parameter.  The function $\varphi_{\delta_n^{\mathrm H}}$ is a $C^\infty$ hyperbolic perturbation whose $C^k$-norm is bounded by $C_k\delta_n^{\mathrm H}$ for every fixed $k \geq 0$, with $C_k$ independent of $n$. For the geometric estimates, Herman requires that $\varphi_{\delta_n^{\mathrm H}}$ be linear in a neighborhood of the hyperbolic fixed point. See Fig. \ref{fi31}.  The function \(\chi_n^{\mathrm H}\) is a \(C^\infty\) bump function symmetric about \(\frac12\) and with support of radius \({n^{-\frac{a}{2}}}\).

	\begin{figure}[htbp]
		\small \centering
		\IfFileExists{pic_x.png}{\includegraphics[width=7.5cm]{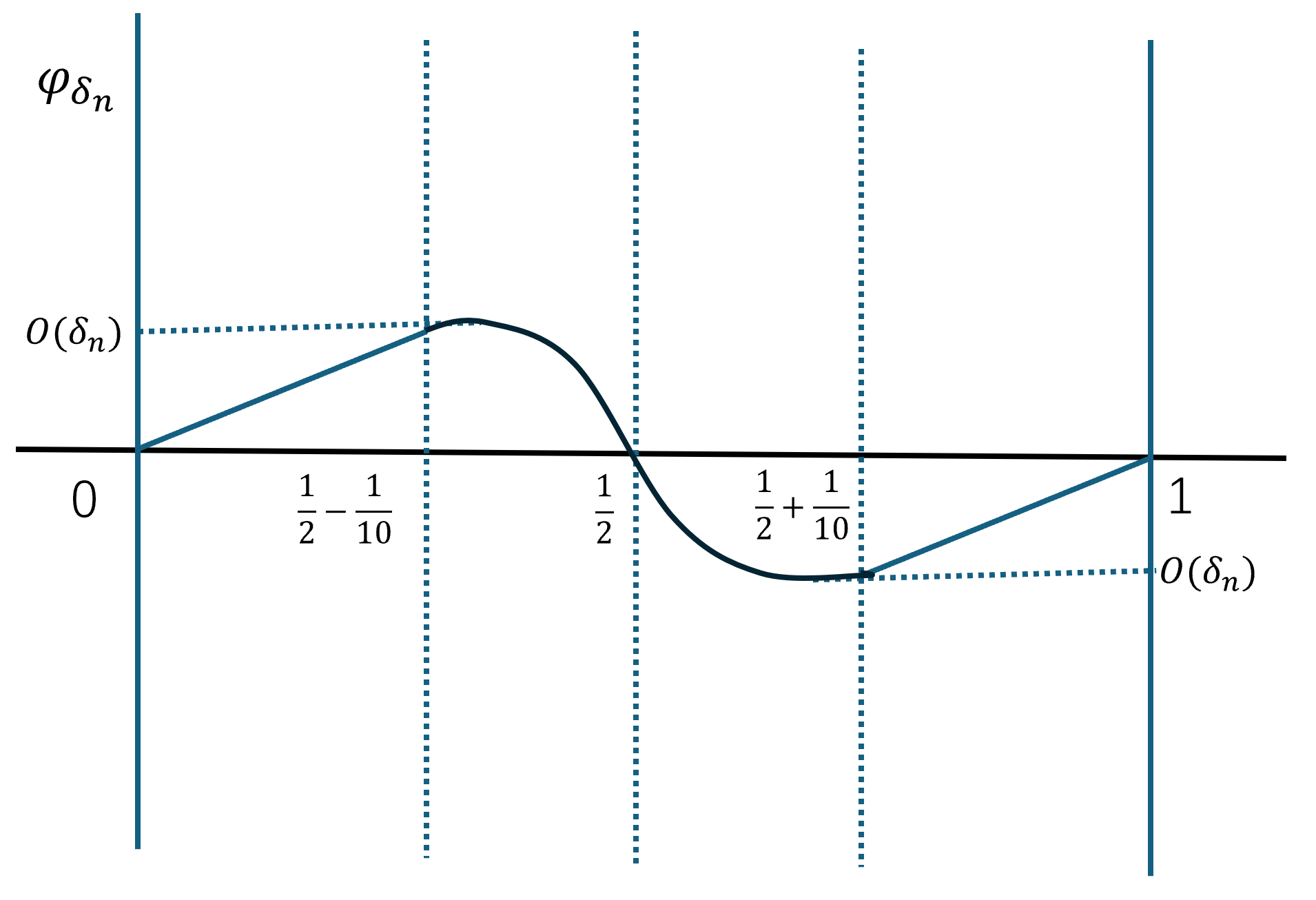}}{\resizebox{7.5cm}{!}{\HermanHyperbolicFigure}}
		\caption{The hyperbolic perturbation in Herman's counterexample}
		\label{fi31}
	\end{figure}

	\begin{Remark}
		In our notation, Herman's parameter is $\lambda_n^{\mathrm H}:=n^{-a/2}$. Thus the previously introduced parameter satisfies $\delta_n^{\mathrm H}=(\lambda_n^{\mathrm H})^2$.
	\end{Remark}
	
	\subsection{On the necessity of the second perturbation}
	
	Herman proved \cite[p.~71, Theorem 4.6]{H1}:
	
	\noindent\textbf{Theorem H.}\label{hermth}
	{\it For fixed $a>1$ and every sufficiently small $\varepsilon>0$, there exists an integer $N_0$ (depending on $\varepsilon$) such that for all $n \geq N_0$, the map $F_{n,a}^{\mathrm H}$ possesses no invariant circle with rotation number $\omega$ satisfying
		\begin{equation}\label{omhe}
			0 < |\omega| < n^{-\frac{a}{2} - \varepsilon}.
	\end{equation}}

	In \cite[p.~78, Remark~4.7.3]{H1}, Herman noted (using his original notation): \textbf{\textit{C'est parce que je ne sais pas si le diff\'eomorphisme \( f_{n,a} \) v\'erifie le Th\'eor\`eme H (avec \( \omega \neq 0 \)) qu'il a \'et\'e n\'ecessaire de perturber \( f_{n,a} \) en \( F_{n,a} \).}} (It is because I do not know whether the diffeomorphism \( f_{n,a} \) satisfies Theorem H (for \( \omega \neq 0 \)) that it was necessary to perturb \( f_{n,a} \) into \( F_{n,a} \).)
	
	\vspace{1ex}
	
	Using variational methods from Aubry--Mather theory, one can replace
the localized hyperbolic perturbation by a sine term
$\delta_n^{\mathrm H}\sin(2\pi x)$; see \cite{W}. Both models have hyperbolic fixed
points, but local linearity alone does not determine the global
separatrix geometry. For the standard map, transverse homoclinic
intersections and exponentially small separatrix splitting are known
for sufficiently small positive perturbation parameters
\cite{G99,La84,La}. Persistence close to the hyperbolic scale requires quantitative estimates.
	
	In this note, we still cannot answer Herman's question, but inspired by this problem, we replace \(\varphi_{\delta_n^{\mathrm H}}(x)\) by \(\delta_n^{\mathrm H} \sin 2\pi x\) and explore both the existence and the breakdown of invariant circles in the range of Theorem H. The results below compare destruction and persistence at the following frequencies:

\begin{itemize}
\item[(i)] For every sufficiently large $j$, the map $F_{j,a}$ defined
below has no invariant circle with
$0<|\omega|<n_j^{-a/2-\varepsilon}$
(Theorem~\ref{Mthm1}).
\item[(ii)] Along $j=q_m$, put $N_m:=n_{q_m}$. The sine map $f_{q_m,a}$
preserves an invariant circle of rotation number $\omega_m$ satisfying
$0<|\omega_m|<N_m^{-a/2-\varepsilon}$
(Theorem~\ref{Mthm2}).
\end{itemize}
	
We use Aubry--Mather theory to prove destruction for $F_{q_m,a}$.
For persistence under $f_{q_m,a}$, we compare the map with an
integrable map to exponential accuracy and apply analytic iteration
on a rotating level close to its separatrix. The two conclusions
concern the same selected rotation numbers; see
Subsection~\ref{dirkk}.

\subsection{Statement of main results}
	
	Fix a Gevrey index $\alpha>1$ and write $\beta:=1/\alpha$. Consider the subsequence defined by
	\begin{equation}\label{n_m}
		n_m:= \left\lfloor \exp\left(\frac{d}{a} m^{\frac{1}{\alpha}}\right) \right\rfloor,
	\end{equation}where $\lfloor x\rfloor$ denotes the greatest integer less than or equal to $x$.  Define the perturbation size as
	\begin{equation}\label{bval}
		\delta_m :=\exp(-d m^{\frac{1}{\alpha}}),
	\end{equation}
where the letter $d>0$ is  a parameter.
	Write $\T:=\R/\Z$ for the circle. Now, consider the generating functions:
	\begin{equation}\label{h2}
		h_m(x, x'):= \frac{1}{2}(x - x')^2 + \delta_m (1 - \cos 2\pi x'),
	\end{equation}
	\begin{equation}\label{h}
		H_m(x, x'):= \frac{1}{2}(x - x')^2 + \delta_m (1 - \cos 2\pi x') + \xi_m(x').
	\end{equation}
	The map $F_{m,a}:\T\times\R\to\T\times\R$ is generated by $H_m$, while $f_{m,a}:\T\times\R\to\T\times\R$ is generated by $h_m$, where $\xi_m:\T\to\R$ denotes a bump function; the precise construction is given in (\ref{XI}) below.  Explicitly, the maps are given by:
	\begin{equation}\label{f}
		f_{m,a}(x, y) = \left(x + y,\ y + 2\pi\delta_m \sin 2\pi(x + y)\right),
	\end{equation}
	\begin{equation}\label{Ff}
		F_{m,a}(x, y) = \left(x + y,\ y + 2\pi\delta_m \sin 2\pi (x + y) + \chi_m(x + y)\right),
	\end{equation}
	where $\chi_m$ denotes the derivative of $\xi_m$.
	Let us also define the integrable twist map
	\begin{equation}\label{intst}
		T(x, y) = (x + y, y).
	\end{equation}
	
	\subsubsection{\bf On the destruction}
	Aubry--Mather theory and estimates for the Peierls barrier give the following result:
	
	\begin{Theorem}\label{Mthm1}
		Fix $a > 1$, $d > 0$, $0 < \varepsilon < \min\{1,a/2\}$ and
\[\alpha > 1 + \frac{a}{\varepsilon}.\]
 {Choose the bump parameters as in \eqref{bump-lambdas}
and center the bump as in \eqref{taumo}. There exists $M_0$ such that}
for every $m \geq M_0$ and every rotation number $\omega$ satisfying
		\begin{equation}\label{omb}
			0 < |\omega| < n_m^{-\frac{a}{2} - \varepsilon},
		\end{equation}
		where $n_m$ is defined in \eqref{n_m}, the map $F_{m,a}$ admits no invariant circle  with rotation number $\omega$. Moreover, for any  $L \in (0, d)$,
		\[
		\|F_{m,a} - T\|_{\alpha, L} \to 0 \quad \text{as } m \to \infty.
		\]
	\end{Theorem}
	In Theorem~\ref{Mthm1}, the norm $\|\cdot\|_{\alpha, L}$ is  defined in (\ref{gevr}) below. Bounemoura and F\'{e}joz \cite[Theorem D]{BF} studied an analogous problem for Hamiltonian systems with $\ell$ degrees of freedom using Bessi's variational approach \cite{Bes}. Comparably, Theorem~\ref{Mthm1} states the parameter dependence explicitly.

	\subsubsection{\bf On the persistence}\label{dirkk}
	We next prove persistence without the bump, for the frequencies in
\eqref{capN}. The proof uses an exponentially accurate integrable
comparison and a Newton iteration in area--angle coordinates.
The comparison error is estimated on a fixed neighborhood of the
separatrix, while the angular analyticity width and the small divisors
are controlled separately. This distinction is needed for the
non-Diophantine frequency constructed below.

Let \(\bar{\omega}\) be an irrational number. Let \(\{p_n / q_n\}_{n \in \mathbb{N}}\) be the sequence of convergents in the continued fraction expansion of \(\bar{\omega}\). Then $q_n$ satisfies the recurrence relation:
	\begin{equation}\label{recc}
		q_{n+1}=a_{n+1}q_n+q_{n-1}.
	\end{equation}
	We specify $\bar{\omega}$. Denote  the continued fraction expansion of $\bar{\omega}$ by \(\bar{\omega}=[a_0;a_1,\ldots,a_n,\ldots]\), where \(\{a_n\}_{n\in\N}\) is defined by
	\begin{equation}\label{barome}
		a_{n+1} := \left\lfloor \exp\!\bigl(b q_n^{\frac{1}{\alpha}} \bigr) \right\rfloor,
	\end{equation}
	where
	\[p_{-1}=1,\quad p_0=a_0=0,\quad q_{-1}=0,\quad q_0=1,\quad p_{n+1}=a_{n+1}p_n+p_{n-1}.\]
	This construction ensures that the tail sum
	\[
	\sum_{n=0}^{\infty} \frac{\ln q_{n+m+1}}{q_{n+m}}
	\]
	is dominated by its first term. Appendix~\ref{rerA0} proves the arithmetic properties of $\bar\omega$. For $x \in \mathbb{R}$, define the distance to the nearest integer by
	\begin{equation}\label{2.2}
		\|x\| = \inf_{y \in \mathbb{Z}} |x - y|.
	\end{equation}	We set
	\begin{equation}\label{oqqmx}
		\omega_m:=q_m \bar{\omega}-p_m,
	\end{equation}
	which implies
	\begin{equation}\label{oqqm}
		|\omega_m| = \|q_m \bar{\omega}\|.
	\end{equation}
	We take $b>a>1$, $0<\varepsilon<\min\{1,a/2\}$ and choose $d$ so that
\begin{equation}\label{hc-parameter-window}
 b<d<\frac{2ab}{a+2\varepsilon}.
\end{equation}
This interval is nonempty because $\varepsilon<a/2$.
For example, one may take the midpoint of its two endpoints.
The continued-fraction inequalities give
\[
 \frac1{q_{m+1}+q_m}<|\omega_m|<\frac1{q_{m+1}},
 \qquad
 \log q_{m+1}=bq_m^\beta+\log q_m+o(1).
\]
In particular,
\[
 \log|\omega_m|=-bq_m^\beta-\log q_m+o(1).
\]
Since $b>d(1/2+\varepsilon/a)$, for all sufficiently large $m$ we have
\[
 0<|\omega_m|<
 \exp\left[-d\left(\frac12+\frac\varepsilon a\right)q_m^\beta\right]
 \le N_m^{-a/2-\varepsilon},\qquad N_m:=n_{q_m}.
\]
The inequality involving $N_m$ uses its defining integer part.

Moreover, we establish the following result:
	
	\bthm\label{Mthm2}
	Fix $b>a > 1$, $0 < \varepsilon < \min\{1,a/2\}$ and
\[\alpha > 1 + \frac{a}{\varepsilon},\qquad b<d<\frac{2ab}{a+2\varepsilon}.\] {For $\bar\omega$ given by \eqref{barome}, let
$\omega_m=q_m\bar\omega-p_m$. For all sufficiently large $m$,}

\begin{equation}\label{capN}
		0 < |\omega_m| < N_m^{-\frac{a}{2} -\varepsilon},
	\end{equation}
	with
	\begin{equation}\label{capNx}
		N_m=n_{q_m}= \left\lfloor \exp\left(\frac{d}{a} q_m^{\frac{1}{\alpha}}\right) \right\rfloor.
	\end{equation}
	 Under these conditions, the map $f_{q_m,a}$ admits a real analytic invariant circle of rotation number $\omega_m$, and its restricted dynamics is real analytically conjugate to $R_{\omega_m}$.
	\ethm
	
	\begin{Remark}
By \eqref{n_m} and \eqref{capNx}, $N_m=n_{q_m}$.
Theorem~\ref{Mthm1} excludes invariant circles for $F_{q_m,a}$ at
all rotation numbers $0<|\omega|<N_m^{-a/2-\varepsilon}$.
Theorem~\ref{Mthm2} gives an analytic circle for $f_{q_m,a}$ at
one of these same rotation numbers, namely $\omega_m$.
Thus adding the bump destroys the selected circles of the sine map.
This statement concerns our sinusoidal family and does not identify
it with Herman's original hyperbolic perturbation.
\end{Remark}

The proof in Section~\ref{5} retains the prescribed frequency
$\bar\omega$. Its continued-fraction construction gives a
stretched-exponential lower bound for all small divisors, which
allows analytic iteration even though $\bar\omega$ is not Diophantine.

\subsubsection{\bf Finite covering and reduction of Theorem \ref{Mthm2}}

	The analysis for a given irrational rotation number can be reduced to the case of a sufficiently small rotation number, and vice versa, via a finite covering argument \cite[p.~78, 4.8]{H1}. This reduction is made precise in the following lemma.
	
	\begin{Lemma}\label{Herm}
		Let \( H_P \) be a generating function of the form
		\[
		H_P(x,x') = \frac12(x-x^\prime)^2 + P(x'),
		\]
		where $P\in C^2(\T)$ is $1$-periodic. Define the rescaled function
		\[
		Q(x) = q^{-2} P(qx), \quad q \in \mathbb{N}.
		\]
		Then the exact area-preserving monotone twist map generated by
		\[
		H_Q(x,x') = \frac12(x-x^\prime)^2 + Q(x')
		\]
		admits an invariant circle with rotation number \( \omega \in \mathbb{R} \setminus \mathbb{Q} \)
		if and only if the exact area-preserving monotone twist map generated by \( H_P \)
		admits an invariant circle with rotation number \( q\omega - p \),
        {where $p\in\mathbb Z$ is any fixed integer}.
	\end{Lemma}
%

	Define:
	\[
	\bar{h}_m(x, x'):= \tfrac{1}{2}(x - x')^2 + \frac{\delta_{q_m}}{q_m^2}\bigl(1 - \cos(2\pi q_m x')\bigr),
	\]
	\begin{equation}\label{f11w}
		\bar{f}_{m,a}(x,y) := \left(x + y,\ y + \frac{2\pi\delta_{q_m}}{q_m} \sin(2\pi q_m(x + y)) \right),
	\end{equation}
where $\delta_{q_m}$ is defined by (\ref{bval}), i.e.
\begin{equation}\label{xx}
		\delta_{q_m} =\exp(-d q_m^{\frac{1}{\alpha}}).
	\end{equation}
	It follows that  $\bar{f}_{m,a} $ is generated by $\bar{h}_m$. 		  By applying Lemma~\ref{Herm}, $\bar{f}_{m,a} $ admits an invariant circle with rotation number $\bar{\omega}$ if and only if ${f}_{q_m,a}$ defined by (\ref{f}) admits an invariant circle with rotation number $q_m\bar{\omega}-p_m$. Moreover,
	Theorem~\ref{Mthm2}  can be reduced to establishing the following result.
	
	\begin{Theorem}\label{Mthm21}
		Fix $b>a > 1$, $0 < \varepsilon < \min\{1,a/2\}$ and
\[\alpha > 1 + \frac{a}{\varepsilon},\qquad b<d<\frac{2ab}{a+2\varepsilon}.\] There exists an integer $M_1$ such that for all $m \geq M_1$,  the map $\bar{f}_{m,a}$ admits a real analytic invariant circle with rotation number $\bar{\omega}$, and its dynamics on the circle is real analytically conjugate to $R_{\bar\omega}$.
	\end{Theorem}

	\subsubsection[Persistence and destruction at a Brjuno rotation number]{Persistence and destruction at a Brjuno rotation number}

	{
For small analytic perturbations of integrable exact twist maps,
the Brjuno condition suffices for persistence of an analytic invariant
circle with the prescribed rotation number
\cite[Theorem~1.1]{Gentile}. See also R\"ussmann's work on
near-integrable Hamiltonian systems \cite{R3}.
} {For an irrational $\omega$, let
$\widehat p_n/\widehat q_n$ denote its continued-fraction convergents.} The unhatted convergents $p_n/q_n$ remain reserved for the fixed frequency $\bar\omega$. The Diophantine condition can be expressed in terms of this expansion as
	\begin{equation}\label{dionu}
		\ln \widehat q_{n+1} \leq C \ln \widehat q_n, \quad n \geq 2,
	\end{equation}
	where $C > 0$  is a constant independent of $n$.
	
	{In terms of the convergents, the \textit{Brjuno condition} reads}
	\[
	B_1(\omega):=\sum_{n \geq 0} \frac{\ln \widehat q_{n+1}}{\widehat q_n} < +\infty{.}
	\]
	For Gevrey-$\alpha$ systems, it is natural to introduce another arithmetic condition on the rotation number. For $\alpha \geq 1$, the number $\omega$ is called an $\alpha$-Brjuno--R\"{u}ssmann number if it satisfies the condition
	\begin{equation}\label{brcon}
		\sum_{n \geq 0} \frac{\ln \widehat q_{n+1}}{\widehat q_n^{1/\alpha}} < +\infty.
	\end{equation}
	In particular, when $\alpha = 1$, $\omega$ is referred to as a Brjuno number. We denote by $\mathcal{B}$ the set of Brjuno numbers, and by $\mathcal{BR}_\alpha$ the set of $\alpha$-Brjuno--R\"{u}ssmann numbers.

	\begingroup
For the rescaled map with a bump, use the choice
$\lambda_j=\lambda_*j^{1/(\alpha-1)}$ in \eqref{bump-lambdas}.
With $q=q_m$ put
\begin{equation}\label{F11w}
 \bar F_{m,a}(x,y):=\bar f_{m,a}(x,y)
       +\bigl(0,q^{-1}\xi_q'(q(x+y))\bigr).
\end{equation}
Its generating potential is
$q^{-2}[\delta_q(1-\cos(2\pi qx))+\xi_q(qx)]$.

\begin{Corollary}\label{Mthm3}
Fix $b>a>1$, $0<\varepsilon<\min\{1,a/2\}$,
$\alpha>1+a/\varepsilon$ and $b<d<2ab/(a+2\varepsilon)$.
The frequency $\bar\omega$ in \eqref{barome} belongs to
$\mathcal B\setminus\mathcal{BR}_\alpha$.
For every sufficiently large $m$, $\bar f_{m,a}$ has a real analytic
invariant circle on which its dynamics is real analytically conjugate
to $R_{\bar\omega}$, and $\bar F_{m,a}$ has no invariant circle
of rotation number $\bar\omega$. For every fixed $0<L\leq\bar d$, where
\begin{equation}\label{barbb}
 \bar d:=\frac{d}{\alpha(2\pi)^{1/\alpha}},
\end{equation}
we have
\begin{equation}\label{barxx}
 \|\bar f_{m,a}-T\|_{\alpha,L}
 +\|\bar F_{m,a}-T\|_{\alpha,L}\longrightarrow0 .
\end{equation}
\end{Corollary}
\endgroup

{The arithmetic assertion follows from Appendix~\ref{rerA0}.
Theorem~\ref{Mthm21} proves persistence. For destruction, an invariant
circle of $\bar F_{m,a}$ at $\bar\omega$ would give one of $F_{q_m,a}$
at $\omega_m=q_m\bar\omega-p_m$ by Lemma~\ref{Herm} with $p=p_m$.
This contradicts Theorem~\ref{Mthm1} in view of \eqref{capN}.}

	Let
\[
 P_m(x):=\frac{2\pi\delta_{q_m}}{q_m}\sin(2\pi q_mx),
 \qquad s_m:=L(2\pi q_m)^{1/\alpha}.
\]
By the norm convention in \eqref{gevr},
\[
 \|\bar f_{m,a}-T\|_{\alpha,L}
 =\|P_m\|_{\alpha,L}
 =\frac{2\pi}{q_m}e^{-dq_m^{1/\alpha}}
   \sum_{k\geq0}\left(\frac{s_m^k}{k!}\right)^\alpha.
\]
To estimate this series, put $\mathsf p_k(t):=e^{-t}t^k/k!$.
Stirling's inequalities imply
$\max_{k\geq0}\mathsf p_k(t)\leq C(1+t)^{-1/2}$, while
$\sum_{k\geq0}\mathsf p_k(t)=1$.  Consequently,
\[
 \sum_{k\geq0}\left(\frac{t^k}{k!}\right)^\alpha
 =e^{\alpha t}\sum_{k\geq0}\mathsf p_k(t)^\alpha
 \leq C_\alpha e^{\alpha t}(1+t)^{-(\alpha-1)/2}.
\]
For every fixed $L>0$, it follows that
\[
 \|\bar f_{m,a}-T\|_{\alpha,L}
 \leq\frac{C_{\alpha,L}}{q_m}
 \frac{\exp\left((\alpha L(2\pi)^{1/\alpha}-d)
                         q_m^{1/\alpha}\right)}
      {(1+q_m^{1/\alpha})^{(\alpha-1)/2}}.
\]
{The sine term therefore tends to zero also at $L=\bar d$.
For the second map, \eqref{covered-bump-norm} below gives
$\|q_m^{-1}\xi'_{q_m}(q_m\,\cdot)\|_{\alpha,L}
\le C_Lq_m^{-2}\delta_{q_m}$ at all these widths, since $\bar d<4d$.
The triangle inequality proves \eqref{barxx}. For $L>\bar d$, the term
$k=\lfloor s_m\rfloor$ and Stirling's formula give
\[
 \log\|P_m\|_{\alpha,L}
 =\bigl(\alpha L(2\pi)^{1/\alpha}-d\bigr)q_m^{1/\alpha}
   +O(\log q_m)\longrightarrow+\infty .
\]
In fact $\|\bar F_{m,a}-T\|_{\alpha,L}\geq\|P_m\|_{\alpha,L}$.
For large $m$ the bump vanishes in neighborhoods of the phases
$q_m x=0,1/4\pmod1$, where the odd and even derivatives of the sine,
respectively, attain their supremum in absolute value.}
For comparison, if $1\leq\alpha'<\alpha$ and $L>0$, retaining the
single term with $k=\lfloor L(2\pi q_m)^{1/\alpha'}\rfloor$
and applying Stirling's inequalities gives
\[
 \log\|P_m\|_{\alpha',L}
 \geq\alpha'L(2\pi q_m)^{1/\alpha'}
       -dq_m^{1/\alpha}-O(\log q_m)\longrightarrow+\infty.
\]
These estimates distinguish the Gevrey index $\alpha$ from every smaller
index $\alpha'<\alpha$.

The Gevrey KAM theorem of Bounemoura and F\'ejoz
\cite[Theorem F]{BF} applies to each fixed frequency in
$\mathcal{BR}_\alpha$ once the perturbation is sufficiently small.
Corollary~\ref{Mthm3} gives persistence at a frequency outside this set
for the particular sequence $\bar f_{m,a}$.
\begin{Remark}
{Both sequences approach the same integrable twist map at
the same fixed width. At the selected frequency outside
$\mathcal{BR}_\alpha$, the pure sine sequence gives persistence and
the additional localized perturbation gives destruction.}
\end{Remark}

\vspace{1em}
	\noindent{\bf Organization of this note.}
Section~\ref{2} constructs the Gevrey bump, including the norm estimate
for its derivative. Section~\ref{3} proves Theorem~\ref{Mthm1} by
finite-action comparisons and the modulus of continuity of the
Peierls barrier. Section~\ref{5} constructs the integrable comparison,
estimates its area--angle coordinates, and proves
Theorems~\ref{Mthm2} and~\ref{Mthm21} by analytic iteration.
Appendix~\ref{rerA0} establishes the arithmetic estimates used in
this iteration and in Corollary~\ref{Mthm3}.

\vspace{2em}
\noindent\textbf{Acknowledgement.}
	This work was partially supported by the National Natural Science Foundation of China (Grant No. 12122109) and the Beijing Natural Science Foundation (Grant No. QY25241). The authors are also grateful to Professor Giovanni Forni for his suggestions.
	
	\vspace{1em}

\section{\sc Construction of the bump}\label{2}
\subsection{Gevrey-$\alpha$ function}
Fix $\alpha>1$. For $L>0$, set
\begin{equation}\label{gevr}
 \|\phi\|_{\alpha,L}:=\sum_{k=0}^{\infty}
 \frac{L^{k\alpha}}{(k!)^\alpha}\|\partial^k\phi\|_{C^0(\T)}.
\end{equation}
For a bounded smooth vector-valued function $U$ on $\T\times\R$,
we use the compatible total-order norm
\[
 \|U\|_{\alpha,L}:=
 \sum_{k\geq0}\frac{L^{k\alpha}}{(k!)^\alpha}
 \max_{i+j=k}\|\partial_x^i\partial_y^jU\|_{C^0},
\]
where the norm of a vector is the maximum of the absolute values
of its components.  In particular,
\[
 \|(0,P(x+y))\|_{\alpha,L}=\|P\|_{\alpha,L}.
\]
This convention fixes the Gevrey width used in all map estimates below.

We also define
\[
G^{\alpha,L}(\T):=\{\phi\in C^\infty(\T):\|\phi\|_{\alpha,L}<\infty\},
\qquad G^\alpha(\T):=\bigcup_{L>0}G^{\alpha,L}(\T).
\]
The same notation on $\R$ means that the supremum is taken on $\R$.
\begin{Definition}\label{gev}
A function on $\T$ is Gevrey-$\alpha$ if it belongs to $G^\alpha(\T)$.
\end{Definition}
The Leibniz rule and $\binom{k}{j}^{1-\alpha}\leq1$ give
\begin{equation}\label{leib}
 \|\phi\psi\|_{\alpha,L}\leq
 \|\phi\|_{\alpha,L}\|\psi\|_{\alpha,L}.
\end{equation}
We will also use the following derivative estimate, which requires a loss
of radius. If $0<L<R$, then
\begin{equation}\label{gev-derivative-loss}
 \|\phi'\|_{\alpha,L}
 \leq R^{-\alpha}\sup_{k\geq0}
 \left\{(k+1)^\alpha(L/R)^{k\alpha}\right\}
 \|\phi\|_{\alpha,R}.
\end{equation}
The supremum is finite. This follows by comparing the $k$th term on the
left with the $(k+1)$st term of the norm on the right.

For $\lambda>0$, define
\[
f_\lambda(x):=
\begin{cases}
0,&x\leq0,\\
\exp(-\lambda\sqrt2\,x^{-1/(\alpha-1)}),&x>0.
\end{cases}
\]
\begin{Lemma}\label{gealp}
Let $p:=1/(\alpha-1)$ and
$\sigma:=\frac\pi4\min\{1,1/p\}$. If
\[
\lambda>\frac{(2L^\alpha/\sin\sigma)^p}{p},
\]
then $f_\lambda\in G^{\alpha,L}(\R)$.
\end{Lemma}
\begin{proof}
Use the principal branch of $z^{-p}$ in
$\C\setminus(-\infty,0]$. For $x>0$, the disk with centre $x$ and
radius $x\sin\sigma$ lies in $|\arg z|\leq\sigma$; on this disk,
\[
\Re(z^{-p})\geq\frac{1}{\sqrt2|z|^p},\qquad |z|\leq2x.
\]
Cauchy's estimate gives, for $k\geq1$,
\[
 |f_\lambda^{(k)}(x)|\leq
 \frac{k!}{(x\sin\sigma)^k}\exp\left(-\frac{\lambda}{(2x)^p}\right).
\]
In particular, every derivative tends to zero as $x\downarrow0$,
so the extension by zero is $C^\infty$. Maximizing
$y^k e^{-\lambda y^p}$ gives
\begin{equation}\label{kd}
 \|f_\lambda^{(k)}\|_{C^0(\R)}
 \leq\left(\frac2{\sin\sigma}\right)^k
 \left(\frac{k}{\lambda pe}\right)^{k/p}k!.
\end{equation}
Since $k!\geq(k/e)^k$ and $\alpha-1=1/p$, the $k$th summand in
\eqref{gevr} is bounded by
\[
 \left[\frac{2L^\alpha}{\sin\sigma}(\lambda p)^{-1/p}\right]^k.
\]
The bracket is smaller than one, and the zeroth summand is at most one.
\end{proof}

\subsection{Construction of $\xi_m$}
With $\beta$ as above, set
\[
 \Delta_m:=\exp\left(-\frac d2m^\beta\right)=\sqrt{\delta_m}.
\]
\begingroup
Fix $\alpha>1+a/\varepsilon$ and put $p:=1/(\alpha-1)$.
Choose $\lambda_*$ satisfying the strict inequality in Lemma~\ref{gealp}
at radius $R:=4d$. In Theorem~\ref{Mthm1} one may choose any sequence
\begin{equation}\label{bump-lambdas}
 \lambda_*\leq\lambda_m\leq\lambda_*m^p.
\end{equation}
This includes the constant sequence $\lambda_m=\lambda_*$.
Corollary~\ref{Mthm3} uses the upper endpoint for its rescaled family.
By \eqref{kd}, $\sup_{\lambda\geq\lambda_*}\|f_\lambda\|_{\alpha,R}$
is finite. For $x\in[0,1]$ put
\[
 v_m(x):=\Delta_m^2
 f_{\lambda_m}\left(\tfrac18\Delta_m-\tfrac12+x\right)
 f_{\lambda_m}\left(\tfrac18\Delta_m+\tfrac12-x\right),
\]
\endgroup
and extend periodically. Both factors are flat at their zero endpoints,
so this is a smooth periodic function supported in
$[1/2-\Delta_m/8,1/2+\Delta_m/8]$ modulo $\Z$. Translation and reflection
preserve the Gevrey norm. Thus \eqref{leib} gives
\[
 \|v_m\|_{\alpha,R}\leq C\Delta_m^2\longrightarrow0.
\]
By strict convexity of $t\mapsto t^{-p}$, the maximum occurs at $1/2$,
and its exact value is
\begin{equation}\label{maxv}
 v_m(1/2)=\max_xv_m(x)
 =\Delta_m^2\exp\left(-{\lambda_m} 2^{\frac3{\alpha-1}+\frac32}
                     \Delta_m^{-\frac1{\alpha-1}}\right).
\end{equation}
For later use, the unmaximized Cauchy estimate in the preceding proof
also shows that, for every fixed $k,N\geq0$,
\begin{equation}\label{bump-fixed-derivatives}
 \|v_m\|_{C^k}\leq C_k\Delta_m^{-M_k}
       e^{-c\Delta_m^{-1/(\alpha-1)}}=o(\Delta_m^N).
\end{equation}
Here $M_k$ and $C_k$ are independent of $m$: throughout the support,
each positive argument of {$f_{\lambda_m}$} is at most $\Delta_m/4$, {$\lambda_m\geq\lambda_*$}, and
exponential decay absorbs every fixed negative power of that argument.

Choose {$\tau=\tau_m\in[3/8,5/8]$} as in \eqref{taumo} below, and set
\begin{equation}\label{XI}
 \xi_m(x):=v_m\left(x-(\tau-1/2)\right).
\end{equation}
Then
\[
 \operatorname{supp}\xi_m\cap[0,1]
 =[\tau-\Delta_m/8,\tau+\Delta_m/8],\qquad
 \xi_m(\tau)=v_m(1/2),\qquad
 \|\xi_m\|_{\alpha,R}\leq C\Delta_m^2.
\]
In particular, \eqref{gev-derivative-loss} gives
$\|\chi_m\|_{\alpha,L}=O(\Delta_m^2)$ for $\chi_m=\xi_m'$ and
every fixed $0<L<d$. Under our total-order map norm, all mixed
derivatives of order $k$ of $w(x+y)$ equal $w^{(k)}(x+y)$.
Consequently,
\begin{equation}\label{map-gevrey-small}
 \|F_{m,a}-T\|_{\alpha,L}
 =\|2\pi\delta_m\sin(2\pi\,\cdot)+\chi_m\|_{\alpha,L}
 \leq C_L\delta_m\longrightarrow0.
\end{equation}

\begingroup
For the covering take $\lambda_q=\lambda_*q^p$ and let
$W_q(x):=\xi_q(qx)$. On the cell $k/q\le x\le(k+1)/q$, put
$c=(k+\tau)/q$ and $s=\Delta_q/(8q)$, where $\tau$ is the center
chosen for $\xi_q$. The identity
$f_{\lambda_*q^p}(qt)=f_{\lambda_*}(t)$ gives
\[
 W_q(x)=\Delta_q^2
 f_{\lambda_*}(s+x-c)\,f_{\lambda_*}(s+c-x).
\]
The function vanishes near the cell boundaries. Each cell is a translate
of the first, so \eqref{leib} on $\R$ gives
$\|W_q\|_{\alpha,R}\leq\delta_q\|f_{\lambda_*}\|_{\alpha,R}^2$.
For every fixed $0<L<R$, \eqref{gev-derivative-loss} now yields
\begin{equation}\label{covered-bump-norm}
 \big\|q^{-1}\xi_q'(q\,\cdot)\big\|_{\alpha,L}
 =q^{-2}\|W_q'\|_{\alpha,L}\leq C_L q^{-2}\delta_q .
\end{equation}

The scaling of $\lambda_q$ can also be seen at a left endpoint $x_*$
of the support. Temporarily allow an arbitrary $\lambda_q>0$.
With $g_q(x)=q^{-1}\xi_q'(qx)$, the defining formula
gives
\[
 \lim_{t\downarrow0} t^p\log|g_q(x_*+t)|
       =-\sqrt2\,\lambda_q q^{-p}.
\]
Indeed, the other factor has a positive limit, and differentiating
the flat factor adds only a power of $t$ to its leading term.
If $g_q\in G^{\alpha,L}$, flatness and Taylor's formula imply
\[
 |g_q(x_*+t)|\leq\|g_q\|_{\alpha,L}
        (k!)^{\alpha-1}(t/L^\alpha)^k \qquad(k\geq1).
\]
Taking $k=\lfloor(L^\alpha/t)^p\rfloor$ and applying Stirling's
formula gives $\sqrt2\,\lambda_q q^{-p}\geq
(\alpha-1)L^{\alpha p}$. A constant $\lambda_q$ thus fails the
necessary condition at every fixed positive $L$ for large $q$.
For each $q$ the sine term has finite norm at every such $L$, so its
addition cannot restore membership in $G^{\alpha,L}$.
\endgroup

\section{\sc Destruction of $\Gamma_\omega$ for $F_{m,a}$}\label{3}
\subsection{Basic tools from Aubry--Mather theory}
Let $G(x,y)=(X(x,y),Y(x,y))$ be a lift of a diffeomorphism of the cylinder.
We assume
\begin{itemize}
\item $G(x+1,y)=G(x,y)+(1,0)$;
\item $\partial_yX>0$, and $(x,y)\mapsto(x,X(x,y))$ is a diffeomorphism
of $\R^2$;
\item the form $Y\,dX-y\,dx$ has a primitive periodic in $x$ in the
original $(x,y)$ coordinates.
\end{itemize}
Expressing that primitive in the coordinates $(x,X)$ gives a generating
function $h$ satisfying
\[
 h(x+1,X+1)=h(x,X),\qquad
 y=-\partial_1h(x,X),\qquad Y=\partial_2h(x,X).
\]
We use the usual coercivity and uniform twist assumptions of
Aubry--Mather theory; they hold for \eqref{h2} and \eqref{h}, since their
mixed derivative is $-1$ and their kinetic term is quadratic.

\subsubsection{\bf Minimal configuration}
An orbit of $G$ projects to a stationary configuration
$\boldsymbol{x}=(x_i)_{i\in\Z}$ satisfying
\[
 \partial_1h(x_i,x_{i+1})+\partial_2h(x_{i-1},x_i)=0.
\]
For $r<s$, write
\[
 \mathcal S^h_{r,s}(\boldsymbol{x})
 :=\sum_{i=r}^{s-1}h(x_i,x_{i+1}).
\]
A configuration is minimal if every such finite segment minimizes the
corresponding action with its two endpoints fixed. The number of
entries is kept fixed in this definition. Every minimal configuration
is stationary. Its piecewise affine interpolation is called its Aubry
diagram. We recall the following facts from \cite{B,M3}.
\begin{itemize}
\item Two distinct minimal configurations have Aubry diagrams that
cross at most once.
\item Every minimal configuration has a rotation number
\[
 \rho(\boldsymbol{x})=\lim_{n\to\infty}\frac{x_{i+n}-x_i}{n},
\]
independent of $i$, and minimal configurations exist for every real
rotation number.
\item For irrational $\omega$, the set $M_\omega^h$ of minimal
configurations of rotation number $\omega$ is totally ordered.
\item At a rational rotation number $p/q$ in lowest terms, a minimal
configuration satisfies exactly one of
\[
 x_{i+q}>x_i+p\ (\forall i),\qquad
 x_{i+q}=x_i+p\ (\forall i),\qquad
 x_{i+q}<x_i+p\ (\forall i).
\]
Its rotation symbol is respectively $p/q+$, $p/q$, or $p/q-$.
Here $M_{p/q}^h$ denotes the periodic configurations; $M_{p/q+}^h$
contains the periodic configurations and those with symbol $p/q+$,
and $M_{p/q-}^h$ is defined analogously. Each of these three sets is
totally ordered.
\item For every rotation symbol $\sigma$, the projected set
$\mathcal A_\sigma^h:=\{x_0:\boldsymbol{x}\in M_\sigma^h\}$ is closed
and invariant under integer translations.
\end{itemize}

\subsubsection{\bf Peierls's barrier}\label{persb}
For irrational $\omega$, Mather's criterion states that an invariant
circle of rotation number $\omega$ exists if and only if
$P_\omega^h$ vanishes identically \cite{M3}. We distinguish rational
rotation numbers from the one-sided symbols $p/q\pm$ below.
If $\xi\in\mathcal A_\sigma^h$, put $P_\sigma^h(\xi)=0$.
Otherwise let $(\xi^-,\xi^+)$ be the complementary interval containing
$\xi$, and let $\boldsymbol{x}^{\pm}\in M_\sigma^h$ have
$x_0^{\pm}=\xi^{\pm}$. Then
\[
 P_\sigma^h(\xi):=
 \min\left\{\sum_{i\in I}\bigl(h(z_i,z_{i+1})-
 h(x_i^-,x_{i+1}^-)\bigr):
 z_0=\xi,\ x_i^-\leq z_i\leq x_i^+\right\}.
\]
Here $I=\Z$ for irrational or one-sided rational symbols. For
$\sigma=p/q$, take $I=\{0,\ldots,q-1\}$ and impose
$z_{i+q}=z_i+p$. {For an infinite index set the summand
is the action difference of a single edge. Existence of the minimum,
nonnegativity, periodicity, and the identity
$\{P_\sigma^h=0\}=\mathcal A_\sigma^h$ follow from \cite[\S6]{M3}.
A configuration attaining this minimum outside $\mathcal A_\sigma^h$
minimizes the relative action subject to $z_0=\xi$ and the displayed
order constraints. In particular, stationarity at $i=0$ cannot be
inferred from this minimization.}
We will also use the one-sided continuity theorem
\cite[Theorem~2.2]{M4}:
\begin{equation}\label{rational-sided-barrier}
 P_\omega^h\longrightarrow P_{p/q\pm}^h
 \quad\text{uniformly as irrational }\omega\longrightarrow(p/q)^{\pm}.
\end{equation}
This concerns the one-sided barriers, and does not assert continuity
at the periodic symbol $p/q$.

\subsection{A lower bound of the Peierls barrier}
Recall that
\[
 h_m(x,X)=\tfrac12(X-x)^2+u_m(X),\qquad
 H_m(x,X)=\tfrac12(X-x)^2+V_m(X),
\]
where
\[
 u_m(x):=\delta_m(1-\cos2\pi x),\qquad
 V_m(x):=u_m(x)+\xi_m(x).
\]
We first work with a minimizing heteroclinic configuration for $h_m$
from $0$ to $1$. Such a configuration minimizes the finite total action
\[
 \mathcal K_{h_m}:=\inf_{\boldsymbol{x}\in\mathscr C_{01}}
 \sum_{i\in\Z}h_m(x_i,x_{i+1}){.}
\]
{Here $\mathscr C_{01}$ consists of nondecreasing
configurations with limits $0$ and $1$ at $-\infty$ and $+\infty$.
A step from $0$ to $1$ gives a finite comparison action. Translate
the indices in each member of a minimizing sequence so that $x_0\leq 1/2\leq x_1$.
A coordinatewise limit in $[0,1]^{\Z}$ is nondecreasing and has finite
action by Fatou's lemma. Both end limits must be zeros of $u_m$.
The normalization therefore forces the limits to be $0$ and $1$.

We record a finite-segment consequence used below for both
$W=u_m$ and $W=V_m$. Write
$S_W(\boldsymbol z):=\sum_i [\tfrac12(z_{i+1}-z_i)^2+W(z_{i+1})]$
and let $\mathscr D_{01}$ consist of all configurations with limits
$0$ and $1$. Then
\begin{equation}\label{hc-hetero-free-class}
 \inf_{\boldsymbol z\in\mathscr D_{01}} S_W(\boldsymbol z)
 =\min_{\boldsymbol z\in\mathscr C_{01}} S_W(\boldsymbol z).
\end{equation}
The same identity holds when both extrema are restricted by
$z_0=s$, for any $s\in(0,1)$. The compactness proof applies to either
potential. With
$z_0=s$ the prescribed entry fixes the index and separates the two
end limits. To prove equality, take a finite-action competitor in
$\mathscr D_{01}$ and project its entries to $[0,1]$. Projection
decreases each jump length and replaces any potential outside the
interval by zero. Truncate at $-N,N$ and append constant $0$ and $1$
tails. The two new edges have total cost tending to zero. Sort the
finite list between the appended endpoints in nondecreasing order.
This preserves the potential sum. Every gap in the sorted list was
crossed at least once by the original list. A jump crossing gaps of
lengths $d_1,\ldots,d_l$ contributes $(\sum_j d_j)^2\geq\sum_j d_j^2$
to twice the kinetic action. Sorting therefore decreases the kinetic
sum. In the constrained problem translate the indices of the sorted
list to put the occurrence of $s$ back at index zero. Passing to the
limit gives \eqref{hc-hetero-free-class}.

It follows that a minimizing heteroclinic minimizes on every finite
segment with fixed endpoints, with the extra condition at index zero
when it is prescribed. In the free problem every entry is stationary.
If $x_i=x_{i+1}$, the Euler--Lagrange equations at $i,i+1$ give
$x_{i-1}=x_{i+2}$. Monotonicity then forces these four entries to be
equal, and iteration would make the configuration constant. Thus the
free minimizer for $h_m$ is strictly increasing and has symbol $0^+$.}

\begin{Lemma}\label{lowstep}
{Let $\boldsymbol x\in\mathscr C_{01}$ minimize the
total $h_m$-action from $0$ to $1$. For sufficiently large $m$,
\[
 x_i\in[1/4,3/4]\quad\Longrightarrow\quad
 x_{i+1}-x_i\geq\tfrac12\Delta_m.
\]}
\end{Lemma}
\begin{proof}
Deleting $x_i$ and shifting the subsequent indices produces another
member of $\mathscr C_{01}$. Both total actions are finite, so
minimality of $\mathcal K_{h_m}$ gives
\[
 0\leq h_m(x_{i-1},x_{i+1})-h_m(x_{i-1},x_i)-h_m(x_i,x_{i+1})
 =(x_{i+1}-x_i)(x_i-x_{i-1})-u_m(x_i).
\]
Thus
\begin{equation}\label{uwith}
 x_{i+1}-x_{i-1}\geq2\sqrt{u_m(x_i)}.
\end{equation}
For $x_i\in[1/4,3/4]$, this gives
\begin{equation}\label{ls}
 x_{i+1}-x_{i-1}\geq2\Delta_m.
\end{equation}
Stationarity and the correct derivative of $u_m$ yield
\[
 x_{i+1}-2x_i+x_{i-1}
 =2\pi\Delta_m^2\sin(2\pi x_i).
\]
Adding this identity to \eqref{ls} and dividing by two gives
$x_{i+1}-x_i\geq\Delta_m-\pi\Delta_m^2\geq\Delta_m/2$.
Notice that the deletion comparison is made in the finite-total-action
heteroclinic problem; it is not an allowed fixed-length variation of
an arbitrary configuration of positive rotation number.
\end{proof}

Insertion of the midpoint of $[x_i,x_{i+1}]$ is another admissible
competitor for $\mathcal K_{h_m}$, and gives
\[
 0\leq-\tfrac14(x_{i+1}-x_i)^2+
 u_m\bigl((x_i+x_{i+1})/2\bigr),\qquad
 x_{i+1}-x_i\leq2\sqrt2\,\Delta_m.
\]
Choose a gap with $x_j\leq1/2<x_{j+1}$ and set
$\tau:=(x_j+x_{j+1})/2$. For large $m$, its endpoints lie in
$[1/4,3/4]$ and $\tau\in[3/8,5/8]$. Lemma~\ref{lowstep} implies
that the closed interval of radius $\Delta_m/8$ about $\tau$ is strictly
inside this gap. The choice \eqref{XI} satisfies
\begin{equation}\label{taumo}
 \{x_i:i\in\Z\}\cap\operatorname{supp}\xi_m=\varnothing.
\end{equation}

Define
\[
 \mathcal K_m:=\inf_{\boldsymbol{z}\in\mathscr C_{01}}
 \sum_{i\in\Z}H_m(z_i,z_{i+1}),\qquad
 \mathcal K_m(\tau):=\inf_{\substack{\boldsymbol{z}\in\mathscr C_{01}\\z_0=\tau}}
 \sum_{i\in\Z}H_m(z_i,z_{i+1}).
\]
{Both minima are attained by
\eqref{hc-hetero-free-class}. Since $H_m(x,x)=V_m(x)$ is nonnegative
and has zero set $\Z$, the set of minimizing period-one entries is
$\mathcal A_0^{H_m}=\Z$. Thus $(0,1)$ is the sole gap modulo $\Z$ in
Mather's heteroclinic formula \cite[pp.~207--208, (4.2)]{M4}, giving
\[
 P_{0^+}^{H_m}(\tau)=\mathcal K_m(\tau)-\mathcal K_m.
\]
A minimizer for $\mathcal K_m(\tau)$ is stationary away from the
prescribed entry by \eqref{hc-hetero-free-class}.}
Since $\xi_m\geq0$ and the $h_m$-minimizer avoids its support,
\[
 \mathcal K_m=\mathcal K_{h_m},\qquad
 \mathcal K_m(\tau)\geq\mathcal K_{h_m}+\xi_m(\tau).
\]
Consequently,
\begin{equation}\label{lowb}
 P_{0^+}^{H_m}(\tau)\geq\Delta_m^2
 \exp\left[-{\lambda_m}\,2^{\frac3{\alpha-1}+\frac32}
 \exp\left(\frac{d}{2(\alpha-1)}m^\beta\right)\right].
\end{equation}

\subsection{The modulus of continuity of the Peierls barrier}
We follow the hyperbolic comparison of Peierls barriers in
\cite[Lemma~5.1 and Section~5]{W} and its Gevrey version
\cite[Lemma~4.3]{W15}, giving the finite-action details for the
present parameters and for constrained configurations.
\begin{Lemma}\label{pw}
Fix $a>1$, $d>0$, {$0<\varepsilon<\min\{1,a/2\}$} and
$\alpha>1+a/\varepsilon$. For all sufficiently large $m$, every
irrational $\omega$ satisfying $0<\omega<n_m^{-a/2-\varepsilon}$ obeys
\begin{equation}\label{app}
 \left|P_\omega^{H_m}(\tau)-P_{0^+}^{H_m}(\tau)\right|
 \leq C\exp\left[-2\exp\left(\frac{d\varepsilon}{2a}m^\beta\right)\right],
\end{equation}
where {$\tau$ is the bump centre in \eqref{taumo}.} The constant is independent of $m$ and
$\omega$ in this range.
\end{Lemma}
\begin{proof}
The argument is the finite-action comparison of \cite[Section~5]{W}
and \cite[Lemma~4.3]{W15}. We give the details needed at constrained
points and at the ends of the truncated configurations.
Put
\[
 t_m:=\exp((d/a)m^\beta),\qquad
 \varepsilon_m:=\exp(-t_m^{\varepsilon/2}),\qquad
 D_m:=\Delta_m^{-1}\log(1/\varepsilon_m)=t_m^{a/2+\varepsilon/2}.
\]
Thus $\delta_m=t_m^{-a}$, $n_m=\lfloor t_m\rfloor$, and the frequency
hypothesis gives $\omega^{-1}\geq c t_m^{a/2+\varepsilon}\gg D_m$.
This uses $n_m/t_m\to1$, not an equality between $n_m$ and $t_m$.
All constants below are independent of $m$ and $\omega$ in this range.

\emph{Counting and choice of cuts.}
We first justify a uniform counting estimate. For an integer $j$,
and a dyadic interval on either side of $j$ where
$r\leq |x-j|\leq2r\leq1/2$, we have
\[
 c\Delta_m^2r^2\leq V_m(x),\qquad
 \sup_{|x-j|\leq2r}V_m(x)\leq C\Delta_m^2r^2.
\]
The upper bound follows from \eqref{bump-fixed-derivatives} and the
fact that the bump stays a fixed distance from the integers.
A monotone minimizing segment with $N$ steps in this interval has
action at least $cN\Delta_m^2r^2$. If $N\geq2\lceil\Delta_m^{-1}\rceil$,
compare it with a segment of the same length which moves linearly
from its first endpoint to $j$ in $\lceil\Delta_m^{-1}\rceil$ steps,
waits at $j$, and moves to the last endpoint in the same number of
steps. Its action is at most $C\Delta_m r^2$. Minimality
implies $N\leq C\Delta_m^{-1}$; for shorter segments the same bound
is immediate. Summing over dyadic intervals gives
\begin{equation}\label{sig}
 \#\{i:x_i\in[\varepsilon_m,1-\varepsilon_m]\}
 \leq C\Delta_m^{-1}\log(1/\varepsilon_m)=CD_m.
\end{equation}
The same bound holds in
$[j-1/2,j+1/2]\setminus(j-\varepsilon_m,j+\varepsilon_m)$.
\begingroup
By \eqref{hc-hetero-free-class}, each free heteroclinic minimizer
also minimizes every fixed-length segment with its endpoints fixed.
For a heteroclinic minimizer pinned at index zero, the same conclusion
holds on each half of the sequence, including segments ending or
starting at zero. If an interval of indices in one dyadic shell
contains zero, split it there. Each half admits the preceding comparison
through the nearest integer. This gives \eqref{sig}, with at most twice
the former constant, for the pinned heteroclinic as well.

If $\tau\notin\mathcal A_\omega^{H_m}$, let
$\boldsymbol{x}^{-},\boldsymbol{x}^{+}$ be the minimizing
configurations at the ends of its gap and let $\boldsymbol{z}$
minimize the relative action in that strip subject to $z_0=\tau$.
The shift $(x^-_{i+1})_i$ is a minimizing configuration strictly
above $\boldsymbol{x}^-$ since $\omega>0$. Adjacency therefore gives
\[
 x_i^-\leq z_i\leq x_i^+\leq x^-_{i+1}\qquad(i\in\Z).
\]
If $\tau\in\mathcal A_\omega^{H_m}$, choose a minimizing configuration
through $\tau$ and set $\boldsymbol{x}^- = \boldsymbol{x}^+
=\boldsymbol{z}$ equal to it. In both cases $x_0^-\leq\tau\leq x_1^-$.

For clarity, take a half-open unit interval $[u,u+1)$ and put
$k=\lfloor1/\omega\rfloor$. The order of the integer translates of
$\boldsymbol{x}^-$ gives
\[
 x^-_{i+k}< x^-_i+1 <x^-_{i+k+1}
\]
because $k\omega<1<(k+1)\omega$. If $\ell$ is the first index with
$x^-_\ell\geq u$, then $x^-_{\ell+k-1}<x^-_{\ell-1}+1<u+1$
and $x^-_{\ell+k+1}>x^-_\ell+1\geq u+1$. Thus the number of entries
in the interval is $k$ or $k+1$. Apply this to
$[j-1/2,j+1/2)$ for $j=0,1$. The shell estimate leaves at least
$k-CD_m$ consecutive entries within $\varepsilon_m$ of $j$.
The block at $0$ ends by index zero since $x^-_1\geq\tau$,
and the block at $1$ starts after index zero since $x^-_0\leq\tau$.
Their middle indices $i^-$ and $i^+$ satisfy
\begin{equation}\label{5ww}
 \begin{split}
 \max_{|h-i^-|\leq2}|x_h^-|&\leq\varepsilon_m,\qquad
 \max_{|h-i^+|\leq2}|x_h^--1|\leq\varepsilon_m,\\
 \min\{-i^-,i^+\}&\geq c\omega^{-1}\gg D_m .
 \end{split}
\end{equation}
By interlacing, the same endpoint bounds hold for $z_h$ with
$|h-i^\pm|\leq1$.
\endgroup
\emph{Finite actions versus heteroclinic actions.}
\begingroup
Even the forward kinetic action of a configuration with rotation number
$\omega>0$ diverges, since
\[
 \sum_{i=0}^{N-1}\tfrac12(x_{i+1}-x_i)^2
 \geq (x_N-x_0)^2/(2N)=N\omega^2/2+o(N).
\]
We compare fixed-endpoint finite actions before passing to sums of
edgewise differences.
On any finite segment with endpoints fixed and with its zeroth
entry fixed if the segment contains zero, $\boldsymbol{z}$ minimizes
over all competitors. In the gap case, let $w$ be such a competitor
and write $\wedge,\vee$ for componentwise minimum and maximum.
The twist inequality, summed on the segment, gives
\[
 \mathcal S(w\wedge\boldsymbol{x}^+)+
 \mathcal S(w\vee\boldsymbol{x}^+)\leq
 \mathcal S(w)+\mathcal S(\boldsymbol{x}^+).
\]
The endpoints of $w\vee\boldsymbol{x}^+$ are those of
$\boldsymbol{x}^+$. Its action is at least
$\mathcal S(\boldsymbol{x}^+)$ by minimality. Clipping $w$ from
above thus decreases its action or leaves it unchanged. The same
argument clips from below at $\boldsymbol{x}^-$. Both operations
preserve the endpoints and the prescribed value $\tau$.
The clipped segment can replace the corresponding part of
$\boldsymbol{z}$ in the definition of the relative minimum.
In the Aubry-set case the asserted finite-segment property follows
directly from the global minimality of $\boldsymbol{z}$.
\endgroup
Set $\mathfrak A_m:=\mathcal S^{H_m}_{i^-,i^+}(\boldsymbol{x}^-)$ and
$\mathfrak Z_m:=\mathcal S^{H_m}_{i^-,i^+}(\boldsymbol{z})$.
We claim
\[
 |\mathfrak A_m-\mathcal K_m|+|\mathfrak Z_m-\mathcal K_m(\tau)|\leq C\varepsilon_m^2.
\]
\begingroup
For the lower bounds, project the entries of either segment to $[0,1]$.
Projection decreases the kinetic term and can only remove nonnegative
potential contributions. Append constant tails at $0$ and $1$.
The entry adjacent to each added tail is within $\varepsilon_m$ of
the corresponding integer, where $V_m(s)\leq
C\Delta_m^2\operatorname{dist}(s,\Z)^2$. The two new edges therefore
cost at most $C\varepsilon_m^2$ in total. The result is admissible
by \eqref{hc-hetero-free-class}, with $z_0=\tau$ in the constrained
comparison. This proves $\mathcal K_m\leq\mathfrak A_m+
C\varepsilon_m^2$ and $\mathcal K_m(\tau)\leq\mathfrak Z_m+
C\varepsilon_m^2$.

Conversely, take a free or pinned minimizing heteroclinic
$\boldsymbol{w}$. In the free case reindex a step that crosses
$1/2$ to have its left endpoint at index zero. Retain the segment
from its last entry at most $\varepsilon_m$ to its first entry at
least $1-\varepsilon_m$. By \eqref{sig} it uses at most $CD_m+2$
steps. In the free case either the retained segment contains index
zero or zero is one of its endpoints, and in the pinned case it
contains $w_0=\tau$. By \eqref{5ww} there is room for this segment
on both sides of its original index zero, with at least two unused
indices at either end. Fill these indices with $0$'s on the left
and $1$'s on the right, except for the two prescribed endpoint values
at $i^-,i^+$. All four connecting edges join points within
$\varepsilon_m$ of the same integer and cost $O(\varepsilon_m^2)$.
The retained action is bounded by the original nonnegative total
action. The finite-segment minimizing properties of
$\boldsymbol{x}^-$ and $\boldsymbol{z}$ now give the upper bounds.
\endgroup
Let $\boldsymbol{y}$ agree with $\boldsymbol{z}$ for $i^-<i<i^+$
and with $\boldsymbol{x}^-$ outside. Replacing the two endpoint
values changes $\mathfrak Z_m$ by $O(\varepsilon_m^2)$. Thus
\begin{equation}\label{13}
 \begin{split}
 |\mathcal S^{H_m}_{i^-,i^+}(\boldsymbol{x}^-)-\mathcal K_m|
 &\leq C\varepsilon_m^2,\\
 |\mathcal S^{H_m}_{i^-,i^+}(\boldsymbol{y})-\mathcal K_m(\tau)|
 &\leq C\varepsilon_m^2.
 \end{split}
\end{equation}
For example, the left endpoint calculation is
\[
 H_m(0,z)-H_m(x,z)=xz-\tfrac12x^2=O(\varepsilon_m^2)
 \quad (|x|+|z|=O(\varepsilon_m)),
\]
and at the right endpoint one also uses
$V_m(1)=V_m'(1)=0$ and $V_m''=O(\Delta_m^2)$ near $1$.

\emph{Truncating the action difference.}
The configuration $\boldsymbol{y}$ is an admissible constrained
competitor in the Peierls definition, giving the first inequality in
\begin{equation}\label{step1}
 P_\omega^{H_m}(\tau)
 \leq\mathcal S^{H_m}_{i^-,i^+}(\boldsymbol{y})
       -\mathcal S^{H_m}_{i^-,i^+}(\boldsymbol{x}^-)
 \leq P_\omega^{H_m}(\tau)+C\varepsilon_m^2.
\end{equation}
\begingroup
For the upper bound, change $z_{i^-},z_{i^+}$ to the reference
endpoint values, leaving the rest of $\boldsymbol{z}$ fixed.
The four incident edges have both entries within $\varepsilon_m$ of
the corresponding integer before and after this change. Their total
change of action is $O(\varepsilon_m^2)$. The relative action now
splits into the finite middle difference in \eqref{step1} and two
sums of edgewise differences on the tails. Each tail sum is
nonnegative. Indeed, truncate a tail at index $j$ and change its
outermost endpoint to $x^-_j$. Both endpoints then agree with the
reference configuration, so that finite minimality applies.
The change in the last edge is at most $C|x^+_j-x^-_j|$, since
the displacements are uniformly bounded by interlacing and the
unit-translate order. For irrational rotation the translates of the
projected gap along the orbit are disjoint modulo one, giving
$\sum_{j\in\Z}(x^+_j-x^-_j)\leq1$. The same Lipschitz estimate on
each edge gives absolute convergence of the sums of differences,
and the outer endpoint error tends to zero. This proves the upper
bound in \eqref{step1}. When $\tau\in\mathcal A_\omega^{H_m}$ all
three configurations coincide, so \eqref{step1} holds with both
differences zero; \eqref{13} now yields
$0\leq P_{0^+}^{H_m}(\tau)=\mathcal K_m(\tau)-\mathcal K_m
\leq C\varepsilon_m^2$, since the two finite actions in that display
are equal. This proves the required estimate also on the Aubry set.
\endgroup
Combining \eqref{13} and \eqref{step1} gives
\[
 |P_\omega^{H_m}(\tau)-P_{0^+}^{H_m}(\tau)|
 \leq C\varepsilon_m^2,
\]
which is exactly \eqref{app}.
\end{proof}

We now complete the proof of Theorem~\ref{Mthm1}.
Let $B_m$ be the right-hand side of \eqref{lowb}, and let $E_m$ be the
\begingroup
right-hand side of \eqref{app}. Write $p=1/(\alpha-1)$.
The choice $\alpha>1+a/\varepsilon$ gives
$dp/2<d\varepsilon/(2a)$. Since $\lambda_m=O(m^p)$,
\[
 \lambda_m\exp\bigl((dp/2)m^\beta\bigr)
 =o\!\left(\exp\bigl((d\varepsilon/(2a))m^\beta\bigr)\right).
\]
Taking logarithms of $E_m/B_m$, including the factor
$\Delta_m^2=\exp(-dm^\beta)$ in $B_m$, gives $E_m/B_m\to0$.
In particular, uniformly over positive irrational $\omega$ in
\eqref{omb},
\begin{equation}\label{uniform-positive-barrier}
 P_\omega^{H_m}(\tau)\geq B_m-E_m\geq\tfrac12 B_m>0
\end{equation}
\endgroup
for sufficiently large $m$. Mather's criterion excludes every such
invariant circle.

\begingroup
Fix $m$ large and a rational $r=p/q$ in the open interval
$(0,n_m^{-a/2-\varepsilon})$. Approach $r$ from either side by
irrationals in this interval. The lower bound in
\eqref{uniform-positive-barrier} is independent of these irrationals;
Mather's one-sided estimate \eqref{rational-sided-barrier} thus yields
\[
 P_{r+}^{H_m}(\tau)\geq B_m/2,\qquad
 P_{r-}^{H_m}(\tau)\geq B_m/2 .
\]
The two inequalities cover the two possible orientations of a
nonperiodic orbit on a rational invariant circle.
\endgroup
Suppose an invariant circle of rotation number $r$ existed. By
Birkhoff's theorem it is a Lipschitz graph $y=\psi(x)$, and the induced
lift $g(x):=x+\psi(x)$ is an increasing homeomorphism.
We verify directly that every orbit on this graph is minimizing.
Put $c:=\int_0^1\psi(x)\,dx$ and choose a periodic $C^1$ function $S$
with $S'=\psi-c$. Define
\[
 Q(x,X):=H_m(x,X)-c(X-x)+S(x)-S(X).
\]
Then $\partial_xQ(x,X)=g(x)-X$, so for fixed $X$ its unique minimum
is at $x=g^{-1}(X)$. Graph invariance gives
$\psi(g(x))=\psi(x)+V_m'(g(x))$, and hence both partial derivatives of
$Q$ vanish on $(x,g(x))$. Since this parametrized curve is locally
Lipschitz, $Q(x,g(x))$ is constant, say $C_0$. Thus
\[
 H_m(x,X)\geq c(X-x)+S(X)-S(x)+C_0,
\]
with equality along graph orbits. Summation shows that each finite
segment on the graph minimizes action with its endpoints fixed.
The orbit through $x_0=\tau$ consequently has symbol $r$, $r+$, or
$r-$, so
$\tau\in\mathcal A_{r+}^{H_m}\cup\mathcal A_{r-}^{H_m}$.
At least one of the two barriers must vanish, a contradiction.

Finally, no symmetry of the shifted bump is needed for negative
rotation numbers. For any periodic $V_m$, the involution
\[
 \mathcal R_m(x,y):=(x,V_m'(x)-y)
\]
{satisfies, on the chosen lifts,
\[
 \mathcal R_mF_{m,a}\mathcal R_m(x,y)
   =(x-y+V_m'(x),\,y-V_m'(x))=F_{m,a}^{-1}(x,y).
\]
Its first coordinate is $x$. Hence it maps a graph to a graph,
and the forward displacement on the image equals the backward
displacement on the original graph. Its rotation number is $-\omega$.} Thus the positive case
also excludes every negative rotation number in \eqref{omb}.
The Gevrey convergence was proved in \eqref{map-gevrey-small}, completing
the proof.

\section{\sc Persistence of $\Gamma_{\bar\omega}$ for $\bar f_{m,a}$}\label{5}

We construct an invariant circle for $f_{q_m,a}$ at rotation
$\omega_m$ and lift it to $\bar f_{m,a}$ at rotation $\bar\omega$.
The interpolation result and the symplectic-frame identity used
below are stated in Appendix~\ref{hc-external-results}.
The estimates for the small rotation number and its non-Diophantine
small divisors are proved in this section.

\emph{Parameters and norms.}
Fix $a,b,d,\alpha,\varepsilon$ as in Theorem~\ref{Mthm2}.
Unless a dependence is stated, $c,C$ may depend on these fixed
parameters and on the complex neighborhoods chosen below, but not on
$\mu,\eta$ or $m$ for $0<\mu<\mu_0$, $0<\eta<\eta_0$.
In a $C^k$ estimate, $k$ is fixed and the constant may also depend
on $k$. Thus ``uniformly in $\mu,\eta$'' means that the same constant
works for all these values and all points of the indicated domain.
The variable $\eta$ denotes a level gap; $d$ retains its meaning in
\eqref{bval}.

For $\rho,R>0$, use complex cylinders
\[
 \mathcal V_{\rho,R}:=\{(X,P)\in(\C/\Z)\times\C:
                  |\operatorname{Im}X|<\rho,\ |P|<R\}.
\]
The analytic norm is the supremum norm on the stated domain, and
$\|h\|_{C^k}$ includes all partial derivatives of order at most $k$.
For vectors and matrices, take component norms and a compatible
matrix norm. For a degree-one map, its zeroth-order norm refers to
its periodic displacement from the identity.
Choose fixed nested cylinders containing the real separatrix of
$P^2/2+\cos(2\pi X)/(2\pi)$, with $R$ larger than the height of that
separatrix. Their widths and radii are independent of $\mu,\eta,m$.

\emph{Cauchy's formula and domain reduction.}
If $h$ is holomorphic on $\mathcal V_{\rho,R}$ and
$0<\ell<\min(\rho,R)$, then for $(X,P)\in\mathcal V_{\rho-\ell,R-\ell}$,
\[
 \partial_X^i\partial_P^j h(X,P)
 =\frac{i!j!}{(2\pi\mathrm i)^2}
   \int_{|\zeta-X|=\ell}\int_{|w-P|=\ell}
   \frac{h(\zeta,w)}{(\zeta-X)^{i+1}(w-P)^{j+1}}
                 \,\dd w\,\dd\zeta.
\]
One may first use circles of radius less than $\ell$ and then pass
to the limit. Their lengths are $2\pi\ell$, so
\begin{equation}\label{hc-cauchy-derivatives}
 \|\partial_X^i\partial_P^j h\|_{\rho-\ell,R-\ell}
       \le i!j!\ell^{-(i+j)}\|h\|_{\rho,R},\qquad
 \|h^{(j)}\|_{s-\ell}\le j!\ell^{-j}\|h\|_s.
\end{equation}
The second inequality is the one-variable version on
$\{\theta\in\C/\Z:|\operatorname{Im}\theta|<s\}$.
A domain loss $\ell$ means precisely this reduction of the width
or radius. Every derivative estimate below specifies the larger
and smaller domains, or their positive difference.

\emph{The operator $\mathfrak D_g$ and Lie brackets.}
For scalar functions $g,h$ (functions with one real or complex
value), put
\[
 X_g:=(g_P,-g_X),\qquad
 \mathfrak D_g h:=g_Ph_X-g_Xh_P=X_g\cdot\nabla h.
\]
Thus $\mathfrak D_g h$ is the directional derivative of $h$ along
$X_g$, also written $\mathcal L_{X_g}h$.
With the convention $\{h,g\}:=h_Xg_P-h_Pg_X$ for the Poisson
bracket, $\mathfrak D_g h=\{h,g\}$. The related Lie bracket of vector
fields is
\[
 [X_g,X_h]:=DX_h\,X_g-DX_g\,X_h
     =X_{\mathfrak D_g h},\qquad
 [\mathfrak D_g,\mathfrak D_h]f
     =\mathfrak D_{\mathfrak D_g h}f.
\]
These equalities follow by differentiating the two components and
cancelling the mixed partial derivatives.

For an analytic scalar function $g$, write
\[
 \mathcal B_{t,g}(X,P):=
 \left(\sum_{n\ge0}\frac{t^n}{n!}\mathfrak D_g^nX,
       \sum_{n\ge0}\frac{t^n}{n!}\mathfrak D_g^nP\right).
\]
Here is an explicit convergence bound. Suppose
$|g_X|+|g_P|\le K$ on an intermediate cylinder and leave a fixed
additional width $\ell$ to its boundary. Applying
\eqref{hc-cauchy-derivatives} on $n$ nested cylinders with successive
differences $\ell/n$ gives, for $n\ge1$,
\[
 \|\mathfrak D_g^nX\|+\|\mathfrak D_g^nP\|
 \le K(2Kn/\ell)^{n-1}\le C_0C_1^n n!.
\]
For $X$, the first application is $\mathfrak D_gX=g_P$, a bounded
periodic function. The inequality $n^n\le e^n n!$ follows from
the Stirling formula.
The coordinate series converge for $|t|<(2C_1)^{-1}$, uniformly
for scalar functions with the same bound $K$ and the same domains.
Termwise differentiation and the product rule give
\[
 \partial_t\mathcal B_{t,g}=X_g\circ\mathcal B_{t,g},\qquad
 g\circ\mathcal B_{t,g}=g,\qquad
 \mathcal B_{t,g}\mathcal B_{u,g}=\mathcal B_{t+u,g}
\]
where the compositions are defined. If
$\mathcal B_{t,g}=(X_t,P_t)$, ordinary differentiation also gives
\[
 \frac{\dd}{\dd t}(P_t\,\dd X_t)
 =\dd\bigl((Pg_P-g)\circ\mathcal B_{t,g}\bigr).
\]
For periodic $g$, integration in $t$ gives a periodic primitive for
$\mathcal B_{t,g}^*(P\,\dd X)-P\,\dd X$; hence these maps are exact.

\subsection{An exponentially accurate comparison map}

Put $\mu=\sqrt{2\pi\delta}$ and write $y=\mu P$, $x=X$.
Then $f(x,y)=(x+y,y+2\pi\delta\sin(2\pi(x+y)))$ becomes
\[
 M_\mu=\mathcal K_\mu\mathcal D_\mu,\qquad
 \mathcal D_z(X,P)=(X+zP,P),\quad
 \mathcal K_z(X,P)=(X,P+z\sin(2\pi X)).
\]
Its symmetric conjugate and its leading scalar function are
\begin{equation}\label{eq:hc-symmetric-map}
 \widehat M_\mu=\mathcal K_{-\mu/2}M_\mu\mathcal K_{\mu/2}
 =\mathcal K_{\mu/2}\mathcal D_\mu\mathcal K_{\mu/2},\qquad
 \Inv_0(X,P)=\frac{P^2}{2}+\frac{\cos(2\pi X)}{2\pi}.
\end{equation}

\begin{Lemma}\label{lem:hc-comparison}
For sufficiently small $\mu>0$, there are a real periodic analytic
scalar function $\Inv_\mu$ and an exact analytic map
$\mathcal B_\mu=\mathcal B_{\mu,\Inv_\mu}$ such that, on fixed inner
complex cylinders and for each fixed $k$,
\begin{equation}\label{eq:hc-comparison-estimates}
 \Inv_\mu\circ\mathcal B_\mu=\Inv_\mu,\qquad
 \|\Inv_\mu-\Inv_0\|_{C^k}\le C_k\mu^2,\qquad
 \|\widehat M_\mu-\mathcal B_\mu\|_{C^k}\le C_k e^{-c/\mu}.
\end{equation}
The constants $c,C_k$ are independent of $0<\mu<\mu_0$.
\end{Lemma}

\begin{proof}
Apply Proposition~\ref{hc-external-interpolation} to
\[
 S_\mu:=\mathcal D_\mu\mathcal K_\mu
       =(X+\mu P+\mu^2\sin(2\pi X),P+\mu\sin(2\pi X)).
\]
Its mixed generating function is
\[
 \mathscr S_\mu(P',X)=P'X+
 \mu\left(\frac{(P')^2}{2}+\frac{\cos(2\pi X)}{2\pi}\right),
 \qquad
 P=\partial_X\mathscr S_\mu,\quad X'=\partial_{P'}\mathscr S_\mu.
\]
The part after $P'X$ is globally defined and periodic in $X$.
On $\mathcal V_{\rho,R}$, the only nonzero displacement coefficients
are $F_1=(P,s)$ and $F_2=(s,0)$, where $s=\sin(2\pi X)$;
$\|F_j\|\le\Gamma:=1+R+e^{2\pi\rho}$ for $j=1,2$.
Thus the hypotheses of Proposition~\ref{hc-external-interpolation}
hold with coefficient bound $\gamma=1$.

Normalize $g_0(0,0)=\Inv_0(0,0)$ and $g_j(0,0)=0$ for $j\ge1$.
For the first two coefficients of the scalar function, put
$v_0=(P,s)$ and $v_1=(\partial_Pg_1,-\partial_Xg_1)$.
The coordinate series through degree two is
\[
 (X,P)+\mu v_0+\mu^2\left(v_1+\tfrac12Dv_0\,v_0\right),
 \qquad Dv_0\,v_0=(s,Ps'),\quad s'=2\pi\cos(2\pi X).
\]
Matching it with $S_\mu=(X,P)+\mu(P,s)+\mu^2(s,0)$ gives
\[
 v_1=(s/2,-Ps'/2),\qquad g_1=Ps/2
\]
with $g_0=\Inv_0$ under the chosen normalization.
Choose the truncation constant $c_0$ and $\mu_0$ so that
$C_1(c_0+2\mu_0)\le1/2$. For $2\le j<N(\mu)$,
\[
 \frac{\mu^{j+1}C_1^{j+1}(j+2)!}
      {\mu^jC_1^j(j+1)!}
 =\mu C_1(j+2)\le C_1(c_0+2\mu_0)\le\frac12.
\]
Thus the factorial bound in Proposition~\ref{hc-external-interpolation}
gives
\[
 \left\|\sum_{j=2}^{N(\mu)}\mu^jg_j\right\|_{C^1}
 \le C_0\sum_{j=2}^{N(\mu)}(\mu C_1)^j(j+1)!
 \le12C_0C_1^2\mu^2.
\]
The same bound in $C^k$ on fixed smaller cylinders follows from
\eqref{hc-cauchy-derivatives}. The proposition supplies
the periodic analytic scalar function $G_\mu$ and its map
$\mathcal B_{\mu,G_\mu}$ with
\[
 G_\mu=\Inv_0+\frac\mu2P\sin(2\pi X)+R_\mu,
 \quad \|R_\mu\|_{C^k}\le C_k\mu^2,
 \quad
 \|S_\mu-\mathcal B_{\mu,G_\mu}\|_{C^k}\le C_k e^{-c/\mu}.
\]
Set
\[
 \Inv_\mu:=G_\mu\circ\mathcal K_{-\mu/2},\qquad
 \mathcal B_\mu:=\mathcal K_{\mu/2}
              \mathcal B_{\mu,G_\mu}\mathcal K_{-\mu/2}.
\]
The chain rule and $\det D\mathcal K_{\mu/2}=1$ identify this
conjugate with $\mathcal B_{\mu,\Inv_\mu}$. In particular, it
preserves $\Inv_\mu$. Writing $s=\sin(2\pi X)$ gives the cancellation
\[
 \Inv_0(X,P-\mu s/2)+\frac\mu2(P-\mu s/2)s
       =\Inv_0(X,P)-\frac{\mu^2s^2}{8}.
\]
On fixed smaller cylinders, the derivatives of $R_\mu\circ
\mathcal K_{-\mu/2}$ are bounded by $C_k\mu^2$.
Also $\widehat M_\mu=\mathcal K_{\mu/2}S_\mu\mathcal K_{-\mu/2}$,
so the last bound in \eqref{eq:hc-comparison-estimates} follows
from the chain rule and the displayed bound for $S_\mu$.
All derivatives of these half-kicks on the fixed cylinders are
bounded independently of $\mu$.
\end{proof}

\subsection{The area derivative of rotating levels}

Let $s_\mu$ be the saddle of $\Inv_\mu$ near $(0,0)$, and put
\[
 E_s(\mu)=\Inv_\mu(s_\mu),\qquad
 \chi_\mu:=\sqrt{-\det D^2\Inv_\mu(s_\mu)}.
\]
For the upper rotating root $P=P(X,E)$ of $\Inv_\mu(X,P)=E>E_s(\mu)$,
define
\[
 I(E)=\int_0^1P(X,E)\,\mathrm dX,\qquad
 T_\mu(\eta)=I'(E_s(\mu)+\eta),\quad \eta>0.
\]
Write $\Omega_\mu(I(E))$ for the real rotation number of
$\mathcal B_\mu$ on this branch.

\begin{Lemma}\label{lem:hc-area-derivative}
For some fixed $\eta_0>0$,
\begin{equation}\label{eq:hc-area-expansion}
 T_\mu(\eta)=\mathsf A_\mu(\eta)\log(\eta_0/\eta)
                 +\mathsf B_\mu(\eta),\qquad
 \mathsf A_\mu(0)=\chi_\mu^{-1},
\end{equation}
where $\mathsf A_\mu,\mathsf B_\mu$ are analytic on a fixed disk
$|\eta|<r_0$ and satisfy
\[\sup_{0<\mu<\mu_0}\sup_{|\eta|<r_0}
 (|\mathsf A_\mu(\eta)|+|\mathsf B_\mu(\eta)|)\le C.\] For sufficiently small $\eta>0$,
\[
 T_\mu(\eta)\asymp\log(1/\eta),\qquad
 -T_\mu'(\eta)\asymp\eta^{-1},\qquad
 \Omega_\mu(I(E_s(\mu)+\eta))=\frac{\mu}{T_\mu(\eta)}.
\]
The constants in these bounds are independent of
$0<\mu<\mu_0$ and $0<\eta<\eta_0$.
\end{Lemma}

\begin{proof}
\emph{1. The saddle and the rotating branch.}
By \eqref{eq:hc-comparison-estimates},
$\Inv_{\mu,PP}=1+O(\mu^2)>1/2$.
The equation $\Inv_{\mu,P}(X,\psi(X))=0$ has an analytic periodic
solution $\psi=O(\mu^2)$. Put
$W(X)=\Inv_\mu(X,\psi(X))$.
Since $W$ is $C^r$-close to $\cos(2\pi X)/(2\pi)$, fixed small
neighborhoods $U_0,U_{1/2}$ of $0,1/2$ satisfy
\[
 |W'|\ge c\quad\hbox{outside }U_0\cup U_{1/2},\qquad
 W''\le-c\quad\hbox{on }U_0,\qquad
 W''\ge c\quad\hbox{on }U_{1/2}.
\]
Thus $W$ has one maximum and one minimum per period. Its maximum
at $X_s$ gives the unique saddle
$s_\mu=(X_s,\psi(X_s))$ near $(0,0)$, with
\[
 s_\mu=O(\mu^2),\qquad
 E_s(\mu)=\frac1{2\pi}+O(\mu^2),\qquad
 \chi_\mu=\sqrt{2\pi}+O(\mu^2).
\]
Strict convexity in $P$ gives the upper rotating root for $E>E_s(\mu)$
sufficiently close to $E_s(\mu)$.

\emph{2. The logarithmic term.}
Write $E_s=E_s(\mu)$. Taylor's formula gives, near $s_\mu$,
\[
 \Inv_\mu-W=(P-\psi(X))^2A(X,P),\qquad
 E_s-W=(X-X_s)^2C(X),\qquad A,C>0.
\]
Set
\[
 Z=(P-\psi(X))\sqrt{A(X,P)},\qquad
 Q=(X-X_s)\sqrt{C(X)},\qquad u=Z+Q,\quad v=Z-Q.
\]
On fixed neighborhoods, their Jacobian determinants satisfy
$c\le|\det D(u,v)/D(X,P)|\le C$ for all $0<\mu<\mu_0$, and
\[
 \Inv_\mu-E_s=uv,\qquad
 \mathrm dX\wedge\mathrm dP=a_\mu(u,v)\,\mathrm du\wedge\mathrm dv,
 \qquad a_\mu(0,0)=\chi_\mu^{-1}.
\]
The value of $a_\mu(0,0)$ follows by taking determinants in the
Hessian transformation at $s_\mu$, with the positive orientation
specified by $(u,v)$.

Implicit differentiation gives
\[
 I'(E)=\int_0^1\frac{\mathrm dX}
                         {\Inv_{\mu,P}(X,P(X,E))}.
\]
Choose the exit sections $u=\varrho$, $v=\varrho$, with fixed
$\varrho>0$. On the upper branch $u,v>0$, $uv=\eta$, so
$u$ increases from $\eta/\varrho$ to $\varrho$; see
Figure~\ref{hc-saddle-diagram}. The saddle contribution is
\[
 \int_{\eta/\varrho}^{\varrho}a_\mu(u,\eta/u)\,\frac{\mathrm du}{u}.
\]
\begin{figure}[htbp]
\centering
\setlength{\unitlength}{1pt}
\begin{picture}(305,160)
 \put(30,20){\vector(1,0){248}}
 \put(30,20){\vector(0,1){126}}
 \put(281,14){\makebox(0,0){$u$}}
 \put(23,149){\makebox(0,0){$v$}}
 \put(20,9){\makebox(0,0){$0$}}
 \multiput(30,130)(5,0){45}{\line(1,0){2}}
 \multiput(250,20)(0,5){23}{\line(0,1){2}}
 \put(16,130){\makebox(0,0){$\varrho$}}
 \put(250,7){\makebox(0,0){$\varrho$}}
 \qbezier(47.6,130)(53.47,93.33)(65.2,75)
 \qbezier(65.2,75)(76.93,56.67)(100.4,47.5)
 \qbezier(100.4,47.5)(123.87,38.33)(170.8,33.75)
 \qbezier(170.8,33.75)(199.3,30.97)(250,28.8)
 \put(47.6,130){\circle*{3}}
 \put(250,28.8){\circle*{3}}
 \put(114,44){\vector(4,-1){22}}
 \put(141,72){\makebox(0,0){$uv=\eta$}}
 \put(113,147){\makebox(0,0){$(\eta/\varrho,\varrho)$}}
 \put(234,53){\makebox(0,0){$(\varrho,\eta/\varrho)$}}
\end{picture}
\caption{Schematic of the upper rotating branch in saddle coordinates.
The arrow is the integration direction from $v=\varrho$ to $u=\varrho$;
its $u$-limits are $\eta/\varrho$ and $\varrho$.}
\label{hc-saddle-diagram}
\end{figure}

Expand $a_\mu(u,v)=\sum_{i,j\ge0}c_{ij}(\mu)u^iv^j$. Its monomial
integrals are
\[
 \int_{\eta/\varrho}^{\varrho}u^i(\eta/u)^j\frac{\mathrm du}{u}
 =\begin{cases}
 \eta^i\log(\varrho^2/\eta),&i=j,\\[2mm]
 \displaystyle\frac{\varrho^{i-j}\eta^j-\varrho^{j-i}\eta^i}{i-j},&i\ne j.
 \end{cases}
\]
Choose $R>\varrho$ so that $a_\mu$ is analytic on the product of
complex disks $|u|<R$, $|v|<R$ (called a bidisk), and
$\sup_{0<\mu<\mu_0}\|a_\mu\|\le A_0$. Applying Cauchy's formula
successively in $u$ and $v$ gives
\[
 |c_{ij}(\mu)|\le A_0R^{-i-j}.
\]
For $|\eta|\le\varrho^2/2$ and $i\ne j$,
\[
 |c_{ij}(\mu)|
 \left|\frac{\varrho^{i-j}\eta^j-\varrho^{j-i}\eta^i}{i-j}\right|
 \le2A_0(\varrho/R)^{i+j},\qquad
 |c_{ii}(\mu)\eta^i|\le A_0(\varrho^2/(2R^2))^i.
\]
The two geometric series converge, with sums independent of $\mu$.
This proves uniform convergence in $\mu$ and $\eta$ on this disk.
Outside the saddle chart, the branch endpoints are analytic in
$\eta$, $|\Inv_{\mu,P}|\ge c$ and the integration length is bounded
independently of $\mu,\eta$. Integrating the local convergent power
series gives an analytic remainder bounded by the same type of
constant.
This proves \eqref{eq:hc-area-expansion}. Differentiating yields
\[
 T_\mu(\eta)=\chi_\mu^{-1}\log(1/\eta)+O(1),\qquad
 T_\mu'(\eta)=-\frac1{\chi_\mu\eta}+O(\log(1/\eta)).
\]

\emph{3. The frequency.}
On a regular upper branch, define the lifted angle by
\[
 \theta(X,E)=\frac1{T_\mu(E-E_s)}
       \int_{X_0}^{X}\frac{\mathrm d\xi}
                {\Inv_{\mu,P}(\xi,P(\xi,E))}.
\]
Regard $\theta$ as a function of $(X,P)$ through
$E=\Inv_\mu(X,P)$. Then
\[
 \theta(X+1,E)=\theta(X,E)+1,\qquad
 \mathfrak D_{\Inv_\mu}\theta=\frac1{T_\mu(E-E_s)},\qquad
 \mathfrak D_{\Inv_\mu}T_\mu(E-E_s)=0.
\]
The defining coordinate series of $\mathcal B_{t,g}$ give, for sufficiently small $t$,
\[
 \theta\circ\mathcal B_{t,\Inv_\mu}
 =\theta+\frac{t}{T_\mu(E-E_s)}.
\]
The addition identity
$\mathcal B_{t+s,\Inv_\mu}
 =\mathcal B_{t,\Inv_\mu}\circ\mathcal B_{s,\Inv_\mu}$
extends this formula to $t=\mu$ by a finite partition of $[0,\mu]$.
Hence $\Omega_\mu=\mu/T_\mu$.
\end{proof}

Since $T_\mu'<0$ and $T_\mu(\eta)\to\infty$ as $\eta\downarrow0$,
for $\nu>0$ with $\mu/\nu$ sufficiently large there is a unique small
$\eta_*>0$ with $T_\mu(\eta_*)=\mu/\nu$. In particular,
\begin{equation}\label{eq:hc-frequency-gap}
 \Omega_\mu(I(E_s(\mu)+\eta_*))=\nu,\qquad
 \log(1/\eta_*)=\chi_\mu\frac{\mu}{\nu}+O(1).
\end{equation}

\subsection{Complex area--angle coordinates}

Retain $E_s$, $I(E)$ and $T_\mu$ from
Lemma~\ref{lem:hc-area-derivative}, and write
\[
 E_*=E_s+\eta,\qquad I_*=I(E_*),\qquad
 T(E):=T_\mu(E-E_s),\qquad T_*:=T(E_*).
\]
Fix a finite derivative order $k$. The constants in this subsection
may depend on $k$ and the fixed complex neighborhoods in
Lemma~\ref{lem:hc-comparison}; they are independent of $\mu,\eta$
for $0<\mu<\mu_0$, $0<\eta<\eta_0$. All bounds on a complex domain
hold at every point of that domain with the same constants.
For constants $c>0$ and $M\in\N$ chosen in the proof, put
\[
 \sigma=\frac{c}{T_*},\qquad r=c\eta^M,\qquad
 \mathscr D_j:=\{(\theta,I):
       |\operatorname{Im}\theta|<j\sigma,\ |I-I_*|<jr\},
 \quad j=1,2,3,4,6,8.
\]
The angular variable belongs to $\C/\Z$. A degree-one chart means
that its lift satisfies
$\mathcal C_\mu(\theta+1,I)=\mathcal C_\mu(\theta,I)+(1,0)$.
The real central circle is
$\{(\theta,I_*):\theta\in\R/\Z\}$.

\begin{Lemma}[Complex area--angle coordinates]\label{lem:sharp-area-angle}
For sufficiently small $\mu,\eta>0$, there is an exact analytic
 degree-one embedding $\mathcal C_\mu:\mathscr D_8\to\C^2/\Z$ such that
\[
 \Inv_\mu\circ\mathcal C_\mu=E(I),\qquad
 \mathcal C_\mu^{-1}\mathcal B_\mu\mathcal C_\mu
    =(\theta+\Omega_\mu(I),I),\qquad
 \Omega_\mu(I)=\frac{\mu}{T(E(I))},
\]
where the conjugacy is defined on $\mathscr D_2$.
There is a constant $B>0$, independent of $\mu,\eta$, such that
\begin{align*}
 \|\mathcal C_\mu\|_{C^k,\mathscr D_4}
 +\|\mathcal C_\mu^{-1}\|_{C^k,\mathcal C_\mu(\mathscr D_4)}
     &\le C\eta^{-B},\\
 \operatorname{dist}\bigl(\mathcal C_\mu(\overline{\mathscr D_4}),
                  \partial\mathcal C_\mu(\mathscr D_8)\bigr)
     &\ge c\eta^B.
\end{align*}
On the real central circle,
$\partial_\theta X\ge cT_*\sqrt\eta>0$ and
$\Omega_\mu'(I_*)\asymp\mu/(\eta T_*^3)$.
\end{Lemma}

\begin{proof}
\emph{1. Bounds on the real rotating branch.}
Solve $\Inv_{\mu,P}(X,\psi(X))=0$ and put
$W(X)=\Inv_\mu(X,\psi(X))$. Strict convexity in $P$ and the
nondegenerate maximum of $W$ give
\[
 \Inv_{\mu,P}(X,P(X,E))\asymp P(X,E)-\psi(X)
       \asymp\sqrt{E-E_s+E_s-W(X)},
 \qquad |W'(X)|\le C\sqrt{E_s-W(X)}.
\]
The last inequality follows from the quadratic expansion at the
maximum; outside a fixed neighborhood, $E_s-W\ge c>0$.
Since $|\Inv_{\mu,XP}|\le C$,
\[
 |\Inv_{\mu,X}(X,P)|
 \le |W'(X)|+C|P-\psi(X)|.
\]
Thus the vector field $V=(\Inv_{\mu,P},-\Inv_{\mu,X})$ satisfies
\begin{equation}\label{eq:sharp-cone}
 V_1\ge c\sqrt\eta,\qquad |V|\le C_0V_1
\end{equation}
on every real upper rotating circle with $|E-E_*|<\eta/4$.

\smallskip
\emph{2. The complex action parameter.}
Take the branch of the logarithm real at $E=E_*$ in
\eqref{eq:hc-area-expansion}. For $|E-E_*|<\eta/4$,
\[
 |T(E)-T_*|\le C\eta^{-1}|E-E_*|,\qquad
 |T^{(j)}(E)|\le C_j\eta^{-j}\quad(j\ge1),\qquad
 T_*\asymp\log(1/\eta).
\]
Define $I(E)$ by $I'(E)=T(E)$ and $I(E_*)=I_*$.
After reducing $\eta_0$, $|T(E)/T_*-1|<1/4$ on this disk.
Consequently
\[
 \operatorname{Re}\frac{I(E_2)-I(E_1)}{T_*(E_2-E_1)}
 =\operatorname{Re}\int_0^1
       \frac{T(E_1+t(E_2-E_1))}{T_*}\,\dd t\ge\frac34
\]
for $E_1\ne E_2$, so $I(E)$ is one-to-one.
For $|I-I_*|<\eta T_*/32$, set
\[
 H_I(E)=E+T_*^{-1}(I-I(E)).
\]
On $|E-E_*|\le\eta/8$,
\[
 |H_I'(E)|\le\tfrac14,\qquad
 |H_I(E)-E_*|
 \le\tfrac14|E-E_*|+T_*^{-1}|I-I_*|
 <\eta/16.
\]
The contraction theorem defines its unique fixed point $E(I)$.
Implicit differentiation and integration along a segment give
\[
 E'(I)=\frac1{T(E(I))},\qquad
 |E(I)-E_*|\le\frac{4|I-I_*|}{3T_*}.
\]
For example, writing $R_I=\eta T_*/32$, Cauchy's formula gives
\[
 \sup_{|I-I_*|<R_I/2}|E^{(j)}(I)|
 \le j!\,(R_I/2)^{-j}
       \sup_{|I-I_*|<R_I}|E(I)-E_*|
 \le j!\,(R_I/2)^{-j}\eta/8.
\]
This is a fixed power of $\eta^{-1}$ for every fixed $j$.

\smallskip
\emph{3. The angular strip and the parameter disk.}
Choose a section $X=X_0$ outside the saddle neighborhood.
Its upper intersection $z_0(E)$ is analytic on the level disk,
with $|z_0(E)|+|z_0'(E)|\le C$. Define
\[
 z(t,E)=\mathcal B_{t,\Inv_\mu}(z_0(E)),\qquad
 \partial_tz=V(z).
\]
For real $E$, one full rotation gives
\[
 z(t+T(E),E)=z(t,E)+(1,0).
\]
All these real circles lie in a fixed compact band. Choose a fixed
complex neighborhood of the band where $|V|+\|DV\|\le L$.
There is a fixed $h>0$ such that $z(t+is,E)$ is defined for
real $t,E$ and $|s|<h$. Along this complex time segment,
\[
 \frac{\dd}{\dd s}V(z(t+is,E))
       =iDV(z(t+is,E))V(z(t+is,E)).
\]
The integral equation and Gronwall's inequality yield
\[
 |V(z(t+is,E))-V(z(t,E))|
       \le(e^{L|s|}-1)|V(z(t,E))|.
\]
Choose $h$ so that $C_0(e^{Lh}-1)<1/2$.
By \eqref{eq:sharp-cone},
\[
 \operatorname{Re}\partial_tX(t+is,E)
       \ge\tfrac12\partial_tX(t,E)\ge c\sqrt\eta.
\]
In particular, $z(T_*\theta,E_*)$ is analytic for
$|\operatorname{Im}\theta|<h/T_*$ and
\[
 \operatorname{Re}\partial_\theta X(T_*\theta,E_*)
       \ge cT_*\sqrt\eta.
\]

To continue in $E$, let $\rho_0>0$ be a fixed radius such that
$V$ is analytic whenever the distance to the real band is less
than $\rho_0$. The variational equation gives
\[
 \|D_{z_0}\mathcal B_{t,\Inv_\mu}\|
       \le e^{Lt},\qquad 0\le t\le2T_*+2.
\]
Since $T_*\le C\log(1/\eta)$, choose a fixed $b_0$ with
$e^{L(2T_*+2)}\le C\eta^{-b_0}$. Continuation over intervals
of length $h/2$ then gives
\[
 |z(t,E)-z(t,E_*)|\le C\eta^{-b_0}|E-E_*|
\]
as long as $C\eta^{-b_0}|E-E_*|<\rho_0/2$.
Require also $2C\eta^{-1}|E-E_*|<h/4$. The function
\[
 Z(\theta,E)=z(T(E)\theta,E)
\]
is then defined for {$0<\operatorname{Re}\theta<3/2$ and}
$|\operatorname{Im}\theta|<h/(4T_*)$, because
\[
 |(T(E)-T_*)\theta|\le2C\eta^{-1}|E-E_*|<h/4.
\]
The derivative identities
\[
 \partial_EZ=(\partial_Ez)(T(E)\theta,E)+V(Z)T'(E)\theta,
 \qquad
 \partial_E\partial_\theta Z
       =T'(E)V(Z)+T(E)DV(Z)\partial_EZ
\]
imply, for a fixed $B_0\ge b_0+2$,
\[
 |Z(\theta,E)-Z(\theta,E_*)|
 +|\partial_\theta Z(\theta,E)-\partial_\theta Z(\theta,E_*)|
       \le C\eta^{-B_0}|E-E_*|.
\]
Choose an integer $M_0>B_0+1$ and a sufficiently small fixed
$c_E>0$. For $|E-E_*|<2c_E\eta^{M_0}$ the two continuation
conditions hold, and
\[
 C\eta^{-B_0}|E-E_*|\le\tfrac12cT_*\sqrt\eta,
 \qquad \operatorname{Re}\partial_\theta Z_1
       \ge\tfrac12cT_*\sqrt\eta.
\]
{On the overlap $0<\operatorname{Re}\theta<1/2$,
both $Z(\theta+1,E)$ and $Z(\theta,E)$ are defined by the flow.
For real $E$, the period identity extends from real to complex
$\theta$ on this overlap by the identity theorem. For each such
$\theta$, both sides are holomorphic in $E$ and agree for real $E$.
Consequently $Z(\theta+1,E)=Z(\theta,E)+(1,0)$ throughout the
overlap and the complex level disk. The integer translates of
$0<\operatorname{Re}\theta<3/2$ cover the full lifted strip, and
the identity makes their analytic extensions agree.}

The derivative bounds follow from Cauchy's integral formula on
explicit smaller domains. Namely, for $\ell=1,2$, $p,q\ge0$,
$|\operatorname{Im}\theta|<h/(8T_*)$, and
$|E-E_*|<c_E\eta^{M_0}$,
\[
 |\partial_\theta^p\partial_E^q
          (Z_\ell(\theta,E)-\delta_{1\ell}\theta)|
 \le C p!q!\left(\frac{8T_*}{h}\right)^p
                 (c_E\eta^{M_0})^{-q}.
\]
Here the two integration circles have radii $h/(8T_*)$ and
$c_E\eta^{M_0}$; periodicity permits the angular circle to cross
$\operatorname{Re}\theta=0,1$. In this calculation, ``loss'' means
reducing the domain by these radii to estimate derivatives.
Since $T_*\le C\eta^{-1}$, each fixed derivative is bounded by a
fixed power of $\eta^{-1}$.

Choose $M>M_0+1$ and reduce $c$ so that
\[
 8\sigma<\frac{h}{8T_*},\qquad
 \frac{32r}{3T_*}<c_E\eta^{M_0}.
\]
Then $\mathcal C_\mu(\theta,I)=Z(\theta,E(I))$ is analytic on
$\mathscr D_8$, with $\operatorname{Re}\partial_\theta X>0$.
For fixed complex $I$ and $\theta_1\ne\theta_2$,
\[
 \operatorname{Re}\frac{X(\theta_2,I)-X(\theta_1,I)}{\theta_2-\theta_1}
 =\operatorname{Re}\int_0^1\partial_\theta X(
       \theta_1+t(\theta_2-\theta_1),I)\,\dd t>0.
\]
Thus $\theta\mapsto X(\theta,I)$ is one-to-one on the lifted
strip. Equality of two chart images gives
$E(I_1)=E(I_2)$, then $I_1=I_2$ and $\theta_1=\theta_2$.
The lift relation gives injectivity on the complex cylinder.

\smallskip
\emph{4. Exactness and distances to the image boundary.}
Since $\partial_\theta\mathcal C_\mu=T(E)V$ and
$E'(I)=T(E)^{-1}$,
\[
 \det D\mathcal C_\mu
   =T(E)(\Inv_{\mu,X}X_I+\Inv_{\mu,P}P_I)
   =T(E)E'(I)=1.
\]
Hence $P\,\dd X-I\,\dd\theta$ is closed. Its integral along
$\theta\in[0,1]$ with fixed real $I$ is
\[
 \int_0^1P(\theta,I)X_\theta(\theta,I)\,\dd\theta-I
 =\int_{X_0}^{X_0+1}P(X,E)\,\dd X-I(E)=0.
\]
The same integral is analytic in $I$, so it is zero also for complex
$I$. Every closed path in $\mathscr D_8$ is deformable to an integer
multiple of this angular loop. Its integral is consequently zero,
and integrating the closed form along paths defines an analytic
primitive. This proves exactness.

For completeness, if $j<l$, Cauchy's formula applied componentwise
on disks of radii $(l-j)\sigma$ and $(l-j)r$ gives
\[
 \|\partial_\theta^p\partial_I^q
       (\mathcal C_\mu-(\theta,0))\|_{\mathscr D_j}
 \le \frac{p!q!\,
       \|\mathcal C_\mu-(\theta,0)\|_{\mathscr D_l}}
              {((l-j)\sigma)^p((l-j)r)^q}.
\]
Use $j=6$, $l=8$ and the preceding bounds. Since the Jacobian
has determinant one, its inverse is its {adjugate}.
In particular, both first derivatives have polynomial bounds.
For the second inverse derivative, with $A=D\mathcal C_\mu$,
\[
 D^2\mathcal C_\mu^{-1}[u,v]
 =-A^{-1}D^2\mathcal C_\mu[A^{-1}u,A^{-1}v].
\]
Repeated differentiation gives the same conclusion for any fixed
order $k$. Choose a fixed $B_1\ge M+1$ so that on
$\mathscr D_6$ and its image,
\[
 \|D\mathcal C_\mu^{-1}\|+\|D^2\mathcal C_\mu\|
       \le C\eta^{-B_1}.
\]
For $z_0\in\overline{\mathscr D_4}$, put
\[
 \rho_s=c\eta^{2B_1+2}<\min(\sigma,r),\qquad
 L_0=D\mathcal C_\mu(z_0)^{-1},\qquad
 \mathcal T_y(z)=z+L_0(y-\mathcal C_\mu(z)).
\]
All balls below use local lifts and the maximum norm in $\C^2$.
On $|z-z_0|\le\rho_s$,
\[
 \|D\mathcal T_y(z)\|
 =\|L_0(D\mathcal C_\mu(z_0)-D\mathcal C_\mu(z))\|
 \le C\eta^{-2B_1}\rho_s<\tfrac12.
\]
Choose $\rho_t=c\eta^{B_1}\rho_s$ with $c$ small enough.
If $|y-\mathcal C_\mu(z_0)|<\rho_t$, then
\[
 |\mathcal T_y(z_0)-z_0|\le\rho_s/2,\qquad
 |\mathcal T_y(z)-z_0|
 \le\tfrac12|z-z_0|+\rho_s/2\le\rho_s.
\]
The contraction theorem gives a preimage of every such $y$ in the
source ball. Thus the ball of radius $\rho_t$ in $(X,P)$ coordinates
about $\mathcal C_\mu(z_0)$ lies in $\mathcal C_\mu(\mathscr D_8)$.
These local inverses agree by injectivity. Increasing $B$ and
reducing $c$ proves the asserted derivative and distance bounds.
Applying the same argument with $z_0\in\overline{\mathscr D_3}$
and source balls contained in $\mathscr D_4$ also gives
\[
 \operatorname{dist}\bigl(\mathcal C_\mu(\overline{\mathscr D_3}),
                \partial\mathcal C_\mu(\mathscr D_4)\bigr)
          \ge c\eta^B.
\]

\smallskip
\emph{5. The comparison map and perturbations.}
The addition identity for $\mathcal B_{t,\Inv_\mu}$ gives an
angle increment $\mu/T(E)$ at $t=\mu$. On the real central circle,
\[
 \partial_\theta X=T_*\Inv_{\mu,P}\ge cT_*\sqrt\eta,
 \qquad
 \Omega_\mu'(I_*)=-\frac{\mu T'(E_*)}{T_*^3}
       \asymp\frac{\mu}{\eta T_*^3}.
\]
For $|I-I_*|<8r$, integration of $\Omega_\mu'$ gives
\[
 |\Omega_\mu(I)-\Omega_\mu(I_*)|
       \le\frac{C\mu r}{\eta T_*^3}<\sigma/2.
\]
The number $\Omega_\mu(I_*)$ is real. Thus
$(\theta,I)\mapsto(\theta+\Omega_\mu(I),I)$ sends
$\mathscr D_2$ into $\mathscr D_3$.
The conjugacy holds first for real $(\theta,I)$ and then for complex
$(\theta,I)$ by applying the one-variable identity theorem first
in $\theta$ and then in $I$.
Finally, if an analytic map $F$ satisfies
\[
 \|F-\mathcal B_\mu\|_{\mathcal C_\mu(\mathscr D_2)}
       <\tfrac12c\eta^B,
\]
the last distance estimate gives
$F(\mathcal C_\mu(\mathscr D_2))
\subset\mathcal C_\mu(\mathscr D_4)$.
The map $\mathcal C_\mu^{-1}F\mathcal C_\mu$ is defined on
$\mathscr D_2$, and the inverse coordinate derivatives on its image
satisfy the bounds proved in Step~4.
\end{proof}

\subsection{Analytic iteration and the average obstructions}

For $s>0$, write
\[
 \mathbb T_s:=\{\theta\in\C/\Z:|\operatorname{Im}\theta|<s\},\qquad
 \langle h\rangle:=\int_0^1h(\theta)\,\dd\theta.
\]
Here $h$ is a $1$-periodic scalar function, or a vector or matrix of
such functions; the average is taken component by component.
The norm $\|\cdot\|_s$ is the supremum norm on $\mathbb T_s$, and
$\|\cdot\|_{C^1,s}$ is the sum of the supremum norms of the function
and its first derivative. For a cylinder map, its zeroth-order norm
means the norm of the periodic displacement of a lift from the identity.

Assume $0<\beta<1/2$, let $A\geq1$, and set
$\vartheta:=\beta/(1-\beta)<1$. Suppose $\nu\in\R\setminus\Q$
and $q\in\mathbb N$ satisfy
\[
 \|k\nu\|_{\mathbb Z}^{-1}\leq A\exp\bigl(A(q|k|)^\beta\bigr)
 \qquad(k\in\mathbb Z\setminus\{0\}).
\]
Let
\[
 \mathcal G_0(\theta,J):=(\theta+\Omega(I_*+r_*J),J),\qquad
 \Omega(I_*)=\nu,\qquad a_*:=r_*\Omega'(I_*)\ne0.
\]
Assume that $\mathcal G$ is a real analytic area-preserving cylinder
map and that both maps extend holomorphically to
$\mathcal U:=\mathbb T_{2\delta}\times\{|J|<2\}$, $0<\delta<1$.
Choose $\mathcal Q\geq2$ with
\[
 \max\{\delta^{-1},|a_*|^{-1},
          \|\mathcal G_0\|_{C^2,\mathcal U},
          \|\mathcal G\|_{C^2,\mathcal U}\}\le\mathcal Q.
\]
Set
\[
 K_0(\theta):=(\theta,0),\qquad
 \epsilon_0:=\|\mathcal G\circ K_0-K_0\circ R_\nu\|_\delta,
 \qquad \Lambda:=\log\mathcal Q+(q/\delta)^\vartheta.
\]
In this subsection, constants denoted by $c$ or $C$ depend only on
$A$ and $\beta$. They are independent of $q,\nu,\delta,\mathcal Q$
and the iteration index. The bounds involving $\mathcal Q$ specify
all dependence on the maps and the twist.

\begin{Lemma}[Analytic iteration]\label{hc-newton}
Under the preceding assumptions, there are $c>0$ and $C_1>C_2>0$
such that
\[
 \|\mathcal G-\mathcal G_0\|_{C^1,\mathcal U}\le c\mathcal Q^{-3},
 \qquad \epsilon_0 e^{C_1\Lambda}<c
\]
imply the existence of a real analytic degree-one embedding $K$ on
$\mathbb T_{\delta/2}$ and $\zeta\in\R$ with
\[
 \mathcal G K=KR_\nu+(0,\zeta),\qquad
 \|K-K_0\|_{C^1,\delta/2}\le\epsilon_0e^{C_2\Lambda}<c\mathcal Q^{-3}.
\]
If $\mathcal G$ is exact for $J\,\dd\theta$, then $\zeta=0$.
\end{Lemma}
\begin{proof}
If $\epsilon_0=0$, take $K=K_0$ and $\zeta=0$. Assume
$\epsilon_0>0$.

\noindent\emph{1. Solving the cohomological equation.}
Put $D_\nu u:=u\circ R_\nu-u$. For $\langle f\rangle=0$, the
zero-average solution of $D_\nu u=f$ is given by
\[
 f_k:=\int_0^1f(x)e^{-2\pi ikx}\,\dd x,\qquad
 u_0=0,\qquad u_k=\frac{f_k}{e^{2\pi ik\nu}-1}\quad(k\ne0).
\]
For $0<s'<s$ and $k>0$, shift the integration segment to
$[0,1]-is'$; the two vertical integrals cancel by periodicity.
For $k<0$, shift it to $[0,1]+is'$. This gives
\[
 |f_k|\leq\|f\|_s e^{-2\pi|k|s'}.
\]
Letting $s'\nearrow s$ proves the same bound with $s$ in place of
$s'$. Since $|e^{2\pi it}-1|\geq4\|t\|_{\mathbb Z}$, for
$0<\sigma<\min(s,1)$ we have
\[
 \|u\|_{s-\sigma}
 \leq\frac A4\|f\|_s
       \sum_{k\ne0}\exp\bigl(Aq^\beta|k|^\beta-2\pi\sigma|k|\bigr).
\]
To estimate this sum, differentiate
$\phi(t):=Aq^\beta t^\beta-\pi\sigma t$ for $t>0$. Its maximum is
attained at
\[
 t_0=\left(\frac{A\beta q^\beta}{\pi\sigma}\right)^{1/(1-\beta)},
 \qquad
 \phi(t_0)=(1-\beta)A
       \left(\frac{A\beta}{\pi}\right)^\vartheta
       (q/\sigma)^\vartheta.
\]
Also,
\[
 \sum_{k\ne0}e^{-\pi\sigma|k|}
 =\frac{2}{e^{\pi\sigma}-1}\leq\frac{2}{\pi\sigma}.
\]
Consequently, after choosing $C$ large enough, the function
\[
 L(\sigma):=C\sigma^{-1}
       \exp\!\bigl(C(q/\sigma)^\vartheta\bigr)\geq1
\]
satisfies
\[
 \|D_\nu^{-1}f\|_{s-\sigma}\leq L(\sigma)\|f\|_s.
\]
The loss $\sigma$ is the decrease of the strip width from $s$ to
$s-\sigma$.

\smallskip
\noindent\emph{2. Reducing the linear equation to two scalar equations.}
At the current embedding $K$ and real number $\zeta$, write
$h_+:=h\circ R_\nu$ and define
\[
 e:=\mathcal G(K)-K_+-(0,\zeta),\qquad v:=K',\qquad
 \mathsf M:=\left(v,\frac{\mathsf Jv}{v^{\mathsf T}v}\right),\qquad
 \mathsf J(v_1,v_2):=(-v_2,v_1).
\]
The transpose $v^{\mathsf T}$ has no complex conjugation. If
$\|K-K_0\|_{C^1,s}<c\mathcal Q^{-3}$, then
\[
 |v^{\mathsf T}v-1|<\tfrac12,\qquad
 \det\mathsf M=1,\qquad
 \|\mathsf M\|_s+\|\mathsf M^{-1}\|_s\leq C.
\]
Put
\[
 \mathsf B:=\mathsf M_+^{-1}D\mathcal G(K)\mathsf M,\qquad
 S:=\mathsf B_{12},\qquad f:=\mathsf M_+^{-1}e'.
\]
The exact frame identity in Proposition~\ref{hc-external-frame}
corresponds to $e=0$. For nonzero $e$, differentiation gives
\[
 D\mathcal G(K)v=v_++e',\qquad
 \mathsf B\binom10=\binom{1+f_1}{f_2}.
\]
Area preservation gives $\det D\mathcal G=1$ on the real domain.
Continuation in each variable gives this equality on $\mathcal U$,
so $\det\mathsf B=1$. Thus
$(1+f_1)\mathsf B_{22}-Sf_2=1$, and
\[
 \mathsf B=\begin{pmatrix}1&S\\0&1\end{pmatrix}+E,\qquad
 E=\begin{pmatrix}
 f_1&0\\ f_2&\displaystyle\frac{Sf_2-f_1}{1+f_1}
 \end{pmatrix}.
\]
Writing $\epsilon:=\|e\|_s$, the strip derivative bound in
\eqref{hc-cauchy-derivatives} gives
\[
 \|e'\|_{s-\sigma}\leq\sigma^{-1}\epsilon,\qquad
 \|f\|_{s-\sigma}\leq C\sigma^{-1}\epsilon,\qquad
 \|S\|_s\leq C\mathcal Q.
\]
If $C\mathcal Q\sigma^{-1}\epsilon<1/2$, then $|1+f_1|\geq1/2$
and
\[
 \|E\|_{s-\sigma}\leq C\mathcal Q\sigma^{-1}\epsilon.
\]
Finally,
\[
 D\mathcal G_0(K_0)=\begin{pmatrix}1&a_*\\0&1\end{pmatrix},\qquad
 \|S-a_*\|_s
 \leq C\mathcal Q\|K-K_0\|_{C^1,s}
       +C\|\mathcal G-\mathcal G_0\|_{C^1,\mathcal U}.
\]
For sufficiently small $c$, the right-hand side is at most
$|a_*|/2$, so
$|\langle S\rangle|\geq|a_*|/2\geq(2\mathcal Q)^{-1}$.

\smallskip
\noindent\emph{3. Choosing the two averages and estimating the correction.}
Put
\[
 g:=\mathsf M_+^{-1}e,\qquad
 \mathbf b:=\mathsf M_+^{-1}\binom01.
\]
Writing $K=(K_1,K_2)$, the inverse of $\mathsf M_+$ gives
$\mathbf b_2=K_{1,+}'$. Since $K_1(\theta+1)=K_1(\theta)+1$,
\[
 \langle\mathbf b_2\rangle=1,\qquad
 \|g\|_s\leq C\epsilon,\qquad\|\mathbf b\|_s\leq C.
\]
Choose $\Delta K:=\mathsf M W$ and $\Delta\zeta$ by
\begin{align*}
 \Delta\zeta&:=\langle g_2\rangle,&
 D_\nu W_2^0&=g_2-\mathbf b_2\Delta\zeta,&
 \langle W_2^0\rangle&=0,\\
 c_2&:=\frac{-\langle g_1\rangle-\langle SW_2^0\rangle
                   +\langle\mathbf b_1\rangle\Delta\zeta}
                  {\langle S\rangle},&
 W_2&:=W_2^0+c_2,\\
 D_\nu W_1&=g_1+SW_2-\mathbf b_1\Delta\zeta,&
 \langle W_1\rangle&=0.
\end{align*}
The two right-hand sides have zero average:
\[
 \langle g_2-\mathbf b_2\Delta\zeta\rangle=0,\qquad
 \langle g_1+SW_2-\mathbf b_1\Delta\zeta\rangle=0.
\]
Step~1 first gives
\[
 |\Delta\zeta|\leq C\epsilon,\qquad
 \|W_2^0\|_{s-\sigma}\leq CL(\sigma)\epsilon.
\]
Using $|\langle S\rangle|^{-1}\leq2\mathcal Q$ and
$\|S\|_s\leq C\mathcal Q$, we obtain
\[
 |c_2|+\|W_2\|_{s-\sigma}
 \leq C\mathcal Q^2L(\sigma)\epsilon.
\]
The second application of Step~1 gives
\[
 \|W_1\|_{s-2\sigma}
 \leq CL(\sigma)
       \bigl[\epsilon+\mathcal Q\|W_2\|_{s-\sigma}\bigr]
 \leq C\mathcal Q^3L(\sigma)^2\epsilon.
\]
Hence
\[
 \|\Delta K\|_{s-2\sigma}
 \leq C\mathcal Q^3L(\sigma)^2\epsilon,
 \qquad
 \|\Delta K\|_{C^1,s-3\sigma}
 \leq C\mathcal Q^3\sigma^{-1}L(\sigma)^2\epsilon.
\]
Here the last decrease of the strip width is used only to apply
\eqref{hc-cauchy-derivatives} to the product $\mathsf M W$.

Set $K_{\rm new}:=K+\Delta K$ and
$\zeta_{\rm new}:=\zeta+\Delta\zeta$. The equations defining $W$
are equivalent to
\[
 \begin{pmatrix}1&S\\0&1\end{pmatrix}W-W_+
       -\mathbf b\Delta\zeta=-g.
\]
Taylor's formula with integral remainder now gives
\[
 e_{\rm new}=\mathsf M_+EW+
 \int_0^1(1-t)D^2\mathcal G(K+t\Delta K)
                         [\Delta K,\Delta K]\,\dd t.
\]
Provided $K+t\Delta K\in\mathcal U$ for $0\leq t\leq1$, the two
terms satisfy, on $\mathbb T_{s-3\sigma}$,
\begin{align*}
 \|\mathsf M_+EW\|
 &\leq C\mathcal Q^4\sigma^{-1}L(\sigma)^2\epsilon^2,\\
 \left\|\int_0^1(1-t)D^2\mathcal G(K+t\Delta K)
                         [\Delta K,\Delta K]\,\dd t\right\|
 &\leq C\mathcal Q^7L(\sigma)^4\epsilon^2.
\end{align*}
Choose
\[
 \mathfrak L(\sigma):=C\mathcal Q^8\sigma^{-2}L(\sigma)^4.
\]
Increasing $C$ if necessary, all estimates needed for one step are
\[
 \|\Delta K\|_{C^1,s-3\sigma}\leq\mathfrak L(\sigma)\epsilon,
 \quad |\Delta\zeta|\leq C\epsilon,
 \quad \|e_{\rm new}\|_{s-3\sigma}
                  \leq\mathfrak L(\sigma)\epsilon^2,
\]
with
\[
 \mathfrak L(\sigma)\geq C\mathcal Q\sigma^{-1},\qquad
 \log\mathfrak L(\sigma)
 \leq C\bigl[\log\mathcal Q+|\log\sigma|
                         +(q/\sigma)^\vartheta\bigr].
\]
Thus the three decreases of the strip width are, in order, the
solution for $W_2^0$, the solution for $W_1$, and differentiation of
$\Delta K$.

\smallskip
\noindent\emph{4. Convergence.}
Start with $K_0$, $\zeta_0=0$, and put
\[
 \sigma_n:=\delta2^{-n}/24,\qquad s_0:=\delta,\qquad
 s_{n+1}:=s_n-3\sigma_n,\qquad s_n\searrow3\delta/4.
\]
Choose positive numbers $L_n\geq\log\mathfrak L(\sigma_n)$ such that
\[
 L_n\leq C\left[\log\mathcal Q+n
                   +(q/\delta)^\vartheta2^{\vartheta n}\right].
\]
For $\mathfrak S:=\sum_{n\geq0}2^{-n-1}L_n$, the geometric sums
\[
 \sum_{n\geq0}2^{-n-1}=1,\quad
 \sum_{n\geq0}n2^{-n-1}=1,\quad
 \sum_{n\geq0}2^{-n-1}2^{\vartheta n}
                 =\frac{1}{2(1-2^{\vartheta-1})}
\]
give
\[
 \mathfrak S\leq C\Lambda,\qquad L_n\leq2^{n+1}\mathfrak S.
\]
The last sum is finite because $\vartheta<1$.
Writing $\epsilon_n:=\|e_n\|_{s_n}$ and
$H:=\log\epsilon_0^{-1}$, induction in the quadratic estimate gives
\[
 \epsilon_n\leq
 \exp\!\left[-2^nH+
       \sum_{j=0}^{n-1}2^{n-j-1}L_j\right]
 \leq\exp\bigl[-2^n(H-\mathfrak S)\bigr].
\]
Choose $D\geq C+3\log\mathcal Q$ and assume $H\geq4\mathfrak S+D$.
For $n=0$, $L_0\leq2\mathfrak S$. For $n\geq1$,
\begin{align*}
 L_n+\log\epsilon_n
 &\leq-2^nH+3\cdot2^n\mathfrak S\\
 &\leq-H+2\mathfrak S-(2^n-1)D.
\end{align*}
It follows that
\[
 \mathfrak L(\sigma_n)\epsilon_n
 \leq\epsilon_0e^{2\mathfrak S}e^{-(2^n-1)D},\qquad
 \sum_{n\geq0}\|\Delta K_n\|_{C^1,s_{n+1}}
 \leq C\epsilon_0e^{2\mathfrak S}<c\mathcal Q^{-3}.
\]
In particular,
$C\mathcal Q\sigma_n^{-1}\epsilon_n<1/2$ at every step.
Choose $C_2$ so that $2\mathfrak S+\log C\leq C_2\Lambda$, and then
choose $C_1>C_2+4$ so that $C_1\Lambda\geq4\mathfrak S+D$.
The hypothesis $\epsilon_0e^{C_1\Lambda}<c$ implies these conditions
and the required bound on the sum of the corrections.

To check the domains, for $\theta\in\mathbb T_{s_{n+1}}$ and
$0\leq t\leq1$, this sum gives
\[
 |\operatorname{Im}(K_n+t\Delta K_n)_1|
 \leq\delta+c\mathcal Q^{-3}<\tfrac32\delta,
 \qquad |(K_n+t\Delta K_n)_2|<1,
\]
since $\mathcal Q\geq\delta^{-1}$. Thus each Taylor segment lies in
$\mathcal U$. The same sum keeps $K_n$ within $c\mathcal Q^{-3}$ of
$K_0$ in $C^1$, as required for the frame and twist estimates.
This proves the induction.

The series of corrections converges in $C^1$ on
$\mathbb T_{\delta/2}$, and $\sum_n|\Delta\zeta_n|<\infty$.
The Fourier equations and all averages preserve reality on the real
axis. The limits $K$ and $\zeta$ are real and satisfy the stated equation.
Writing $K_1(\theta)=\theta+u(\theta)$, we have $u$ periodic and
$\|u'\|_{\delta/2}<1/2$. Integration along the segment between two
points of the covering strip shows that
\[
 |K_1(z)-K_1(w)|\geq\tfrac12|z-w|.
\]
Together with $K_1(\theta+1)=K_1(\theta)+1$, this proves that $K$
is a degree-one embedding. The bound on the sum of the corrections
gives the asserted $C^1$ estimate.

\smallskip
\noindent\emph{5. Removing the vertical translation by exactness.}
Exactness for $J\,\dd\theta$ means that
$\mathcal G^*(J\,\dd\theta)-J\,\dd\theta$ is the differential of
a single-valued function on the cylinder. Its integral over the
closed curve $K([0,1])$ is zero. Using
$\mathcal G(K)=K_++(0,\zeta)$, we obtain
\begin{align*}
 0&=\int_0^1(K_{2,+}+\zeta)K_{1,+}'\,\dd\theta
       -\int_0^1K_2K_1'\,\dd\theta\\
  &=\zeta\int_0^1K_{1,+}'\,\dd\theta=\zeta.
\end{align*}
The two integrals involving $K_2K_1'$ agree by periodicity, and the
last integral equals $1$ by the degree-one property.
\end{proof}

\subsection{Proof of Theorems~\ref{Mthm2} and~\ref{Mthm21}}

\textbf{1. Parameters and level gap.}
Set
\begin{equation}\label{hc-proof-scales}
 Q_m:=q_m^\beta,\qquad
 \mu_m:=\sqrt{2\pi\delta_{q_m}}=\sqrt{2\pi}e^{-dQ_m/2},
 \qquad \nu_m:=|\omega_m|.
\end{equation}
The continued-fraction estimates and \eqref{hc-parameter-window} give
\[
 \nu_m=q_m^{-1}e^{-bQ_m+o(1)},\qquad
 \frac{\mu_m}{\nu_m}
 =\sqrt{2\pi}\,q_m e^{(b-d/2)Q_m+o(1)}\longrightarrow\infty,
 \qquad b<d<2b.
\]
Let $\eta_m$ be the gap given by
Lemma~\ref{lem:hc-area-derivative}, and put
\[
 I_*:=I(E_s(\mu_m)+\eta_m),\qquad
 \sigma_m:=\frac{c}{T_{\mu_m}(\eta_m)},\qquad
 r_m^*:=c\eta_m^M.
\]
Here $M$ is fixed large enough for the coordinate and derivative
bounds in Lemma~\ref{lem:sharp-area-angle}. We have
\begin{equation}\label{hc-proof-gap}
 \Omega_{\mu_m}(I_*)=\nu_m,\qquad
 \log(1/\eta_m)\asymp\frac{\mu_m}{\nu_m},\qquad
 \sigma_m\asymp\frac{\nu_m}{\mu_m}.
\end{equation}
The hypotheses on $\alpha,\varepsilon$ imply
$0<\beta<1/3$. With $\vartheta=\beta/(1-\beta)<1$, put
\[
 \Xi_m:=\log(1/\eta_m)+\log(1/\mu_m)+\log(1/\sigma_m)
              +(q_m/\sigma_m)^\vartheta,
 \qquad
 \kappa:=\frac{d/2-\beta b}{1-\beta}>0.
\]
The positivity of $\kappa$ follows from $d>b$ and $\beta<1/2$.
By \eqref{hc-proof-gap},
\begin{align*}
 \mu_m\log(1/\eta_m)
   &=O\!\left(q_m e^{-(d-b)Q_m}\right)\longrightarrow0,\\
 \mu_m(q_m/\sigma_m)^\vartheta
   &=O\!\left(q_m^{2\vartheta}e^{-\kappa Q_m}\right)
        \longrightarrow0,\\
 \mu_m\bigl[\log(1/\mu_m)+\log(1/\sigma_m)\bigr]
   &=O\!\left((Q_m+\log q_m)e^{-dQ_m/2}\right)\longrightarrow0.
\end{align*}
Thus, for every fixed $C_*>0$,
\begin{equation}\label{hc-deciding-estimate}
 \Xi_m=o(\mu_m^{-1}),\qquad
 -c\mu_m^{-1}+C_*\Xi_m\longrightarrow-\infty.
\end{equation}
In particular, for every fixed $B>0$,
\[
 e^{-c/\mu_m}\eta_m^{-B}\mu_m^{-B}\sigma_m^{-B}
          \longrightarrow0.
\]

\textbf{2. Coordinates and iteration.}
Write $\mu=\mu_m$ and define
\[
 \mathcal N_m(\theta,J):=(\theta,I_*+r_m^*J),\qquad
 \mathcal S_\mu(X,P):=(X,\mu P),\qquad
 \Psi_m:=\mathcal S_\mu\mathcal K_{\mu/2}
                   \mathcal C_\mu\mathcal N_m.
\]
The map $\mathcal C_\mu$ is the chart of
Lemma~\ref{lem:sharp-area-angle}. The nested image margins in its proof and
\eqref{hc-deciding-estimate} give
$Ce^{-c/\mu_m}<c\eta_m^B$ for the required fixed $B$ and all large
$m$. Hence the following conjugates are defined on a complex
neighborhood of $|\operatorname{Im}\theta|<2\sigma_m$, $|J|<2$:
\begin{align*}
 \mathcal G&:=\mathcal N_m^{-1}\mathcal C_\mu^{-1}
                \widehat M_\mu\mathcal C_\mu\mathcal N_m
             =\Psi_m^{-1}f_{q_m,a}\Psi_m,\\
 \mathcal G_0&:=\mathcal N_m^{-1}\mathcal C_\mu^{-1}
                 \mathcal B_\mu\mathcal C_\mu\mathcal N_m
       =\bigl(\theta+\Omega_\mu(I_*+r_m^*J),J\bigr).
\end{align*}
Their central frequency and twist satisfy
\[
 \Omega_\mu(I_*)=\nu_m,\qquad
 a_*:=r_m^*\Omega_\mu'(I_*)
    \asymp\frac{\mu_m\eta_m^{M-1}}{\log^3(1/\eta_m)}>0.
\]
{
Apply Lemma~\ref{lem:sharp-area-angle} with $k=3$.
For $z\in\mathscr D_2$, the point $\mathcal B_\mu\mathcal C_\mu(z)$
belongs to $\mathcal C_\mu(\mathscr D_3)$. The distance estimate in
that lemma and \eqref{hc-deciding-estimate} place the full segment
from this point to $\widehat M_\mu\mathcal C_\mu(z)$ inside
$\mathcal C_\mu(\mathscr D_4)$.
Integrate $D\mathcal C_\mu^{-1}$ along the segment.
After up to two differentiations, each term contains a derivative
of order at most two of
$(\widehat M_\mu-\mathcal B_\mu)\circ\mathcal C_\mu$, possibly of
order zero. The other factors involve at most three derivatives
of the inverse chart and at most two of the direct chart and maps.
The affine rescaling contributes at most a factor $(r_m^*)^{-1}$
from its inverse. Thus \eqref{eq:hc-comparison-estimates} and the
coordinate bounds yield, for a fixed $B_0$,}
\begin{equation}\label{hc-conjugated-error}
 \|\mathcal G-\mathcal G_0\|_{C^2}
     \le E_m:=Ce^{-c/\mu_m}\eta_m^{-B_0}\mu_m^{-B_0},
 \qquad
 \epsilon_0:=\|\mathcal G K_0-K_0R_{\nu_m}\|_{\sigma_m}\le E_m,
\end{equation}
where $K_0(\theta)=(\theta,0)$ and the map norms are taken on
$\mathcal U_m:=\mathbb T_{2\sigma_m}\times\{|J|<2\}$. Choose the bound $\mathcal Q_m$
of Lemma~\ref{hc-newton} to include $\|\Psi_m\|_{C^2}$ and
$(b_m^*)^{-1}$, where
\[
 b_m^*:=\min_{\theta\in\R}
           \partial_\theta(\Psi_m)_1(\theta,0)
       \ge c\sqrt{\eta_m}\log(1/\eta_m)>0.
\]
The half-kick $\mathcal K_{\mu/2}$ leaves $X$ unchanged. The chart
bounds and the formula for $a_*$ give, for a fixed $B_1$,
\begin{equation}\label{hc-geometric-bound}
 \mathcal Q_m\le C\eta_m^{-B_1}\mu_m^{-B_1}\sigma_m^{-B_1},
 \qquad \log\mathcal Q_m\le C\Xi_m.
\end{equation}
Lemma~\ref{hc-arithmetic} supplies the small-divisor hypothesis at
$\nu=\nu_m$, $q=q_m$. Equations
\eqref{hc-conjugated-error}--\eqref{hc-geometric-bound} yield
\begin{align*}
 \|\mathcal G-\mathcal G_0\|_{C^1}\mathcal Q_m^3
      &\le \exp(-c/\mu_m+C\Xi_m)\longrightarrow0,\\
 \epsilon_0\exp\!\left(C[\log\mathcal Q_m
                      +(q_m/\sigma_m)^\vartheta]\right)
      &\le \exp(-c/\mu_m+C'\Xi_m)\longrightarrow0.
\end{align*}
The two hypotheses of Lemma~\ref{hc-newton} hold for all large $m$.

{
The shears $\mathcal K_z,\mathcal D_z$ are exact for $P\,\dd X$.
Conjugation by the exact chart $\mathcal C_\mu$ therefore gives an
exact map for $I\,\dd\theta$. Its pullback under $\mathcal N_m$
makes $\mathcal G$ exact for $(I_*+r_m^*J)\,\dd\theta$.
Subtracting}
\[
 \mathcal G^*(I_*\,\dd\theta)-I_*\,\dd\theta
       =\dd\bigl(I_*(\mathcal G_1-\theta)\bigr)
\]
{and dividing by $r_m^*>0$ shows that
$\mathcal G^*(J\,\dd\theta)-J\,\dd\theta$ is exact.}
Lemma~\ref{hc-newton} gives an analytic degree-one embedding $K$ with
\[
 \mathcal G K=KR_{\nu_m},\qquad
 \|K-K_0\|_{C^1}\le\exp(-c/\mu_m+C\Xi_m).
\]

\textbf{3. Graph property and signed rotation.}
Set $k_+:=\Psi_m K$. The chain rule and
\eqref{hc-deciding-estimate} give
\[
 \|k_+-\Psi_mK_0\|_{C^1}
 \le C\mathcal Q_m\|K-K_0\|_{C^1}
 <\frac{1}{2\mathcal Q_m}\le\frac{b_m^*}{2},\qquad
 (k_+)_1'\ge\frac{b_m^*}{2}>0.
\]
Also $k_+(t+1)=k_+(t)+(1,0)$ and
$f_{q_m,a}k_+=k_+R_{\nu_m}$. Its image is an analytic essential
graph; see Figure~\ref{hc-conjugacy-diagram}.
\begin{figure}[htbp]
\centering
\[
\begin{CD}
 \T @>{R_{\nu_m}}>> \T\\
 @V{K}VV @VV{K}V\\
 K(\T) @>{\mathcal G}>> K(\T)\\
 @V{\Psi_m}VV @VV{\Psi_m}V\\
 k_+(\T) @>{f_{q_m,a}}>> k_+(\T)
\end{CD}
\]
\caption{The constructed conjugacy.
}
\label{hc-conjugacy-diagram}
\end{figure}

If $\omega_m<0$, put
$\mathcal R(x,y):=(-x,-y)$ and $k_-(t):=\mathcal Rk_+(-t)$.
Since $\mathcal Rf_{q_m,a}=f_{q_m,a}\mathcal R$,
\[
 f_{q_m,a}k_-(t)
   =\mathcal Rk_+(-t+\nu_m)=k_-(t-\nu_m),\qquad
 (k_-)_1'(t)=(k_+)_1'(-t)>0.
\]
Let $k_m=k_+$ if $\omega_m>0$ and $k_m=k_-$ otherwise. Then
$f_{q_m,a}k_m=k_mR_{\omega_m}$.
Together with \eqref{capN}, this proves Theorem~\ref{Mthm2}.

\textbf{4. Finite covering.}
Write $k_m=(X_m,Y_m)$ and set
\begin{equation}\label{hc-cover-embedding}
 \bar k_m(t):=\left(\frac{X_m(q_mt)}{q_m},
                      \frac{p_m+Y_m(q_mt)}{q_m}\right).
\end{equation}
For $A_m(x,y):=(q_mx,q_my-p_m)$ and
$T_p(X,Y):=(X+p,Y)$, direct substitution gives
\[
 A_m\bar f_{m,a}=T_{p_m}f_{q_m,a}A_m,\qquad
 A_m\bar k_m(t)=k_m(q_mt).
\]
Using $q_m\bar\omega=p_m+\omega_m$ and
$k_m(t+1)=k_m(t)+(1,0)$, we obtain
\begin{align*}
 A_m\bar f_{m,a}\bar k_m(t)
 &=T_{p_m}k_m(q_mt+\omega_m)\\
 &=k_m(q_mt+q_m\bar\omega)
  =A_m\bar k_m(t+\bar\omega).
\end{align*}
The map $A_m$ is injective on $\R^2$. Hence, as real lifts,
\[
 \bar f_{m,a}\bar k_m=\bar k_mR_{\bar\omega},\qquad
 \bar k_m(t+1)=\bar k_m(t)+(1,0),\qquad
 (\bar k_m)_1'(t)=X_m'(q_mt)>0.
\]
This proves Theorem~\ref{Mthm21}. {Together with
Theorem~\ref{Mthm1}, Lemmas~\ref{Herm} and~\ref{arh}, and the
fixed-width estimates \eqref{barxx} and \eqref{covered-bump-norm},
this completes the proof of Corollary~\ref{Mthm3}.}

\section{\sc AI declaration}
{
The authors used OpenAI's ChatGPT to assist with proofreading,
checking estimates, and revising the exposition and proofs. The authors
take responsibility for the manuscript.
}

\appendix

\section{\sc The arithmetics of the rotation number $\bar{\omega}$}\label{rerA0}
We retain the convergents $p_n/q_n$ of $\bar\omega$ and the notation
$\beta=1/\alpha$, $\mathcal B$ and $\mathcal{BR}_\alpha$ from the introduction.

\begin{Lemma}\label{arh}
For the frequency $\bar\omega$ determined by \eqref{barome}, $\bar\omega\in\mathcal B\setminus\mathcal{BR}_\alpha$ and
\begin{equation}\label{quo}
 \sum_{n\geq m}\frac{\log q_{n+1}}{q_n}
 =(b+o(1))q_m^{\beta-1}\qquad(m\to\infty).
\end{equation}
In particular, the left-hand side is at most
$(b+\eta)q_m^{\beta-1}$ for every fixed $\eta>0$ and all sufficiently
large $m$.
\end{Lemma}
\begin{proof}
The recurrence and the prescribed partial quotients give
\begin{equation}\label{recc2}
 q_{n+1}=a_{n+1}q_n+q_{n-1},\qquad
 a_{n+1}=\lfloor e^{bq_n^\beta}\rfloor.
\end{equation}
Since $a_{n+1}=e^{bq_n^\beta}(1+O(e^{-bq_n^\beta}))$, we have
\[
 \log q_{n+1}=bq_n^\beta+\log q_n+O(e^{-bq_n^\beta}),
 \qquad
 \frac{\log q_{n+1}}{q_n}=(b+o(1))q_n^{\beta-1}.
\]
The error is uniform for $n\geq m$ as $m\to\infty$.
For sufficiently large $m$, set
\[
 \Lambda_m:=\tfrac12 e^{bq_m^\beta},\qquad \rho_m:=\Lambda_m^{\beta-1}.
\]
For all $n\geq m$, the recurrence gives
$q_{n+1}\geq \Lambda_mq_n$; hence $\rho_m\to0$ and
\[
 q_m^{\beta-1}
 \leq\sum_{n\geq m}q_n^{\beta-1}
 \leq\frac{q_m^{\beta-1}}{1-\rho_m}.
\]
This proves \eqref{quo} and the Brjuno condition.
On the other hand,
\[
 \frac{\log q_{n+1}}{q_n^\beta}
 =b+\frac{\log q_n}{q_n^\beta}+o(1)\longrightarrow b>0,
\]
so the $\alpha$-Brjuno--R\"ussmann series diverges.
\end{proof}

\begin{Lemma}\label{hc-arithmetic}
There is $C>1$, independent of $\ell,k,m$, such that
\begin{equation}\label{hc-divisor-bound}
 \|\ell\bar\omega\|\ge C^{-1}e^{-C|\ell|^\beta}
 \quad(\ell\in\Z\setminus\{0\}),\qquad
 \|k\omega_m\|\ge C^{-1}e^{-C(q_m|k|)^\beta}
 \quad(k\in\Z\setminus\{0\}).
\end{equation}
\end{Lemma}
\begin{proof}
For $q_n\le|\ell|<q_{n+1}$, the best-approximation property gives
\[
 \|\ell\bar\omega\|\ge\|q_n\bar\omega\|
 >\frac1{q_{n+1}+q_n}.
\]
By \eqref{recc2}, $\log(2q_{n+1})\le C q_n^\beta$
for every sufficiently large $n$. Enlarging $C$ handles the
finitely many remaining denominators and proves the first bound.
The identity $\|k\omega_m\|=\|kq_m\bar\omega\|$
proves the second.
\end{proof}

\begin{Remark}
The sometimes used weighted condition
\[
 \sum_{n\geq0}\log q_{n+1}
       (q_n^{-\beta}-q_{n+1}^{-\beta})<\infty
\]
is equivalent to the definition of $\mathcal{BR}_\alpha$ used here.
Indeed, their difference is the convergent series
\[
 \sum_{n\geq0}\frac{\log q_{n+1}}{q_{n+1}^\beta}.
\]
Convergence follows because
$q_{n+2}\geq2q_n$ and $(\log t)/t^\beta$ is eventually decreasing.
Thus the failure of the $\alpha$-Brjuno--R\"ussmann condition is
independent of these two conventions.
\end{Remark}
\begin{Remark}
For every $1\leq\alpha'<\alpha$, one has
$\bar\omega\in\mathcal{BR}_{\alpha'}$.
Indeed, with $\beta':=1/\alpha'>\beta$,
\[
 \frac{\log q_{n+1}}{q_n^{\beta'}}
 \leq Cq_n^{-(\beta'-\beta)}
\]
for all sufficiently large $n$, and the series on the right converges
because $q_{n+2}\geq2q_n$.
\end{Remark}

\section{\sc Results from \cite{BG94} and \cite{DLLGJV}}\label{hc-external-results}

We state the two results used in Section~\ref{5}, in its notation,
without proofs.

\subsection{Analytic interpolation}
The source is \cite[Section~2, Proposition~1, pp.~1122--1123,
(2.11), (2.14); the following Remark, p.~1123;
Corollary~1, p.~1124, (2.15)]{BG94}.
The Remark supplies a global scalar function when the mixed
generating function below is globally defined.

\begin{Proposition}\label{hc-external-interpolation}
Let $S_\mu=\operatorname{Id}+\sum_{n\ge1}\mu^nF_n$ be a real analytic
symplectic map on a fixed complex cylinder, with
$\|F_n\|\le\Gamma\gamma^{n-1}$, and suppose it has a mixed generating
function $P'X+\widetilde S_\mu(P',X)$ with periodic, globally defined
$\widetilde S_\mu$.
On fixed smaller cylinders there are real periodic analytic scalar functions $g_j$
with factorial bounds
\[
 \|g_j\|_{C^1}\le C_0C_1^j(j+1)!,
\]
such that the formal coordinate series of
$\mathcal B_{\mu,\sum_{j\ge0}\mu^jg_j}$ agrees coefficient by
coefficient with $S_\mu$.
For $N(\mu)=\lfloor c_0/\mu\rfloor$, with a sufficiently small
fixed $c_0>0$ depending only on the stated data, the finite sum
\[
 G_\mu=\sum_{j=0}^{N(\mu)}\mu^jg_j
\]
satisfies
\[
 \|S_\mu-\mathcal B_{\mu,G_\mu}\|\le C\mu e^{-c/\mu}.
\]
{We take $g_0(0,0)=\Inv_0(0,0)$ in the application above
and $g_j(0,0)=0$ for $j\ge1$. Subtracting $g_j(0,0)$ at most doubles
the displayed $C^1$ bound and leaves the interpolating map unchanged.}
All constants are independent of $0<\mu<\mu_0$, for fixed cylinder
widths, domain differences, $\Gamma$ and $\gamma$.
\end{Proposition}

The estimates for any fixed number of derivatives follow from
\eqref{hc-cauchy-derivatives} on further fixed smaller cylinders.

\subsection{The symplectic frame}
The source is \cite[Section~4.2.1, Proposition~4, p.~871,
(28)--(33)]{DLLGJV}. We use its one-angle geometric identity.

\begin{Proposition}\label{hc-external-frame}
Let $F$ preserve area and let $K$ satisfy $F\circ K=K\circ R_\nu$.
Put $v=K'$ and assume $v^{\mathsf T}v\ne0$. For
\[
 \mathsf M=\left(v,\frac{\mathsf Jv}{v^{\mathsf T}v}\right),\qquad
 \mathsf J(v_1,v_2)=(-v_2,v_1),
\]
there is a scalar function $S$ such that
\[
 \mathsf M(\theta+\nu)^{-1}DF(K(\theta))\mathsf M(\theta)
       =\begin{pmatrix}1&S(\theta)\\0&1\end{pmatrix}.
\]
\end{Proposition}

For a nonzero invariance error, the related calculation appears in
\cite[Section~4.2.2, pp.~873--874, (41)--(44)]{DLLGJV}.
 Section~\ref{5} proves the required small-divisor and Newton
estimates separately.

\end{document}